\documentclass[letter,11pt,english]{amsart}
\usepackage{amsmath}
\usepackage{amssymb}
\usepackage{amsthm}
\usepackage{mathrsfs}
\usepackage{wasysym}
\usepackage{bbm}		
\usepackage{mathtools}
\usepackage[shortlabels]{enumitem}
\usepackage{braket}		
\usepackage{mdwlist}		
\usepackage{xcolor}
\usepackage{tikz}
\usepackage{tikz-cd}
\usetikzlibrary{arrows}
\usepackage{graphicx}
\usepackage[left=1in, right=1in, top=1in, bottom=0.75in]{geometry}
\usetikzlibrary{patterns,positioning,arrows,decorations.markings,calc,decorations.pathmorphing,decorations.pathreplacing}
\usepackage{hyperref}
\usepackage{cleveref}
\usepackage{color}
\usepackage{float}
\usepackage{soul}   
\usepackage{caption}
\usepackage{subcaption}
\usepackage{circledsteps}
\usepackage{multirow}
\usepackage[toc,page]{appendix}
\usepackage{listings}
\definecolor{codegreen}{rgb}{0,0.6,0}
\definecolor{codegray}{rgb}{0.5,0.5,0.5}
\definecolor{codepurple}{rgb}{0.58,0,0.82}
\definecolor{backcolour}{rgb}{0.95,0.95,0.92}

\lstdefinestyle{mystyle}{
    backgroundcolor=\color{backcolour},   
    commentstyle=\color{codegreen},
    keywordstyle=\color{magenta},
    numberstyle=\tiny\color{codegray},
    stringstyle=\color{codepurple},
    basicstyle=\ttfamily\footnotesize,
    breakatwhitespace=false,         
    breaklines=true,                 
    captionpos=b,                    
    keepspaces=true,                 
    numbers=left,                    
    numbersep=5pt,                  
    showspaces=false,                
    showstringspaces=false,
    showtabs=false,                  
    tabsize=2
}

\usepackage[
backend=biber, 
style=alphabetic, 
sorting=ynt 
]{biblatex} 
\usepackage{todonotes}

\theoremstyle{plain}
\newtheorem{theorem}{Theorem}[section]
\newtheorem{lemma}[theorem]{Lemma}

\newtheorem{proposition}[theorem]{Proposition}

\newtheorem{definition}[theorem]{Definition}

\theoremstyle{definition}

\newcommand{\RR}{\mathbb{R}}			
\newcommand{\Z}{\mathbb{Z}}			
\newcommand{\R}{\mathbb{R}}			

\newcommand{\bn}{{\boldsymbol{n}}}

\newcommand{\Acal}{\mathcal{A}}

\newcommand{\Lcal}{\mathcal{L}}

\usepackage[utf8]{inputenc}

\title{Experimental Results on Potential Markov Partitions for Wang Shifts}
\author[H.~Hults]{Harper Hults}
\address[H.~Hults]{University of Washington Bothell, 18115 Campus Way NE, Bothell, WA 98011-8246}
\email{hhhults@uw.edu}

\author[H.~Jitsukawa]{Hikaru Jitsukawa}
\address[H.~Jitsukawa]{Tufts University, 419 Boston Ave, Medford, MA 02155}
\email{hikaru.jitsukawa@tufts.edu}

\author[C.~Mann]{Casey Mann}
\address[C.~Mann]{University of Washington Bothell, 18115 Campus Way NE, Bothell, WA 98011-8246}
\email{cemann@uw.edu}

\author[J.-Zhang]{Justin Zhang}
\address[J.~Zhang]{Georgia Institute of Technology, North Ave NW, Atlanta, GA 30332 - }
\email{jzhang3264@gatech.edu}

\makeatletter
\@namedef{subjclassname@2020}{\textup{2020} Mathematics Subject Classification}
\makeatother



\date{July 2022}

\begin{document}

\maketitle

\begin{abstract}
In this article we discuss potential Markov partitions for three different Wang tile protosets. The first partition is for the order-24 aperiodic Wang tile protoset that was recently shown in the Ph.D. thesis of H. Jang to encode all tilings by the Penrose rhombs. The second is a partition for an order-16 aperiodic Wang protoset that encodes all tilings by the Ammann A2 aperiodic protoset. The third partition is for an order-11 Wang tile protoset identified by Jeandel and Rao as a candidate order-11 aperiodic Wang tile protoset. The emphasis is on some experimental methodology to generate potential Markov partitions that encode tilings. We also apply some of the theory developed by Labb\'{e} in analyzing such an experimentally discovered partition. \end{abstract}

\section{Introduction}

\subsection{Tilings and Aperiodic Protosets}\label{sec:tiling}
A \emph{\textbf{tiling}} of the Euclidean plane $\mathbb{E}^2$ is a collection $\{T_i\}$ of distinct sets called \emph{\textbf{tiles}} (which are typically topological disks) such that $\cup T_i = \mathbb{E}^2$ and for all $i \neq j$, $\text{int}(T_i) \cap \text{int}(T_j) = \emptyset$. A \emph{\textbf{protoset}} for a tiling $\mathscr{T}$ is a collection of tiles $\mathcal{T}$ such that each tile of $\mathscr{T}$ is congruent to a tile in $\mathcal{T}$, in which case we say that $\mathcal{T}$ \emph{\textbf{admits}} the tiling $\mathscr{T}$. The \emph{\textbf{symmetry group}} of a tiling $\mathscr{T}$ is the set of rigid planar transformations $\mathcal{S}(\mathscr{T})$ such that $\sigma(\mathscr{T}) = \mathscr{T}$ for all $\sigma \in \mathcal{S}(\mathscr{T})$. If $\sigma(\mathscr{T})$ contains two nonparallel translations, we say that $\mathscr{T}$ is \textbf{\emph{periodic}}; otherwise, we say $\mathscr{T}$ is \emph{\textbf{nonperiodic}}. If $\mathcal{T}$ admits at least one tiling of the plane and every such tiling is nonperiodic, we say that $\mathcal{T}$ is an \emph{\textbf{aperiodic protoset}}.

Aperiodic protosets are somewhat rare, and indeed, until the 1960s, no aperiodic protosets were known to exist. H. Wang famously conjectured that aperiodic protosets cannot exist \cite{Wang}; that is, Wang conjectured that any protoset that admits a tiling of the plane must admit at least one periodic tiling. Wang's conjecture was refuted in 1964 by his doctoral student, R. Berger, in his PhD thesis \cite{MR2939561,MR216954} where an aperiodic protoset consisting of over 20,000 edge-marked squares called Wang tiles was described. Wang tiles are a bit of a special case in the theory of tilings in that only translations are allowed in forming tilings from Wang tile protosets; if rotations and reflections are also allowed, then any single Wang tile admits periodic tilings of the plane, and so without this restriction the problem is trivial. Since Berger's original discovery, several lower-order aperiodic Wang tile protosets have been discovered, culminating in 2015 with the discovery by E. Jeandel and M. Rao \cite{JR1} of an order-11 aperiodic Wang tile protoset. In this same work the authors proved that 11 is the minimum possible order for an aperiodic Wang tile protoset. The aperiodic Jeandel-Rao protoset is seen in Figures \ref{fig:JR_proto}, and here we are using the same labeling for the Jeandel-Rao protoset as given by Labb\'{e} in \cite{Labb2021}.

\begin{figure}[h]
\begin{center}
\includegraphics[width=\textwidth]{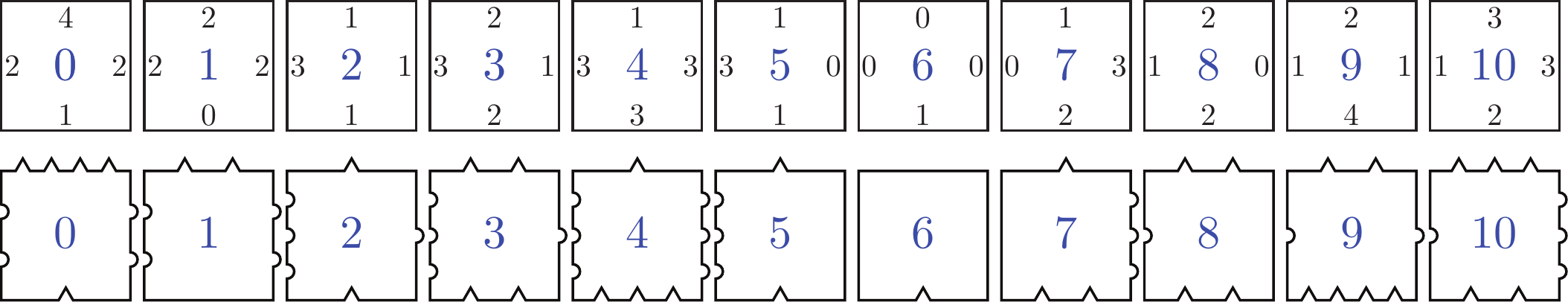}\caption{The Jeandel-Rao aperiodic Wang protoset $\mathcal{T}_0$. The bottom row provides geometric matching rules for the prototiles in the top row.}\label{fig:JR_proto}
\end{center}
\end{figure}

Aperiodic protosets of tiles other than Wang tiles are also known. In 1971, R. Robinson developed a way to encode the idea of Berger's original aperiodic protoset into an order-6 aperiodic protoset consisting of shapes with notched edges and modified corners (Figure \ref{fig:robinson}). In 1974, R. Penrose discovered an order-6 aperiodic protoset that he was later able to modify to produce an order-2 aperiodic protoset. There are a few varieties of the order-2 aperiodic Penrose protoset, one of which, called \emph{P3} or the \emph{Penrose rhombs}, is depicted in Figure \ref{fig:penrose_rhombs}; this protoset will be discussed further this article. More recently, in 2023 two singleton aperiodic protosets was discovered (\cite{hat, spectre}) - these ``aperiodic monotiles" are called the \emph{hat} (Figure \ref{fig:hat}) and \emph{spectre} (Figure \ref{fig:spectre}); we note that no edge-matching rules are needed to enforce nonperiodicity with the hat and spectre, unlike the Penrose rhombs, for example.

\begin{figure}[h]
\centering
\begin{subfigure}[b]{.4\textwidth} 
\centering
\includegraphics[width=.8\textwidth]{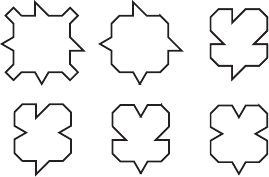}  
\caption{Robinson's order-6 aperiodic protoset}
\label{fig:robinson}
\end{subfigure}
\begin{subfigure}[b]{.4\textwidth} 
\centering
\includegraphics[width=.8\textwidth]{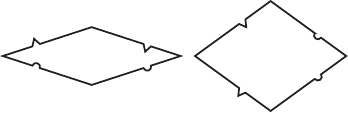}  
\caption{The Penrose P3 aperiodic protoset}
\label{fig:penrose_rhombs}
\end{subfigure}
\begin{subfigure}[b]{.4\textwidth} 
\centering
\includegraphics[width=.4\textwidth]{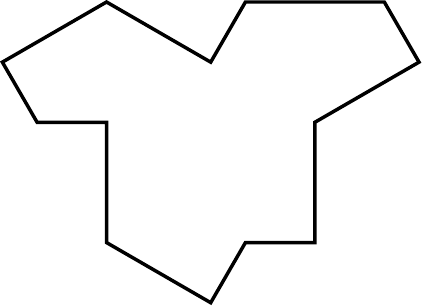}  
\caption{The hat monotile}
\label{fig:hat}
\end{subfigure}
\begin{subfigure}[b]{.4\textwidth} 
\centering
\includegraphics[width=.4\textwidth]{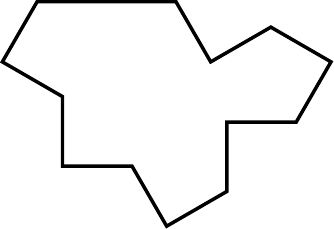}  
\caption{The spectre monotile}
\label{fig:spectre}
\end{subfigure}
\label{fig:aperiodic}\caption{Some low-order aperiodic protosets}\end{figure}


\subsection{Labb{\'{e}}'s Markov Partition for the Jeandel-Rao Protoset}
Let $\mathcal{T}_0 = \{T_0, T_1, T_2, \ldots, T_{10}\}$ be the aperiodic Jeandel-Rao Wang tile protoset (Figure \ref{fig:JR_proto}). The set of all tilings admitted by $\mathcal{T}_0$ is called the \emph{Jeandel-Rao Wang shift} and is denoted $\Omega_0$.  In \cite{Labb2021}, S. Labb\'{e} presented a special partition $\mathcal{P}_0 = \{P_0, P_1, \ldots, P_{10}\}$ (Figure \ref{fig:P0}) of the two-dimensional torus $\mathbf{T}$, along with a $\Z^2$ action on $\mathbf{T}$, that together give rise tilings in $\Omega_0$. Specifically, let $\varphi = (1 + \sqrt{5})/2$ be the golden mean and let $\Gamma_0$ be the lattice in $\RR^2$ generated by $(\varphi,0)$ and $(1,\varphi+3)$. Then the partitioned torus is $\mathbf{T} = \R^2 / \Gamma_0$ and the $\Z^2$ action $R_0$ on $\mathbf{T}$ is defined by $R_{0}^{\mathbf{n}}(\boldsymbol{x}) = \boldsymbol{x} + \mathbf{n} \pmod{\Gamma_0}$. 

Here we describe how the partition $\mathcal{P}_0$ of $\mathbf{T}$ and the action $R_0$ encode tilings by $\mathcal{T}_0$. Let $p \in \mathbf{T}$ and consider the orbit $\mathcal{O}_{\R_0}(p) = \{R_{0}^{\mathbf{n}}(p):\mathbf{n} \in \Z^2\}$. Then $\mathcal{O}_{ R_0}(p)$ corresponds to a tiling in $\Omega_0$ in the following way: \[T_i \in \mathcal{T}_0 \text{ is located at position }\textbf{n} \text{ iff } R_{0}^{\mathbf{n}}(p) \in P_i.\] This idea is depicted in Figure \ref{fig:P0}. There are several technical details that Labb\'{e} proved so that the correspondence between orbits of points in the partitioned torus correspond to tilings in an unambiguous way, but this is the gist of it. It is interesting - even somewhat amazing -  that most tilings admitted by $\mathcal{T}_0$ are encoded by $\mathcal{P}_0$ and $R_0$!

In this article, we present some partitions, discovered experimentally, that are analogous to the Markov partition $\mathcal{P}_0$ for the Jeandel-Rao protoset, but for other Wang tile protosets. Further, we discuss how to identify other potential Markov partitions for Wang shifts in which tilings exhibit ``golden rotational" behavior similar to $R_0$ with $\mathcal{P}_0$. We will also discuss theoretical framework described by Labb\'{e} in \cite{Labb2021} and apply it to one of these  new partitions in an effort to describe the full space of tilings arising from the partition.

\begin{figure}[h]
\begin{center}
\includegraphics[width=\textwidth]{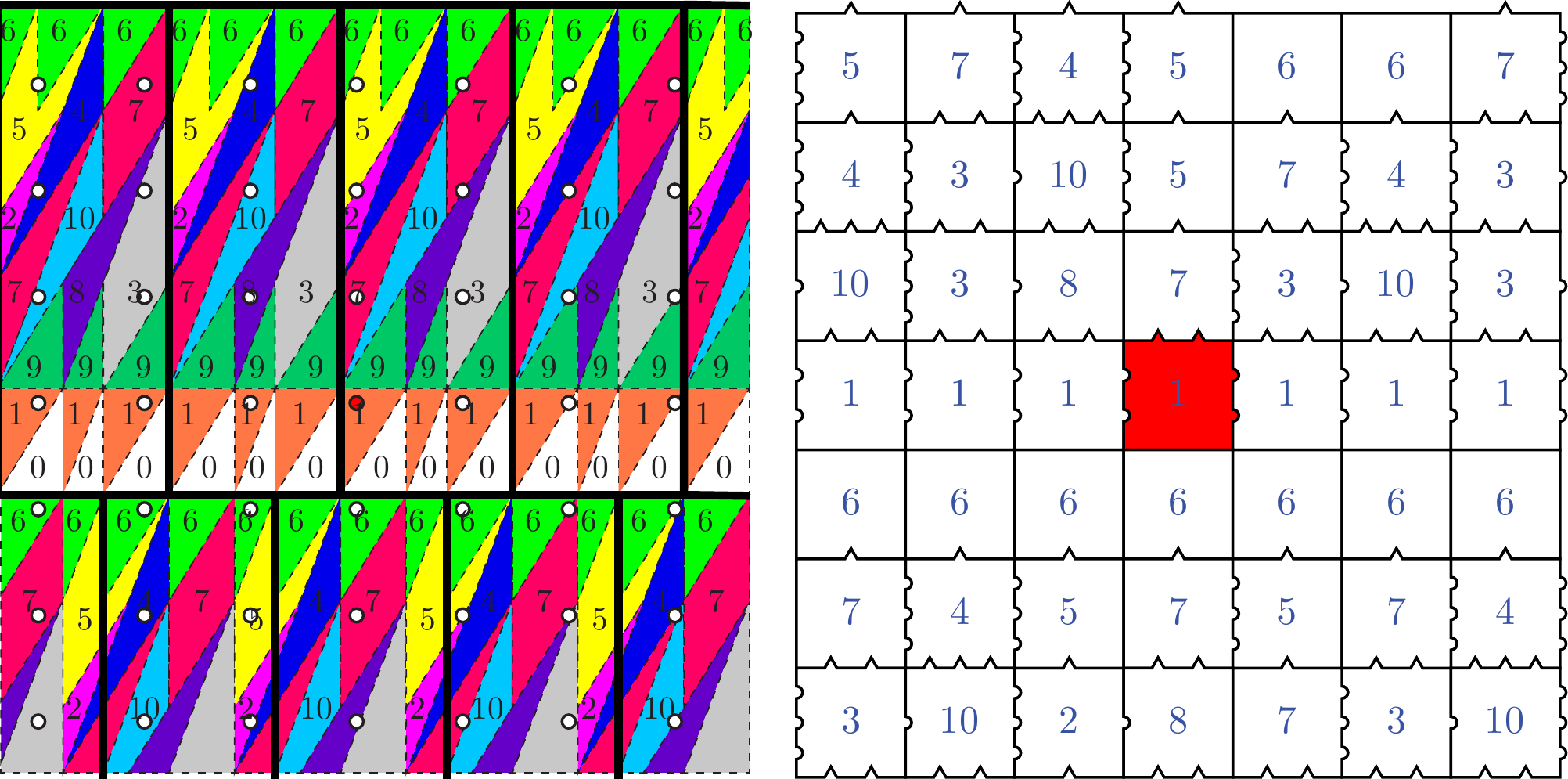}\caption{Labb\'{e}'s partition $\mathcal{P}_0$ of the torus $\mathbf{T}$ is outlined by the heavy black line. The orbit of the red point $\mathbf{p}$ under the $\Z^2$-action $R_0$ definted by $R_{0}^{\mathbf{n}}(\mathbf{p}) = \mathbf{p} + \mathbf{n}$ (the set of white points) corresponds to a Wang tiling in $\Omega_0$, depicted at right.} \label{fig:P0}
\end{center}
\end{figure}

\section{Labb\'{e}'s Construction}
While we are focussed on an experimental approach in this article, it will be necessary to describe some of the technical details of Labb\'{e}'s construction, which are described in the language of dynamical systems. For now, we want to keep this discussion as non-technical as possible to focus on the big picture ideas, and we will provide more technical apsects in Appendix \ref{apx:dyn_sys}. In that spirit, let us start by describing Wang tilings in the language of dynamical systems. This begins by defining Wang tilings in a very symbolic way.

\subsection{Spaces of Wang Tilings as Dynamical Systems}\label{subsec:dyn_sys_intro}
Geometrically, a \emph{\textbf{Wang tile}} is a square with labeled (or colored) sides that serve as edge matching rules; two Wang tiles may meet along an edge if corresponding labels match. In forming a \textbf{\emph{Wang tiling}} from copies of tiles in a protoset of Wang tiles, we allow only translates of the prototiles to be used, and we require that the tiles to meet edge-to-edge, respecting the matching rules. The edge-to-edge  matching rules and restriction to translations can always be enforced geometrically, as in Figure \ref{fig:JR_proto}. Wang tilings can be aligned with the integer lattice $\Z^2$ so that each tile has its lower left-hand corner in $\Z^2$, and thus we may view a Wang tiling as a decoration of $\Z^2$ such that each point of $\Z^2$ is ``colored" with a single tile. 

More formally, a Wang tile can be represented as a tuple of four colors $(r, t, \ell, b) \in V \times H \times V \times H$ where $V$ and $H$ are finite sets (the vertical and horizontal colors, respectively). For any Wang tile $T = (r, t, \ell, b)$, let $\text{Right}(T) = r$, $\text{Top}(T) = t$, $\text{Left}(T) = \ell$, and $\text{Bottom}(T) = b$ be the colors of the right, top, left, and right sides of $\tau$, respectively. Let $\mathcal{T} = \{T_0, T_1, \ldots, T_{n-1}\}$ be a protoset of Wang tiles. In forming a tiling with the protoset $\mathcal{T}$, we will think of the prototiles of $\mathcal{T}$ as symbols used to label the integer lattice $\Z^2$. This motivates us to define \[\mathcal{T}^{\Z^2} = \{x \!:\! \Z^2 \rightarrow \mathcal{T}\}\] to be the \textbf{\emph{configuration space}} of $\mathcal{T}$, and each element $x \in \mathcal{T}^{\Z^2}$ is called a \textbf{\emph{configuration}}. We will use the notation $x_{\mathbf{n}} = x(\mathbf{n})$ to indicate the value of $x$ at $\mathbf{n} \in \Z^2$ so that if $x_{\mathbf{n}} = T_i$ we understand that tile $T_i$ is located at position $\mathbf{n}$ (i.e,, oriented so that its corners lie in $\Z^2$ and the lower left-hand corner is at $\mathbf{n}$). It is possible that a configuration in $\mathcal{T}^{\Z^2}$ does not represent a valid tiling; so far, this construction does not account for the matching rules of the tiles. To account for the matching rules, we define a configuration $x$ to be \textbf{\emph{valid}} if for every $\mathbf{n} \in \Z^2$, we have $\text{Right}(x_{\mathbf{n}}) = 
\text{Left}(x_{\textbf{n} + (1,0)})$ and $\text{Top}(x_{\mathbf{n}}) = 
\text{Bottom}(x_{\textbf{n} + (0,1)})$. The \emph{\textbf{Wang shift}} of $\mathcal{T}$ is the set $\Omega_{\mathcal{T}}$ of all valid configurations in $\mathcal{T}^{\Z^2}$. Thus, $\Omega_{\mathcal{T}}$ captures in a symbolic way the set of all valid Wang tilings admitted by the Wang tile protoset $\mathcal{T}$.

A \textbf{\emph{topological dynamical system}} consists of a topological space $X$ and a group $G$ that acts on $X$ via a continuous function $S\!:\!:G \times X \rightarrow X$, and is denoted by a triple $(X,G,S)$. For fixed $g \in G$, let $S^g\!:\! X \rightarrow X$ be defined by $S^g(x) = S(g,x)$. The \textbf{\emph{orbit}} of a point $x \in X$ under the action of $S$ on $G$ is the set $\mathcal{O}_S(x,G) = \{S^g(x):g \in G\}$, and when the group $G$ is understood by context, we will shorten this notation to just $\mathcal{O}_{S}(x)$. We say $S$ \textbf{\emph{acts freely}} on $X$ if for all $x \in X$, $S^g(x) = x$ if and only if $g = e$ where $e \in G$ is the group identity element. 

Let $\sigma:\Z^2 \times \mathcal{T}^{\Z^2} \rightarrow \mathcal{T}^{\Z^2}$ be defined by $\sigma^{\mathbf{n}}(x_{\mathbf{m}}) = x_{\mathbf{m} + \mathbf{n}}$. We call $\sigma$ the $\Z^2$ \textbf{\emph{shift action}} on $\mathcal{T}^{\Z^2}$; the terminology ``shift action" is accurately descriptive here since, for each $\mathbf{n} \in \Z^2$ and $x \in \mathcal{T}^{\Z^2}$, $\sigma^{\mathbf{n}}(x)$ is a translation of $x$ by $\mathbf{n}$. By placing the appropriate metric on $\mathcal{T}^{\Z^2}$, it is not too hard to see that $\Omega_{\mathcal{T}} \subseteq \mathcal{T}^{\Z^2}$ is a topological space so that $(\Omega_{\mathcal{T}},\Z^2,\sigma)$ satisfies the definition of being a topological dynamical system. Moreover, $(\Omega_{\mathcal{T}},\Z^2,\sigma)$ is a special kind of topological dynamical system called a \emph{shift of finite type} (see Appendix \ref{apx:dyn_sys}).

If $\sigma$ acts freely on $\Omega_{\mathcal{T}}$, we say that $\Omega_{\mathcal{T}}$ is an \emph{\textbf{aperiodic shift}}, and we say that a Wang protoset $\mathcal{T}$ is an \emph{\textbf{aperiodic protoset}} if the corresponding Wang shift $\Omega_{\mathcal{T}}$ is nonempty and aperiodic. Notice that this definition for aperiodic protoset agrees with with more general definition given in Section \ref{sec:tiling} if we restrict the allowable symmetries to translations only.

\subsection{Labb\'{e}'s Partition-Based Dynamical System of Tilings}

Here we will provide the broad strokes of Labb\'{e}'s \cite{Labb2021} general approach and construction of the dynamical system $(\mathcal{X}_{\mathcal{P},R},\Z^2,\sigma)$ that is based on a partition $\mathcal{P}$ of a compact topological space $\mathbf{T}$ and a $\Z^2$ action $R$ on $\mathbf{T}$. The space $\mathcal{X}_{\mathcal{P},R}$ is the set of tilings encoded by the partition $\mathcal{P}$ and the action $R$, as exemplified by Figure \ref{fig:P0}. We will relegate much of the more technical details to Appendix \ref{apx:dyn_sys}. The main point is that methodologies are described in \cite{Labb2021} that can be used to prove when $\mathcal{P}$ and $R$ produce a space $\mathcal{X}_{\mathcal{P},R}$ that is reasonable subspace (or subshift) of the space of all valid tilings $\Omega_{\mathcal{T}}$ admitted by $\mathcal{T}$.

First, let $\mathbf{T}$ be a compact metric space and let $I_n = \{0,1,\ldots,n-1\}$. A collection $\mathcal{P} = \{P_i\}_{i\in I_n}$ of disjoint open sets with $\mathbf{T} = \cup_{i \in I_n}\overline{P}_i$ is called a \textbf{\emph{topological partition}} of $\mathbf{T}$, and the sets $P_i \in \mathcal{P}$ are called the \textbf{\emph{atoms}} of the partition. Let $\mathcal{P} = \{P_i\}_{i\in I_n}$ is a topological partition of $\mathbf{T}$, and let $\mathcal{T} = \{T_i\}_{i \in I_n}$ be a protoset of Wang tiles having the same indexing set $I_n$ as $\mathcal{P}$. Let $R$ be a $\Z^2$ action on $\mathbf{T}$. Then $(\mathbf{T},\Z^2,\R)$ is a topological dynamical system, and with $\mathbf{T}$ so partitioned by $\mathcal{P}$, the hope is that the orbit of a point $p \in \mathbf{T}$ will ``spell out" a tiling in $\Omega_{\mathcal{T}}$. More precisely, for each $p \in \mathbf{T}$, we can consider the orbit $\mathcal{O}_{R}(p) = \{R^{\mathbf{n}}(p):\mathbf{n} \in \Z^2\}$, and for each $\mathbf{n} \in \Z^2$, define the configuration $x\!:\!\Z^2 \rightarrow \mathcal{T}$ in $\mathcal{T}^{\Z^2}$ by \[x_{\mathbf{n}} = T_i \text{ if and only if } R^{\mathbf{n}}(p) \in P_i.\] The configuration $x$ may not be not well defined because for some $\mathbf{n} \in \Z^2$, the point $R^{\mathbf{n}}(p)$ may lie on the boundary of more than one atom of $P_i$. However, Labb\'{e} \cite{Labb2021} was able prove that this ambiguity can be removed in the following way: By specifying a direction $\mathbf{v}$ that is not parallel to any of the sides of the atoms of the partition, then if a point $a \in \mathcal{O}_{R}(p)$ falls on the boundary of an atom, the choice of tile associated with that point is the one in the direction of $\mathbf{v}$ from $a$. In this way, each point $p \in \mathbf{T}$ can be associated uniquely with a configuration $x \in \mathcal{T}^{\Z^2}$, thereby defining a mapping $\text{SR}^{\mathbf{v}} \!:\! \mathbf{T} \rightarrow \mathcal{T}^{\Z^2}$ (``$SR$" for ``symbolic representation"). With this map properly defined, we obtain the space of interest, $\mathcal{X}_{\mathcal{P},R}$, as the topological closure of the image of $SR^{\mathbf{v}}$: \[\mathcal{X}_{\mathcal{P},R} = \overline{\text{SR}^{\mathbf{v}}(\mathbf{T})}.\] Labb\'{e} \cite{Labb2021} shows that the choice of $\mathbf{v}$ doesn't actually matter after taking the topological closure, so the space $\mathcal{X}_{\mathcal{P},R}$ is not dependent on $\mathbf{v}$. Figure \ref{fig:P0} is worth a thousand words in understanding the essential function of $SR$. In that figure, we see how the points of the orbit of a point $p$ in the torus forms a \emph{configuration} of Wang tiles. 

We pause here to point out that the space $\mathcal{X}_{\mathcal{P},R}$ is fully defined in more technical terms in \cite{Labb2021} that facilitate the use of the tools of dynamical systems to prove that if the partition $\mathcal{P}$ and the action $R$ satisfy reasonable technical properties, the partitioned system $(\mathbf{T},\Z^2,R)$ gives rise to $(\mathcal{X}_{\mathcal{P},R}, \Z^2,\sigma)$ along with morphims going both directions. With such formalisms in place, important properties of the space $(\mathcal{X}_{\mathcal{P},R}, \Z^2,\sigma)$ can be determined. For example, we are interested in knowing when $(\mathcal{X}_{\mathcal{P},R}, \Z^2,\sigma)$ is shift of finite type, which is an important property that a Wang shift has. Another important proptery is minimality of the subshift, which can be deduced from properties of the partition and action.

At this point, it may be believable that $\mathcal{X}_{\mathcal{P},R}$ is a reasonable subset of $\mathcal{T}^{\Z^2}$, so configurations of $\mathcal{X}_{\mathcal{P},R}$ can be seen as potential valid configuations (i.e. tilings), but it should not be clear at all how the partition ``knows" what the edge matching rules are and how orbits of points correspond to \emph{valid tilings} that respect the matching rules of the prototiles of $\mathcal{T}$. This is where some magic occurs in the Labb\'{e} construction. The magic is that the partition $\mathcal{P}$ needs to be a refinement of four other partitions $\mathcal{P}_r$, $\mathcal{P}_t$, $\mathcal{P}_{\ell}$, and $\mathcal{P}_b$ called \emph{side partitions} that encode the matching rules of prototiles in $\mathcal{T}$. With these side partitions suitably defined, the main partition $\mathcal{P}$ is defined as $\mathcal{P} = \mathcal{P}_r \cap \mathcal{P}_t \cap \mathcal{P}_{\ell} \cap \mathcal{P}_b$. With $\mathcal{P}$ so defined, and checking that $\mathcal{P}$ has certain necessary dynamical properties, it is argued in \cite{Labb2021} that  $\mathcal{X}_{\mathcal{P},R} \subseteq \Omega_{\mathcal{T}}$. That is, the configurations generated from orbits of points in the partitioned set $\mathbf{T}$ are valid tilings! We use the word ``magic" here to describe the process of finding $\mathcal{P}$  because there is no simple formula for finding such a partition. For a given Wang protoset $\mathcal{T}$, the partition may be found through some educated guesswork and the help of computers. In the next section, we will describe an experimental approach to finding partitions.

\section{Finding Partitions Experimentally}

Let us return to the Jeandel-Rao aperiodic protoset and Labb\'{e}'s corresponding Markov partition from Figure \ref{fig:P0}. Recall that in this construction we have $\mathbf{T}=\R^2 / \Gamma_0$ where $\Gamma_0$ is generated by $(\varphi,0)$ and $(1,\varphi+3)$. Given a tiling $x \in \Omega_{\mathcal{T}_0}$, we will attempt to recover the partition $\mathcal{P}_0$. Suppose that $x \in \Omega_0$ is ``generated" from the orbit of a point $p \in \mathbf{T}$, using the partition $\mathcal{P}_0$ and the action $R_0$. Without loss of generality, we may (by shifting the orientation of the partition $\mathcal{P}_0$) suppose that $p = \mathbf{0} = (0,0)$. For each $\mathbf{n} \in \Z^2$, let $t_{\mathbf{n}} = R_{0}^{\mathbf{n}}(\mathbf{0})$. By definition of $R_0$, we have \begin{equation} \label{eqn:mod} t_{\mathbf{n}} = \mathbf{n} \!\!\! \pmod{\Gamma_0}.\end{equation} Recall the manner in which the partition $\mathcal{P}_0$ encodes the tiling: If $t_\mathbf{n} \in P_i \in \mathcal{P}_0$, then a copy of prototile $T_i \in \mathcal{T}_0$ is located at position $\mathbf{n}$. 

Now, working backward, suppose we have a tiling $x$ generated by the orbit $\mathcal{O}_{R_0}(\mathbf{0})$, but after generating $x$, let us pretend that we have lost all knowledge of the partition. If we colored the points $t_{\mathbf{n}}$ according to the label of the tile located at position $\mathbf{n}$, then the colored set of points \begin{equation}\label{eqn:orig_dots} \mathcal{P}_{\text{dots}} = \{t_{\mathbf{n}} \in \mathbf{T}: n \in \Z^2, t_{\mathbf{n}} \text{ has color }i \text{ if } T_i \text{ is located at position }\mathbf{n} \}\end{equation} would appear as a dense set of points in the rectangle $[\varphi,0]\times [1,\varphi+3]$, and the coloring of the points $t_{\mathbf{n}}$ would reveal the partition $\mathcal{P}_0$. 

Our goal here is to demonstrate that a partition associated with a Wang shift can be revealed by looking at specific tilings in the Wang shift without prior knowledge of the partition. By Equation \ref{eqn:mod}, for each $\mathbf{n} \in \Z^2$, there exists some $\mathbf{a}_{\mathbf{n}} \in \Z^2$ such that \begin{equation} \label{eqn:matrix} t_{\mathbf{n}} = \mathbf{n} + A\mathbf{a}_{\mathbf{n}}\end{equation} where \[ A = \left(\begin{matrix} \varphi & 1\\
0 & \varphi + 3\end{matrix}\right)\] is the matrix whose columns are the generating vectors of $\Gamma_0$. From Equation \ref{eqn:matrix}, we obtain $A^{-1}\mathbf{n} = A^{-1}t_{\mathbf{n}} - \mathbf{a}_{\mathbf{n}}$. Because $\mathbf{a}_{\mathbf{n}} \in \Z^2$, it follows that \begin{equation} \label{eqn:pull_back} A^{-1}\mathbf{n} \!\!\! \pmod{\Z^2} = A^{-1}t_{\mathbf{n}} \!\!\! \pmod{\Z^2}.\end{equation} Let \[\mathcal{P}_{\text{dots}}'=\left\{ A^{-1}\mathbf{n} \!\!\! \pmod{\Z^2}:\mathbf{n} \in \Z^2\right\},\] and assign to each $p_\mathbf{n} = A^{-1}\mathbf{n} \in \mathcal{P}_{\text{dots}}'$ the color $i$ if a copy of prototile $T_i$ is located at position $\mathbf{n}$. From Equation \ref{eqn:pull_back}, we see that the set $\mathcal{P}_{\text{dots}}'$ is just a linearly transformed copy of the original dense set of colored points $\mathcal{P}_{\text{dots}}$, and recall that $\mathcal{P}_{\text{dots}}$ reveals the partition $\mathcal{P}_0$. \emph{Therefore, we see that $\mathcal{P}'_{\text{dots}}$,  being a linear transformation of $\mathcal{P}_{\text{dots}}$, reveals a partition $\mathcal{P}'$ that is essentially equivalent to the original $\mathcal{P}_0$.}

If the only knowledge we have at the beginning of this process is that of a tiling or a large finite portion of a tiling (perhaps generated by a computer algorithm), then we may not be immediately privy to the matrix $A$ (i.e., the generating vectors of the lattice that forms the torus) whose inverse we would use to pull back the tiling to obtain a partition, but the tiling itself may give clues about this matrix. In fact, Labb\'{e} discovered the partition in exactly this way - by noticing some Sturmian-like patterns in the rows of a large patch of a Jeandel-Rao tiling that suggested the correct vectors by which to ``mod-out." \emph{Finding the matrix $A$ is the key to this strategy.}

To see how this experimental approach reveals the partition (once you have guessed the correct generating vectors for the lattice), let us first use Labb\'{e}'s Sagemath package (\cite{slabbe_sage}) to produce a tiling of a given size. Here we produce a 100x100 patch of a Jeandel-Rao Wang tiling:

\begin{lstlisting}
    from slabbe import WangTileSet

    tiles = [(2,4,2,1), (2,2,2,0), (1,1,3,1), (1,2,3,2), (3,1,3,3),
    ....: (0,1,3,1), (0,0,0,1), (3,1,0,2), (0,2,1,2), (1,2,1,4), (3,3,1,2)]

    T0 = WangTileSet(tiles)

    # This produces a 100x100 tiling from the Jeandel-Rao protoset in which the bottom row consists only of copies of prototile 4
    x = 100
    y = 100
    preassigned_tiles = {(a,0):4 for a in range(x)};
    S = T0.solver(x,y,preassigned_tiles=preassigned_tiles)
    %time tiling = S.solve(solver='glucose')
\end{lstlisting}

Next, from the tiling so produced, we can pull back the points of the tiling to the unit square and color each using the label of the tile from which it came. 

\begin{lstlisting}
    # This produces the pull back of the tiling into torus (quotient of the unit square) in which the color of the pulled-back point is the label of the tile from which it came.
    p = (1+sqrt(5))/2;
    M = matrix.column([[p,0],[1,p+3]]);
    P = tiling.plot_points_on_torus(M.inverse(),pointsize=5)
    show(P,aspect_ratio = 1)
\end{lstlisting}

Here is the picture so produced:

\begin{figure}[h]
\begin{center}
\includegraphics[width=0.75\textwidth]{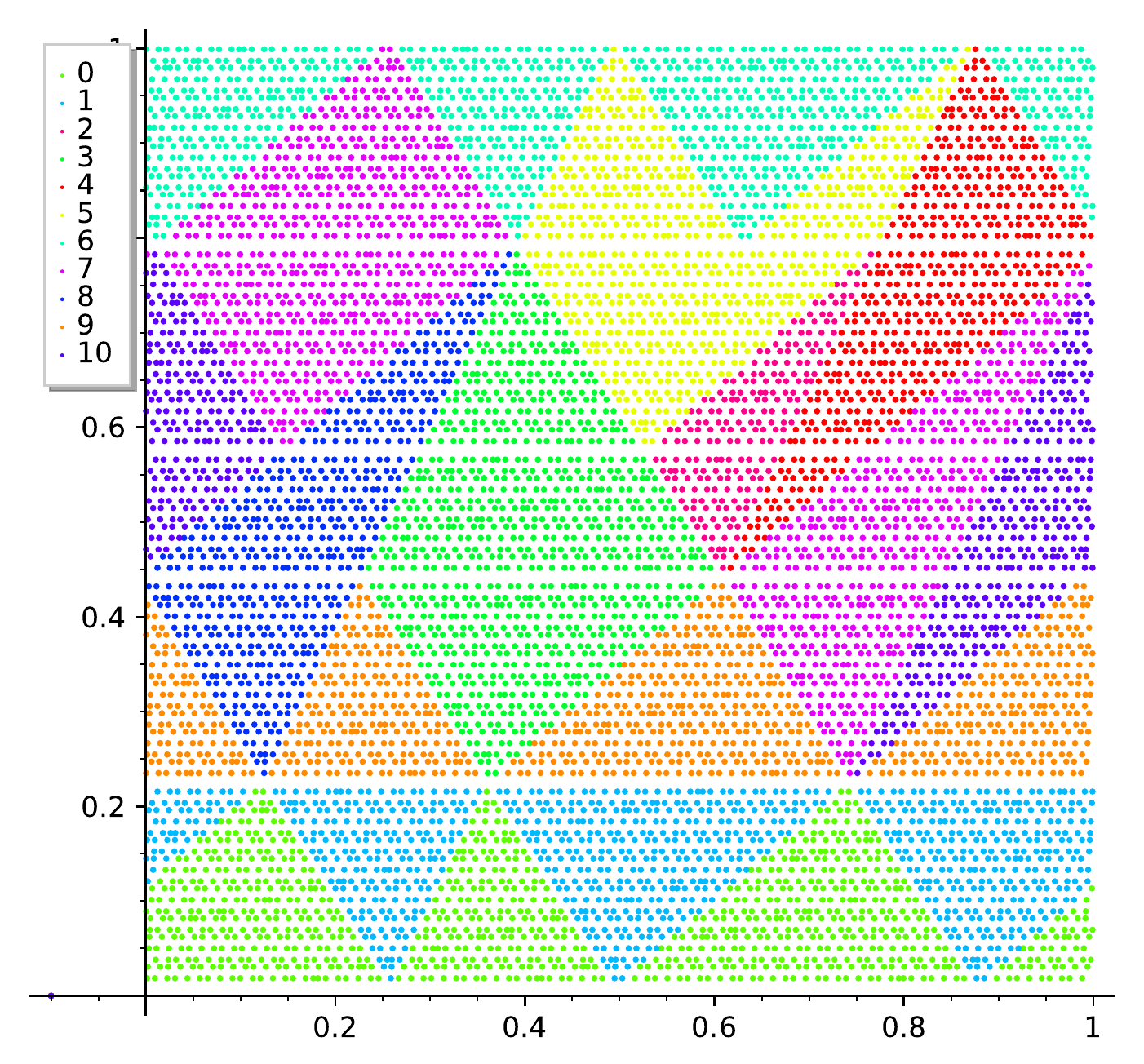}\caption{The result of applying $A^{-1}$ to each tile location in a 100x100 sized patch of a Jeandel-Rao Wang tiling and coloring each point according to the label of the corresponding tile.} \label{fig:JR-Markov-Dots}
\end{center}
\end{figure}

Comparing Figures \ref{fig:P0} and \ref{fig:JR-Markov-Dots}, we see that the dot pattern does indeed capture the atoms of the original partition. It should be noted here that the tiling used to generate the colored dot pattern was not produced using the partition $\mathcal{P}_0$; instead, it was produced using the Glucose solver, which implements a brute force method to finding a tiling. This suggests that the partition is fundamental to the space $\Omega_{0}$, which is the main thrust of what Labb\'{e} proved in \cite{Labb2021}.

\section{The Penrose and Ammann Wang Shifts}\label{sec:Penrose_Ammann}
In \cite[Section 11.1]{GS1987}, the authors discuss methods for converting low-order aperiodic protosets into low-order aperoidic Wang tile protosets, including the order-2 Penrose aperiodic protosets P2 and P3 and the order-2 Ammann aperiodic protoset A2. In particular, there we find an outline of how to convert Penrose tile protosets to two different aperiodic Wang tile protosets - one of order 24 (from P3) and one of order 16 (from P2). We also find that the Ammann A2 protosets can be converted to an order 16 Wang tile protoset, which, interestingly, is the exact same order-16 Wang tile protoset obtained from the Penrose P3 protoset, suggesting that P2 and A2 are really the same on some fundamental level. More recently, in H. Jang's recent Ph.D. thesis \cite[Lemma 6.2.2, p. 75]{Jang2021}, it was shown that any tiling by Penrose rhombs P3 is equivalent to a Wang tiling from the set of 24 Wang tiles mentioned above.

\subsection{The Order-24 Penrose Wang Tile Protoset}\label{subsec:Pen_24}
Following the method suggested in \cite[Section 11.1]{GS1987}, a tiling by the Penrose rhombs (P3) can be coverted to a tiling by Wang tiles using the patches of tiles shown in Figure \ref{fig:Pen_patch_proto}. This protoset of 24 Wang tiles, as was proven in \cite{Jang2021}, produces tilings in exact correspondence with the tilings produced by the P3 protoset.

\begin{figure}[H]
\centering
\begin{subfigure}[H]{.32\textwidth} 
\centering
\includegraphics[width=.87\textwidth]{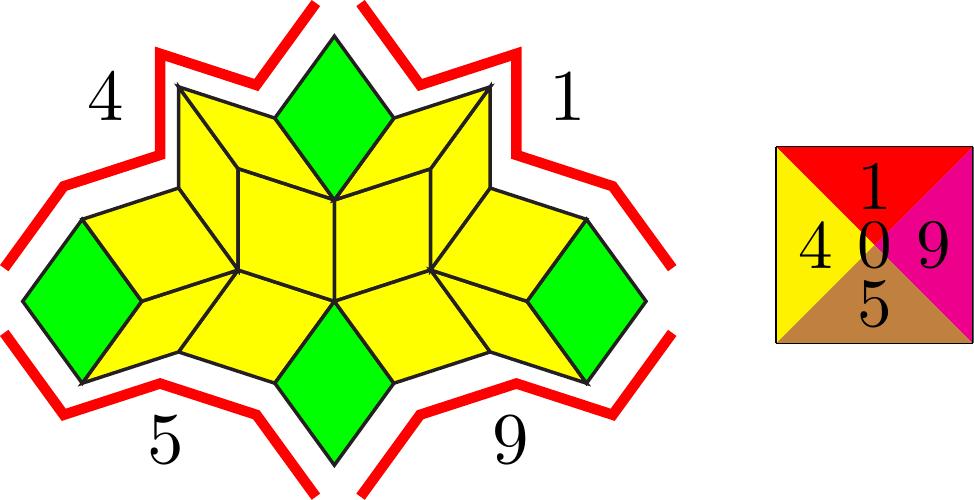}  
\end{subfigure}
\begin{subfigure}[H]{.32\textwidth} 
\centering
\includegraphics[width=.87\textwidth]{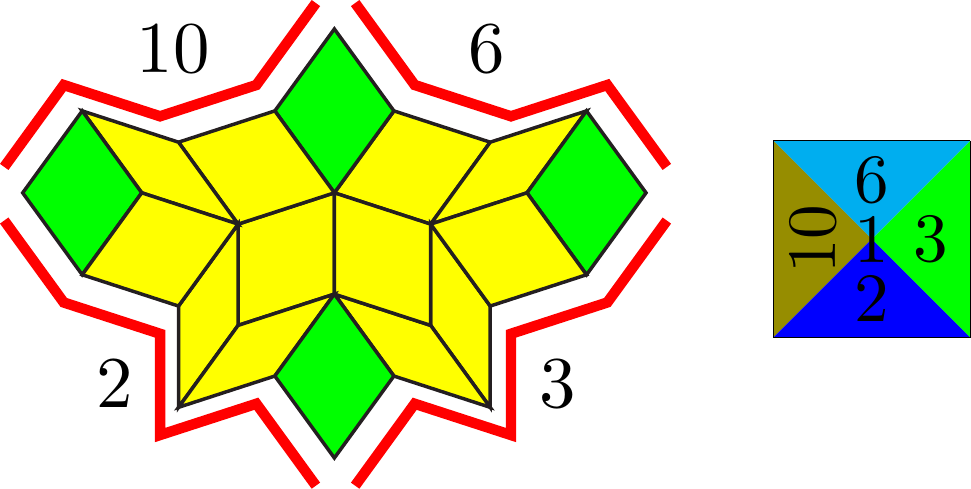}  
\end{subfigure}
\begin{subfigure}[H]{.32\textwidth} 
\centering
\includegraphics[width=.87\textwidth]{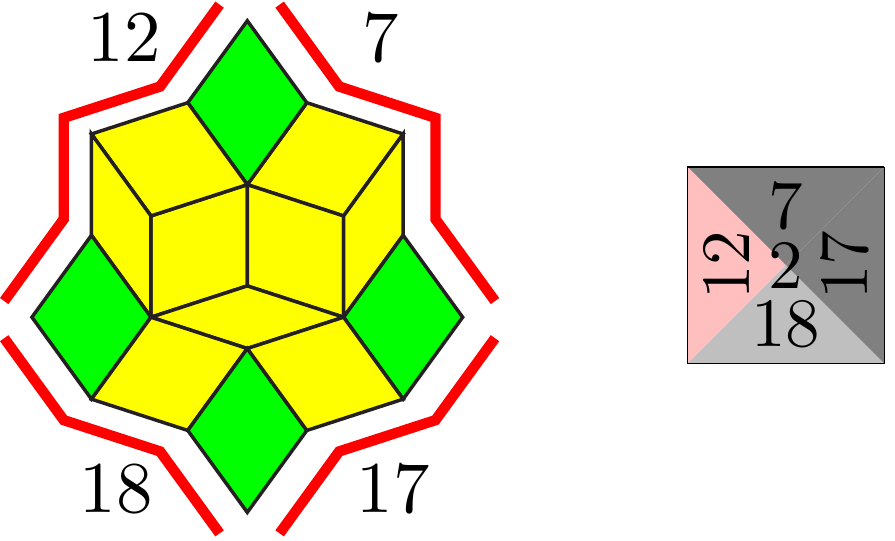}  
\end{subfigure}
\begin{subfigure}[H]{.32\textwidth} 
\centering
\includegraphics[width=.87\textwidth]{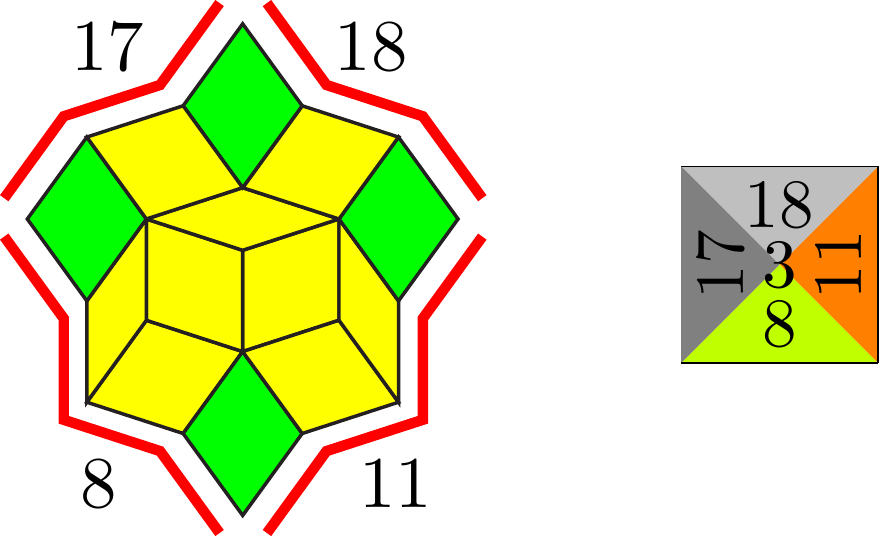}  
\end{subfigure}
\begin{subfigure}[H]{.32\textwidth} 
\centering
\includegraphics[width=.87\textwidth]{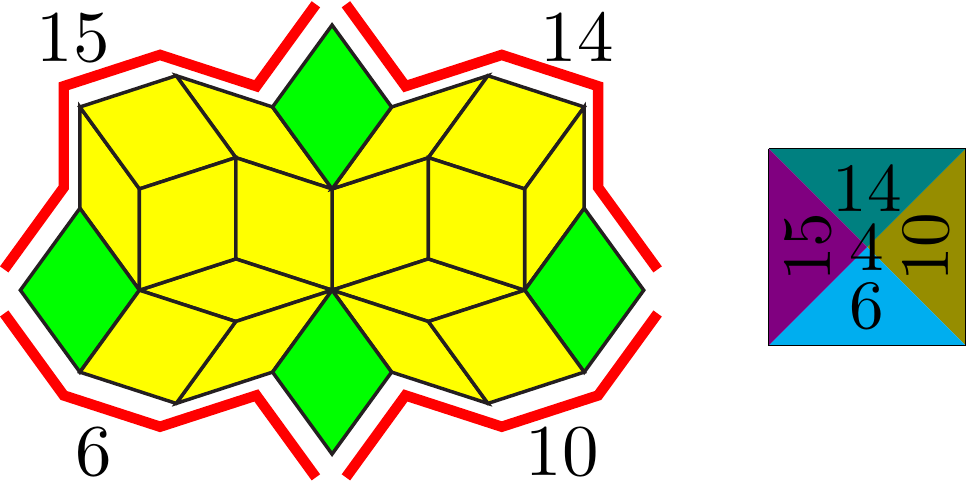}  
\end{subfigure}
\begin{subfigure}[H]{.32\textwidth} 
\centering
\includegraphics[width=.87\textwidth]{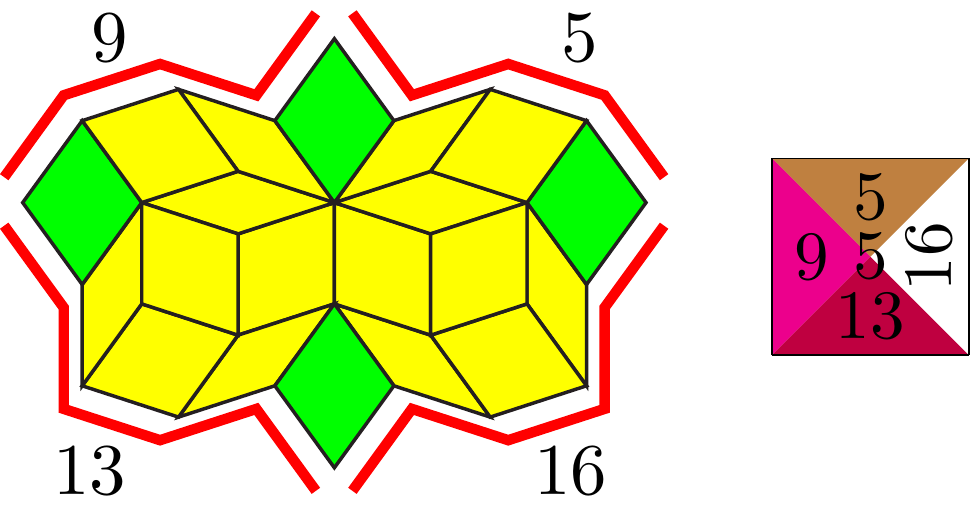}  
\end{subfigure}
\begin{subfigure}[H]{.32\textwidth} 
\centering
\includegraphics[width=.87\textwidth]{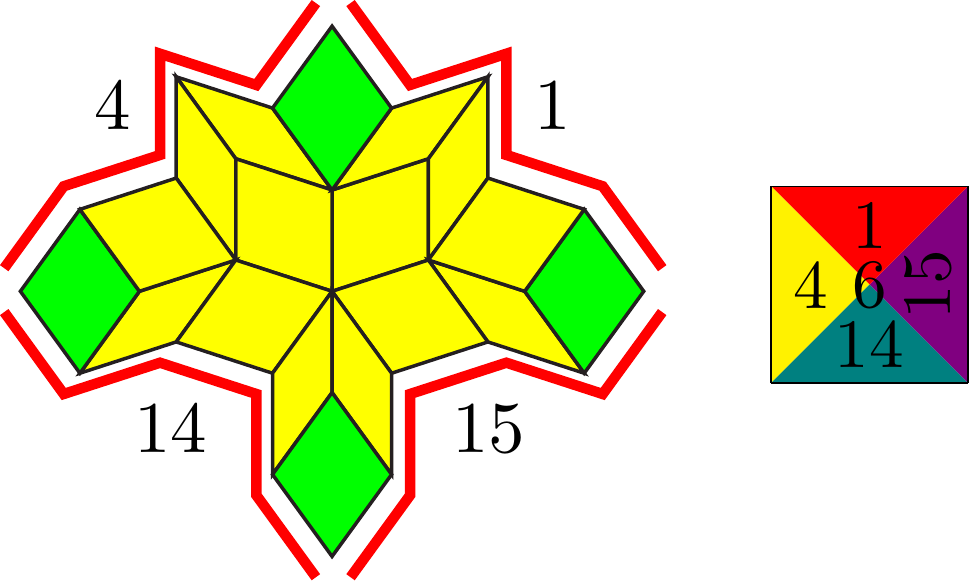}  
\end{subfigure}
\begin{subfigure}[H]{.32\textwidth} 
\centering
\includegraphics[width=.87\textwidth]{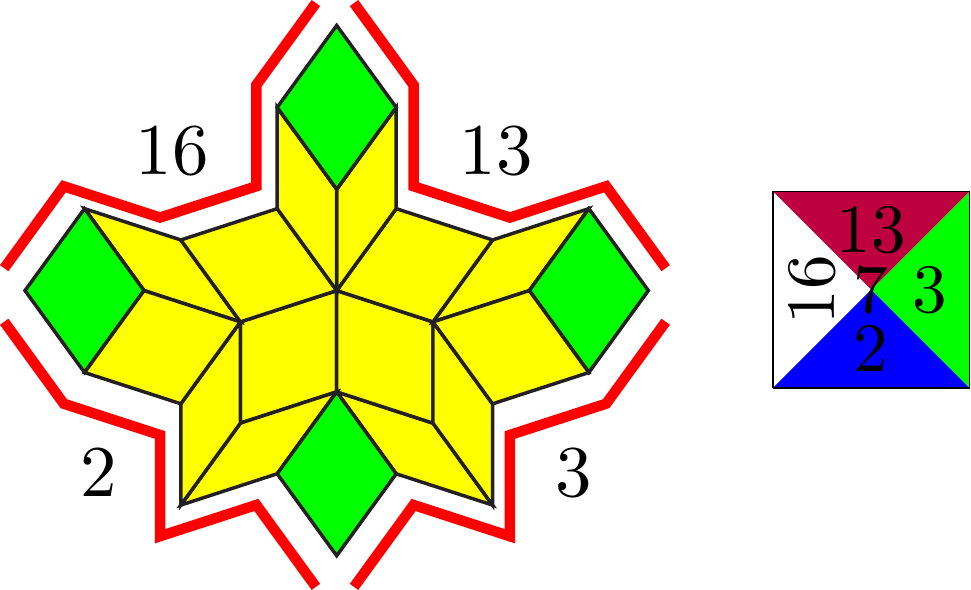}  
\end{subfigure}
\begin{subfigure}[H]{.32\textwidth} 
\centering
\includegraphics[width=.87\textwidth]{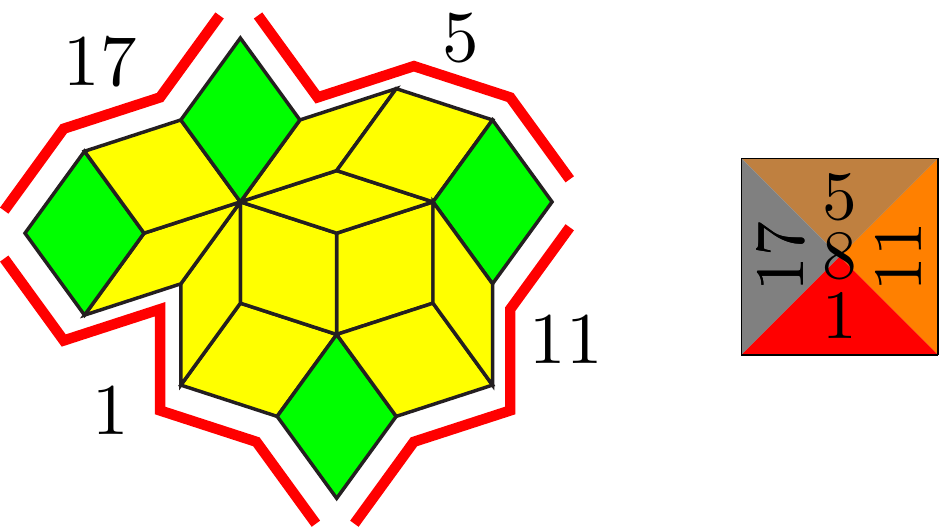}  
\end{subfigure}
\begin{subfigure}[H]{.32\textwidth} 
\centering
\includegraphics[width=.87\textwidth]{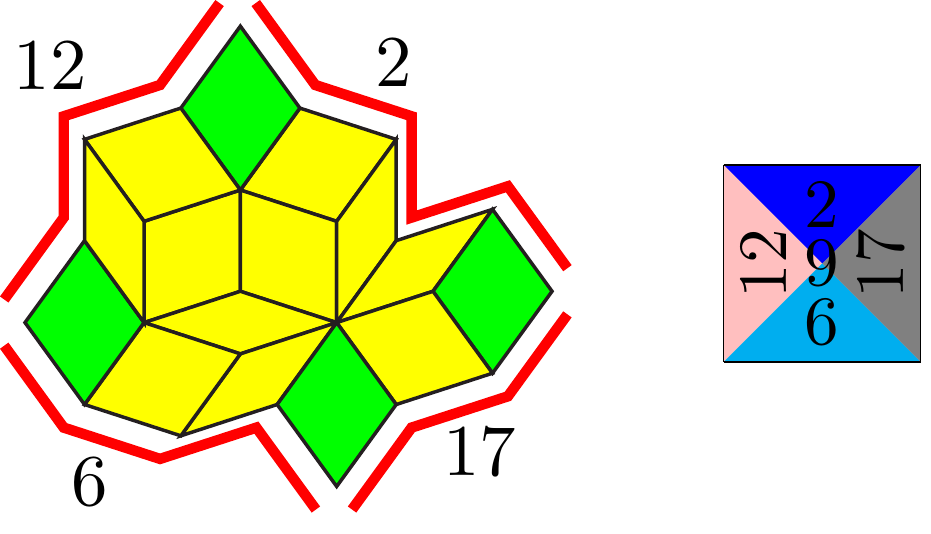}  
\end{subfigure}
\begin{subfigure}[H]{.32\textwidth} 
\centering
\includegraphics[width=.87\textwidth]{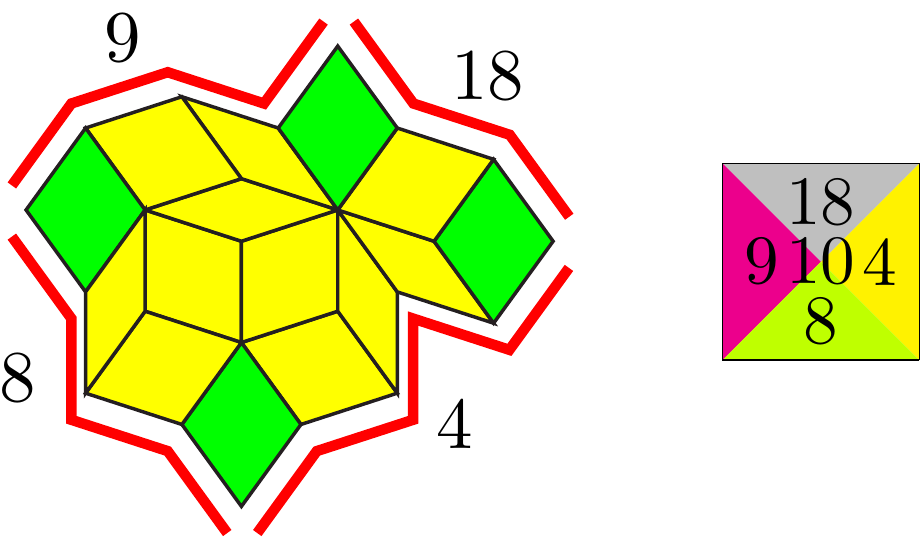}  
\end{subfigure}
\begin{subfigure}[H]{.32\textwidth} 
\centering
\includegraphics[width=.87\textwidth]{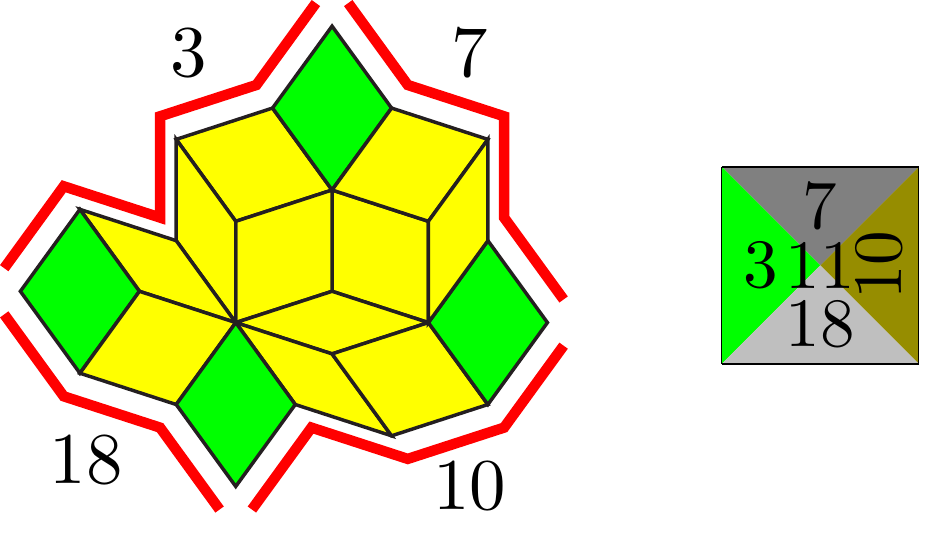}  
\end{subfigure}
\begin{subfigure}[H]{.32\textwidth} 
\centering
\includegraphics[width=.87\textwidth]{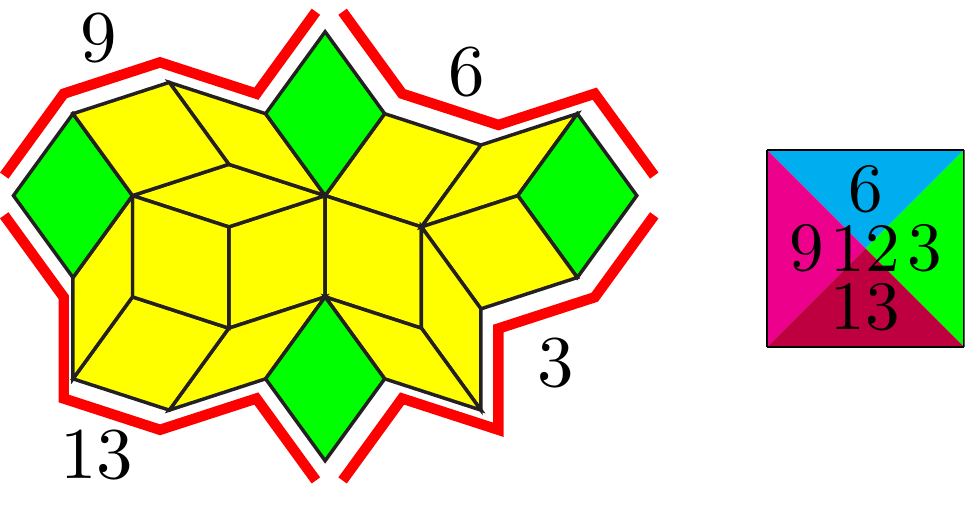}  
\end{subfigure}
\begin{subfigure}[H]{.32\textwidth} 
\centering
\includegraphics[width=.87\textwidth]{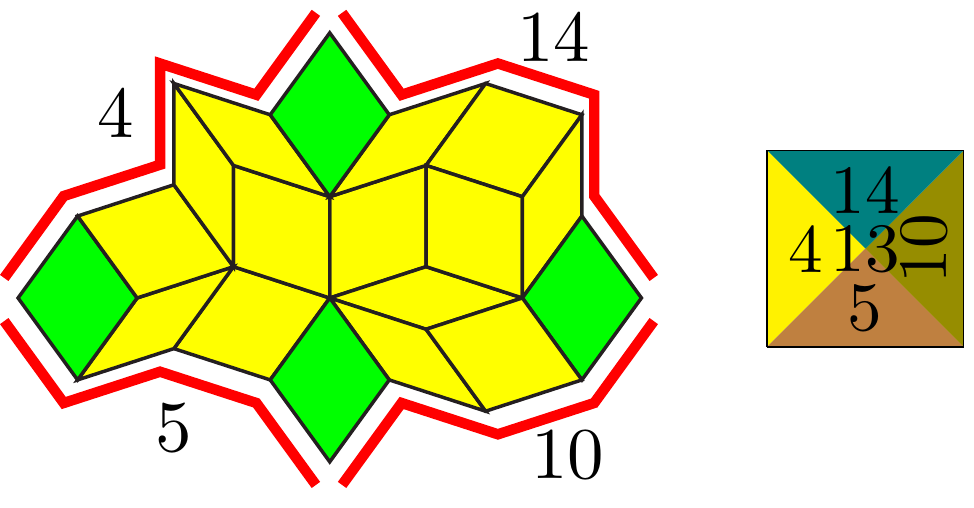}  
\end{subfigure}
\begin{subfigure}[H]{.32\textwidth} 
\centering
\includegraphics[width=.87\textwidth]{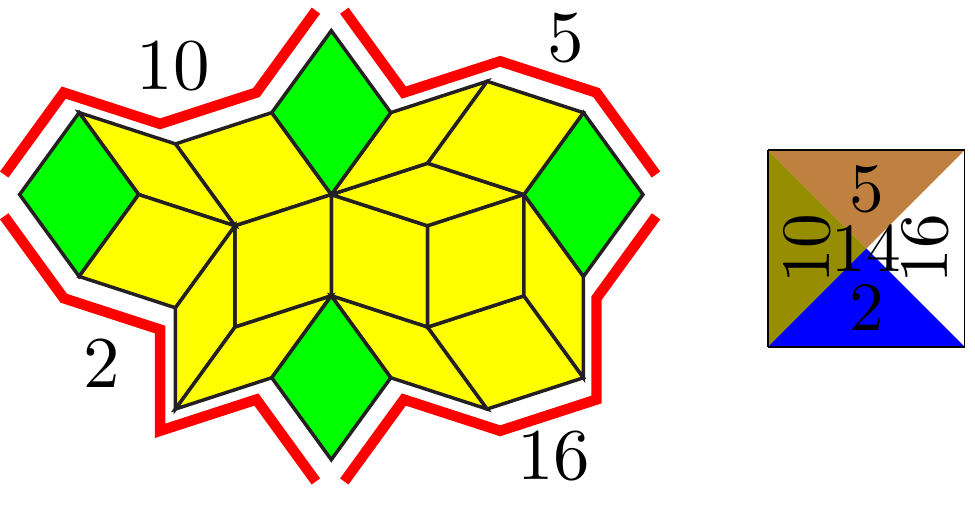}  
\end{subfigure}
\begin{subfigure}[H]{.32\textwidth} 
\centering
\includegraphics[width=.87\textwidth]{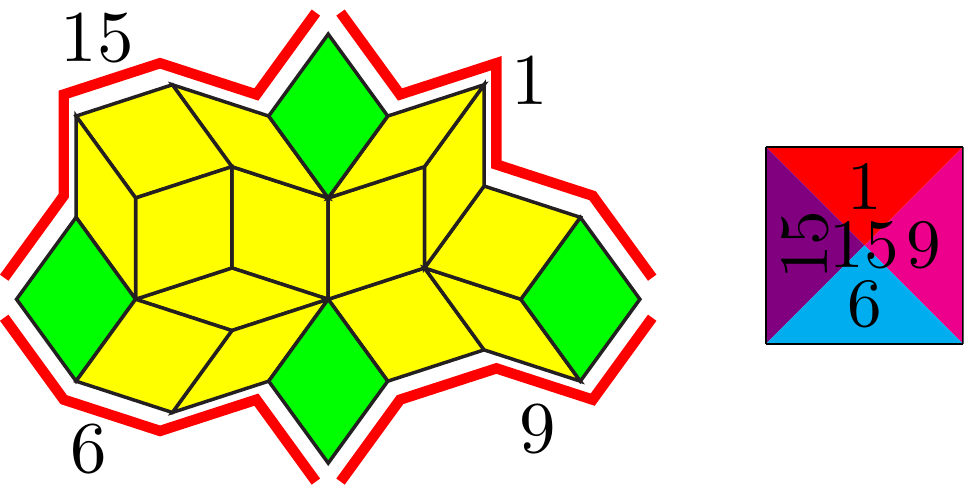}  
\end{subfigure}
\begin{subfigure}[H]{.32\textwidth} 
\centering
\includegraphics[width=.87\textwidth]{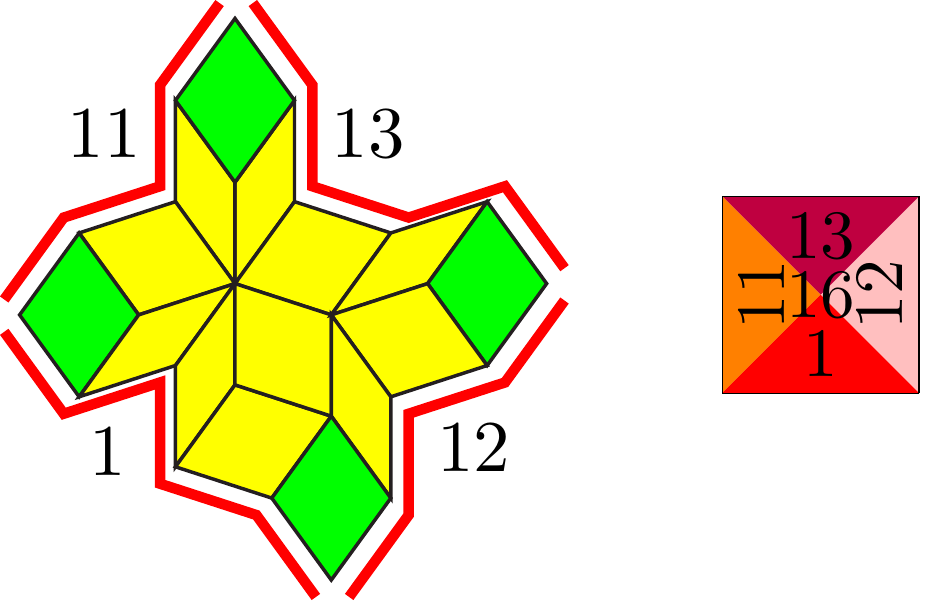}  
\end{subfigure}
\begin{subfigure}[H]{.32\textwidth} 
\centering
\includegraphics[width=.87\textwidth]{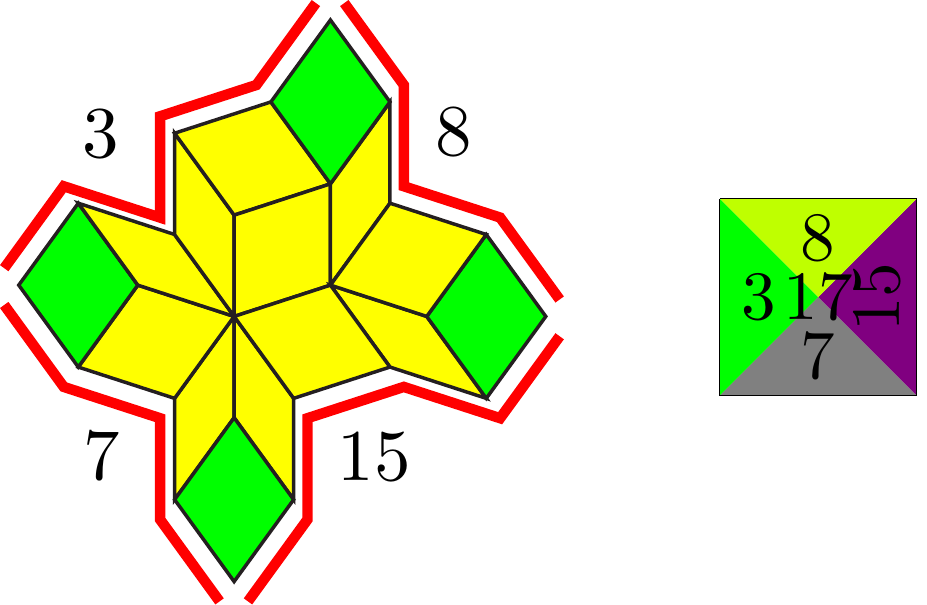}  
\end{subfigure}
\begin{subfigure}[H]{.32\textwidth} 
\centering
\includegraphics[width=.87\textwidth]{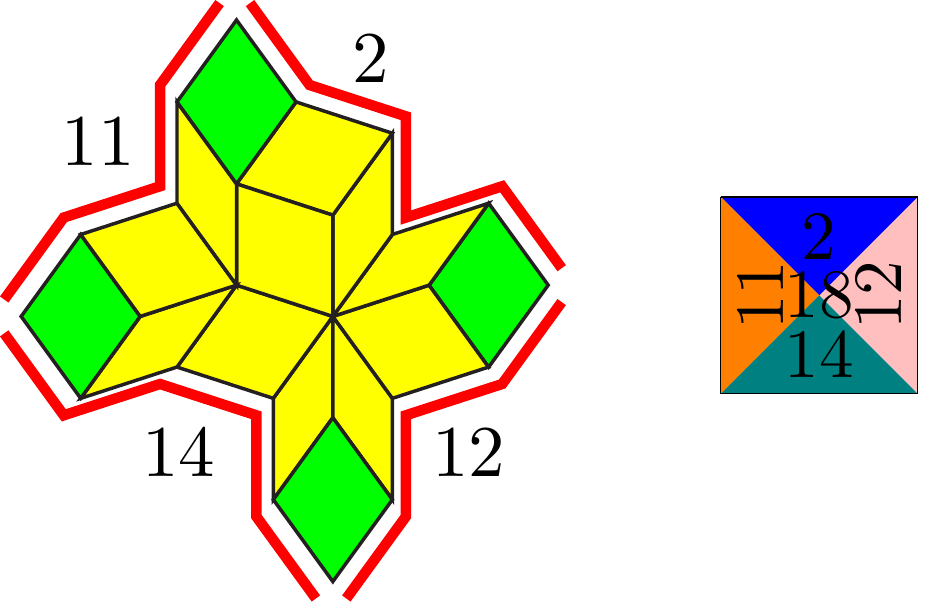}  
\end{subfigure}
\begin{subfigure}[H]{.32\textwidth} 
\centering
\includegraphics[width=.87\textwidth]{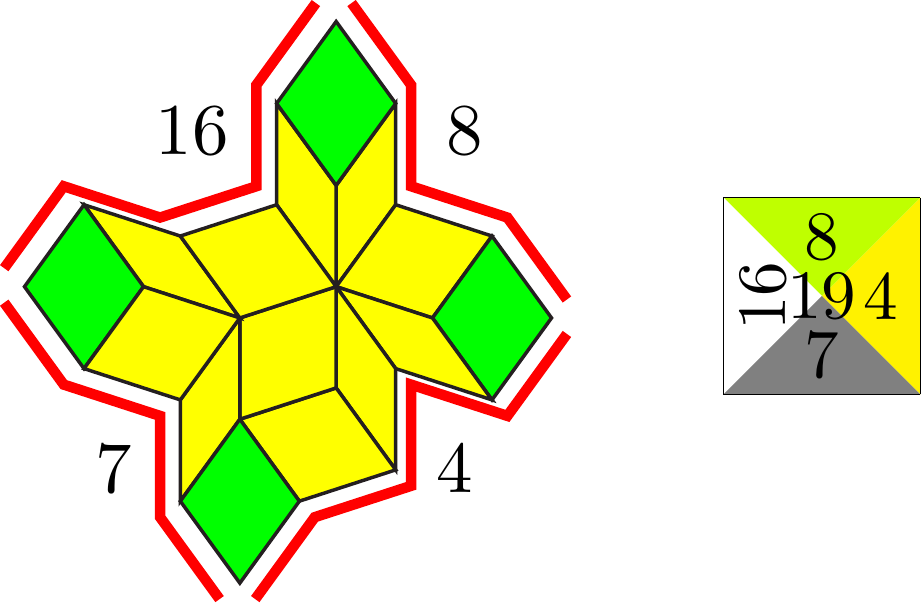}  
\end{subfigure}
\begin{subfigure}[H]{.32\textwidth} 
\centering
\includegraphics[width=.87\textwidth]{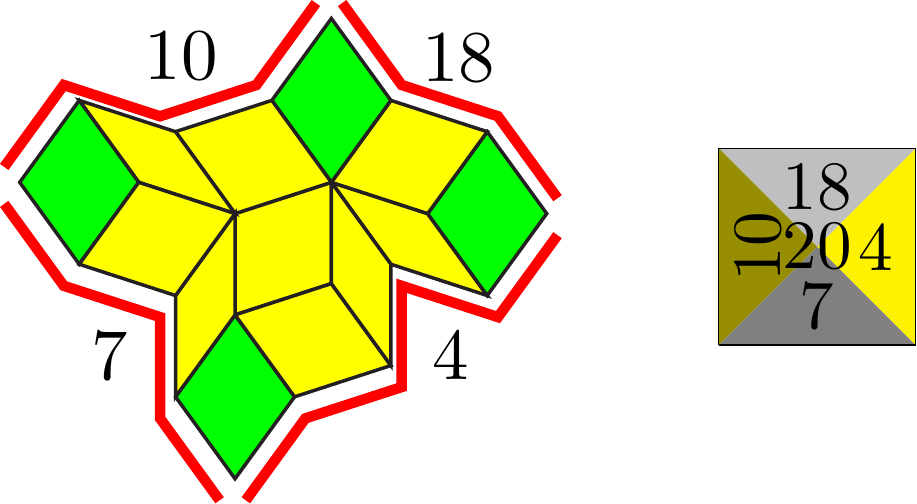}  
\end{subfigure}
\begin{subfigure}[H]{.32\textwidth} 
\centering
\includegraphics[width=.87\textwidth]{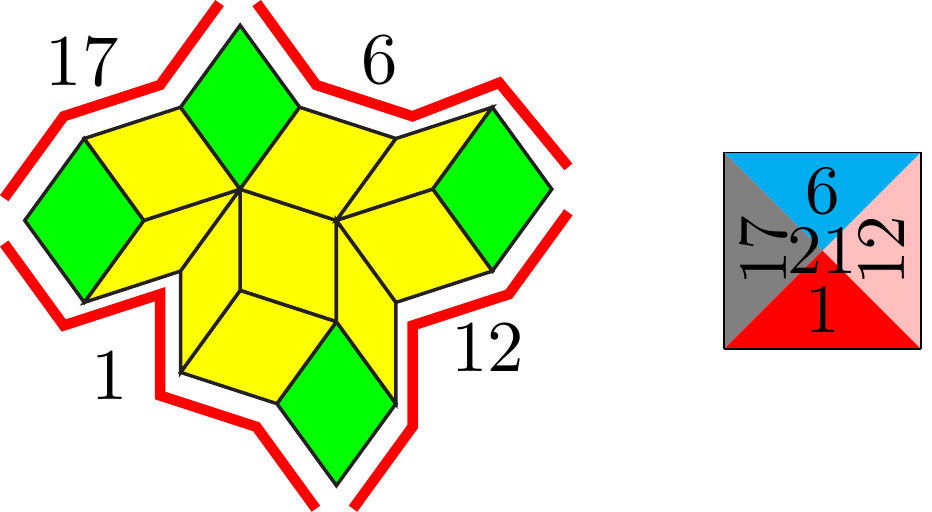}  
\end{subfigure}
\begin{subfigure}[H]{.32\textwidth} 
\centering
\includegraphics[width=.87\textwidth]{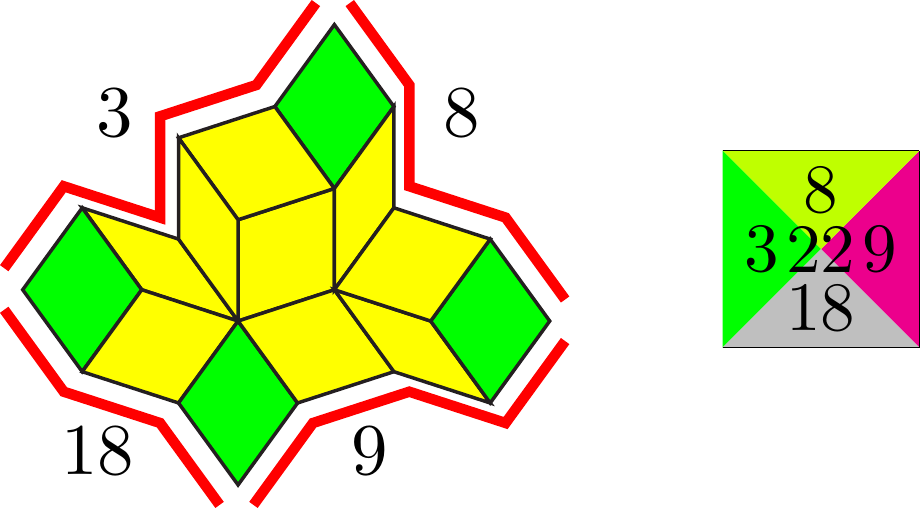}  
\end{subfigure}
\begin{subfigure}[H]{.32\textwidth} 
\centering
\includegraphics[width=.87\textwidth]{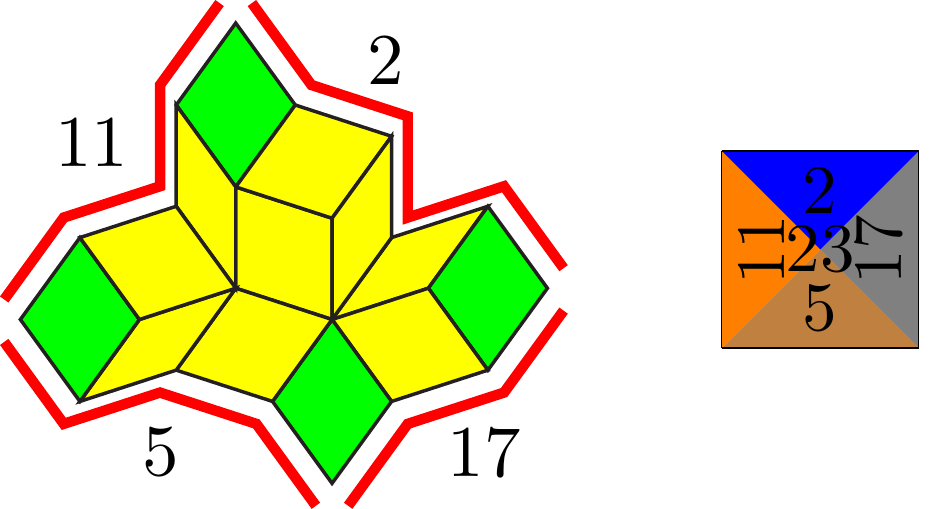}  
\end{subfigure}
\caption{Patches of Penrose rhomb tilings and their corresponding Wang tiles. Red arcs with the same label are translates of one another, and so these labels correspond to colors the sides of the Wang tiles.}\label{fig:Pen_patch_proto}
\end{figure}  

\begin{figure}[h]
\begin{center}
\includegraphics[width=0.8\textwidth]{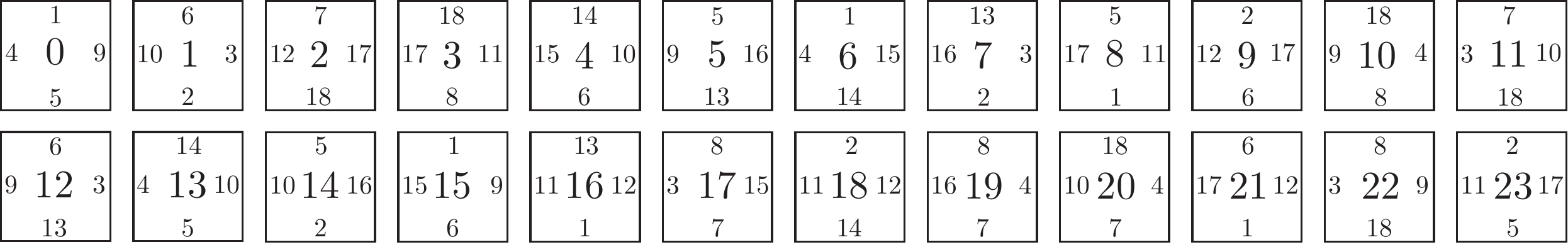}\caption{The order-24 aperiodic Wang protoset, $\mathcal{T}_{24}$ corresponding to P3.} \label{fig:Penrose-24-protoset}
\end{center}
\end{figure}

Let the protoseet of Figure \ref{fig:Penrose-24-protoset} be denoted by $\mathcal{T}_{24}$. We may use the computer to produce a large patch of tiling, and then using a matrix \[A_{24} = \left(\begin{matrix} \varphi & 0\\
0 & \varphi\end{matrix}\right),\] we can produce a dot pattern that reveals a partition for $\Omega_{\mathcal{T}_{24}}$.

\begin{lstlisting}
    Penrose24 = [(9,1,4,5),(3,6,10,2),(17,7,12,18),(11,18,17,8),
             (10,14,15,6),(16,5,9,13),(15,1,4,14),(3,13,16,2),
             (11,5,17,1),(17,2,12,6),(4,18,9,8),(10,7,3,18),
             (3,6,9,13),(10,14,4,5),(16,5,10,2),(9,1,15,6),
             (12,13,11,1),(15,8,3,7),(12,2,11,14),(4,8,16,7),
             (4,18,10,7),(12,6,17,1),(9,8,3,18),(17,2,11,5)]
    Pen24 = WangTileSet(Penrose24)
    PenSol24 = Pen24.solver(130,130)
    Pen24_tiling = PenSol24.solve(solver='glucose')
    
    A = matrix.column([[p,0],[0,p]])
    G = Pen24_tiling.plot_points_on_torus(A.inverse())
    show(G,aspect_ratio=1,figsize=8)
\end{lstlisting}

\begin{figure}[h]
\begin{center}
\includegraphics[width=0.75\textwidth]{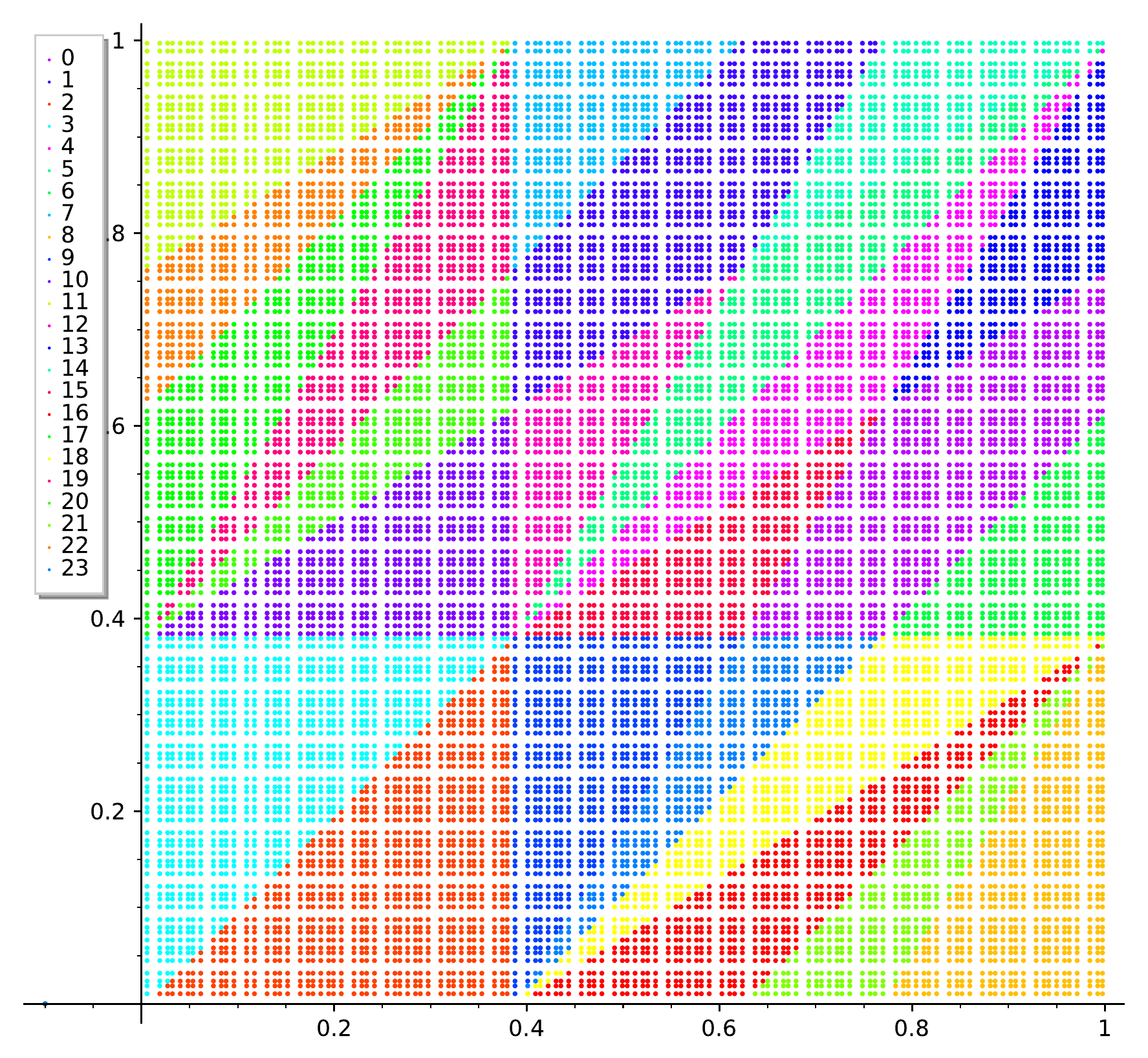}\caption{The result of applying $A^{-1}$ to each tile location in a 130x130 sized patch of a $\mathcal{T}_{24}$ tiling and coloring each point according to the label of the corresponding tile.} \label{fig:Penrose_24_dots}
\end{center}
\end{figure}

From this dot pattern, we can guess at the exact slopes of the lines and produce an exact partition of the torus $\mathbf{T} = \RR^2 \pmod{\Z^2}$, viewed as a quotient of the unit square. We have created such a partition in Figure \ref{fig:Pen_24_partition}. To simplify things, at right Figure \ref{fig:Pen_24_partition} we have scaled the partition by a factor of $\varphi$ to obtain a partition $\mathcal{P}_{24}$ , and we may consider the torus $\mathbf{T} = \RR^2 \pmod{\Gamma_{24}}$ where $\Gamma_{24} = \left<(\varphi,0),(0,\varphi)\right>$ with $\Z^2$ action $R_{24}$ defined on $\mathbf{T}$ by $R_{24}^{\mathbf{n}}(p) = p + \mathbf{n} \pmod{\mathbf{T}}$. From these ingredients, we can now consider the space $\mathcal{X}_{\mathcal{P}_{24},R_{24}}$ and testable hypotheses concerning that space that lends itself to the techniques developed by Labb\'{e} in \cite{Labb2021}, such as:

\begin{itemize}
    \item Are configurations in $\mathcal{X}_{\mathcal{P}_{24},R_{24}}$ always valid tilings in $\Omega_{\mathcal{T}_{24}}$? I.e., Is $\mathcal{X}_{\mathcal{P}_{24},R_{24}} \subseteq \Omega_{\mathcal{T}_{24}}$?
    \item Is $\mathcal{X}_{\mathcal{P}_{24},R_{24}}$ a shift of finite type?
    \item Is $\mathcal{X}_{\mathcal{P}_{24},R_{24}}$ minimal in $\Omega_{\mathcal{T}_{24}}$?
    \item Is it true that $\mathcal{X}_{\mathcal{P}_{24},R_{24}}= \Omega_{\mathcal{T}_{24}}$, so that the partition $\mathcal{P}_{24}$ produces all possible Penrose tilings?
    \item Will $\mathcal{X}_{\mathcal{P}_{24},R_{24}}$ exhibit nonexpansive directions, as corresponding space $\mathcal{X}_{\mathcal{P}_{0},R_{0}}$ did for the Jeandel-Rao shift?
\end{itemize}

\begin{figure}[h]
\begin{center}
\includegraphics[width=\textwidth]{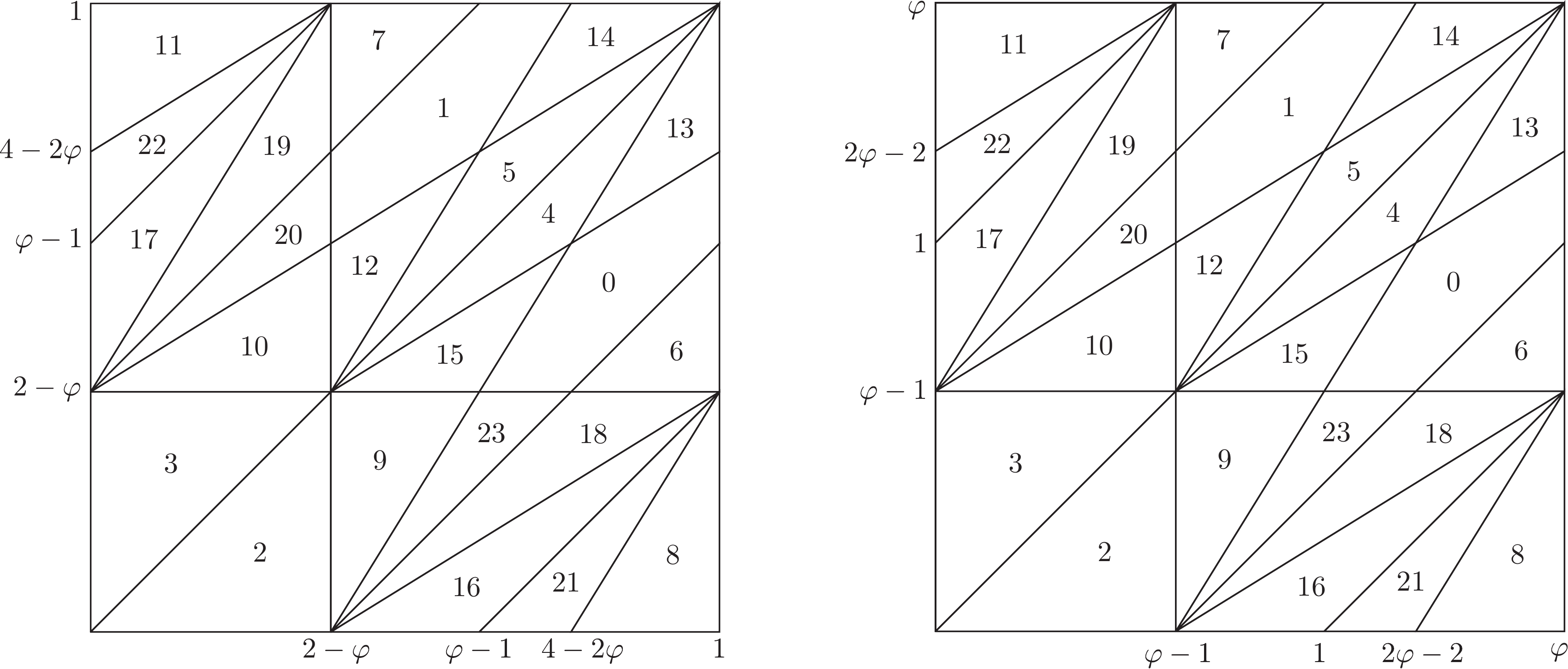}\caption{The partition deduced from the dot pattern of Figure \ref{fig:Penrose_24_dots} is at left. At right, this partition is scaled by a factor of $\varphi$, giving us partition $\mathcal{P}_{24}$, which makes the formula for the $\Z^2$ action $R_{24}$ on the partitioned torus simpler. The slopes of lines in the partition are $0$, $1$, $\infty$, $\varphi$, and $1/\varphi$.} \label{fig:Pen_24_partition}
\end{center}
\end{figure}

We will pursue questions such as these in Appendix \ref{apx:app}, but for now, we will continue with this experimental approach to finding a few other potential interesting partitions for Wang tile protosets.

\subsection{The Order-16 Ammann A2 Wang Tile Protoset}\label{subsec:Amm_16}
The Ammann A2 order-2 aperiodic protoset is shown in Figure \ref{fig:Ammann_A2}. Notice at bottom in Figure \ref{fig:Ammann_A2} that the tiles are adorned with red and blue lines called \emph{Ammann bars}. The matching rule is that when two copies of tiles from A2 meet, Ammann bars of the same color must continue in a straight line. In Figure \ref{fig:A2_patch}, we see how the matching rules are enforced in a small patch of tiling for the A2 protoset. 

\begin{figure}[h]
\begin{center}
\includegraphics[width=0.5\textwidth]{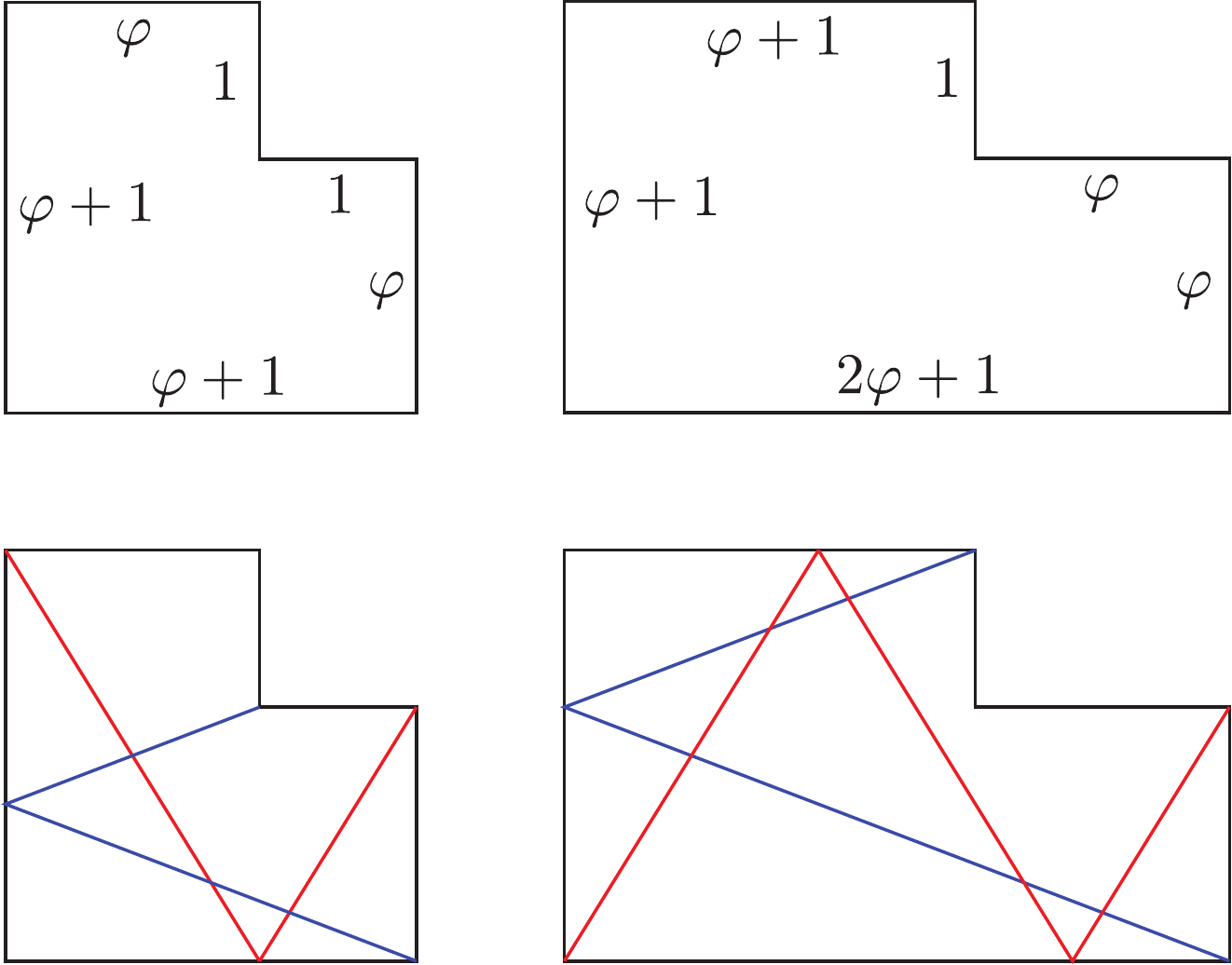}\caption{The Ammann A2 order-2 aperiodic protset. The top row shows the dimensions of the two tiles ($\varphi$ is the golden mean), and on the bottom row, the tiles are adorned with decorations called \emph{Ammann bars} that indicate the matching rules that enforce aperidocity for A2.} \label{fig:Ammann_A2}
\end{center}
\end{figure}

\begin{figure}[h]
\begin{center}
\includegraphics[width=0.5\textwidth]{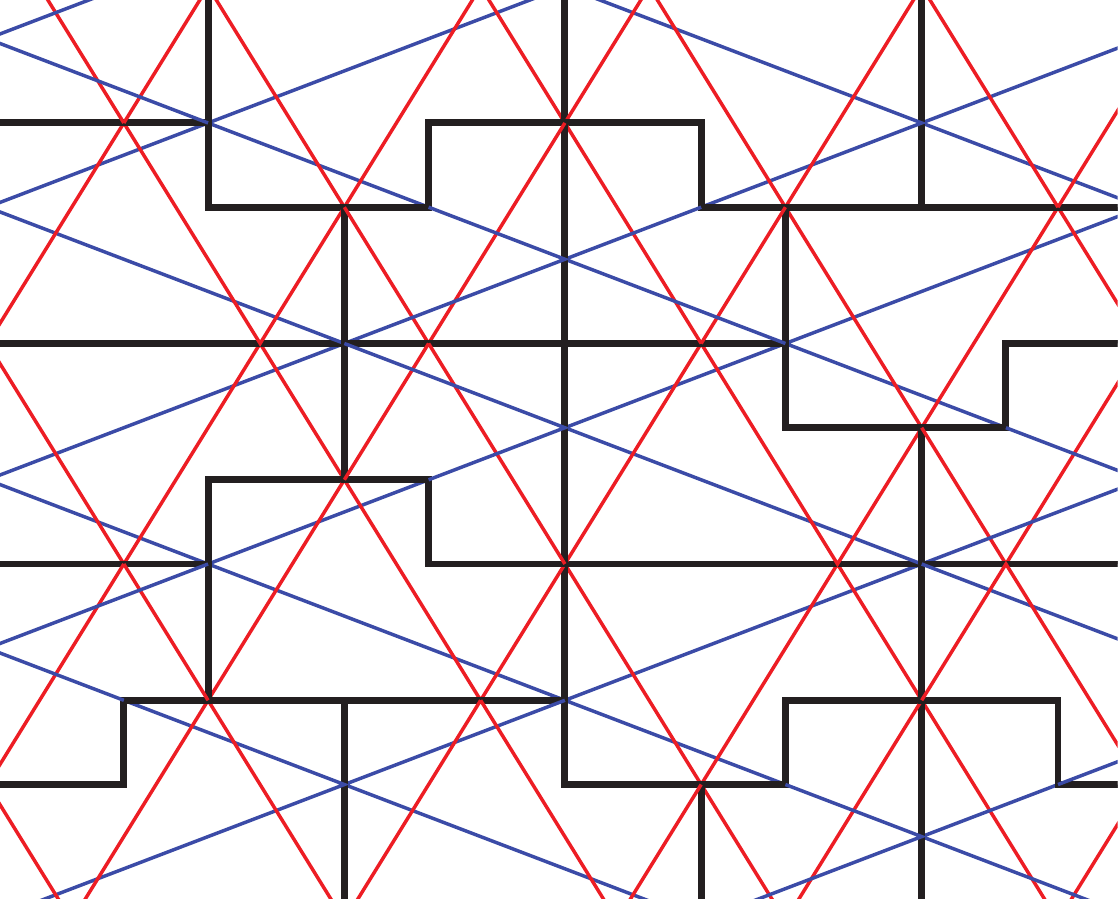}\caption{A patch of tiling formed from the A2 protoset} \label{fig:A2_patch}
\end{center}
\end{figure}

The key to converting the A2 protoset to a Wang tile protoset of order 16 is to consider all of the possible rhombi formed from the red lines, with the blue lines serving as the matching rules for the red rhombi, as explained in In \cite[Section 11.1]{GS1987}. The protoset so produced, $\mathcal{T}_{16}$, is shown in Figure \ref{fig:Ammann_16_Protoset}.

\begin{figure}[h]
\begin{center}
\includegraphics[width=0.8\textwidth]{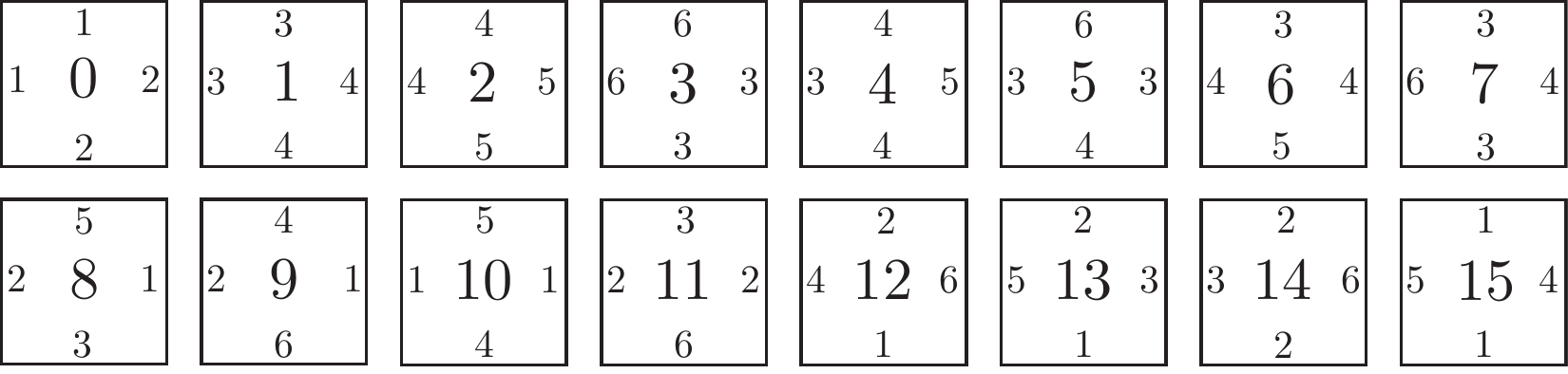}\caption{The order-16 aperiodic Wang tile protoset, $\mathcal{T}_{16}$, derived from the A2 protoset.} \label{fig:Ammann_16_Protoset}
\end{center}
\end{figure}

Again, as we did with the $\mathcal{T}_{24}$, we can use the computer to produce a large patch of tiling, and with the appropriate choice of matrix, we can pull the tiling back onto a torus and see a dot pattern that indicates that a partition underlies the protoset. Indeed, using \[A_{16} = \left( \begin{matrix}
    \varphi & 0\\
    0 & \varphi
\end{matrix}\right), \] we obtain the nice dot pattern shown in Figure \ref{fig:A2_dots}. Interestingly, we see not only that $A_{24} = A_{16}$, but also that the dot patterns for $\mathcal{T}_{24}$ and $\mathcal{T}_{16}$ align - the pattern for P3 seems to be a refinement of the one for A2 (Figure \ref{fig:P3_and_A2_dots}). This is not coincidental. In \cite[Section 11.1]{GS1987} it is pointed out that $\mathcal{T}_{16}$ can be derived from the Penrose darts and kites (P2) protoset using the idea of Ammann bars.

\begin{figure}[h]
\begin{center}
\includegraphics[width=0.8\textwidth]{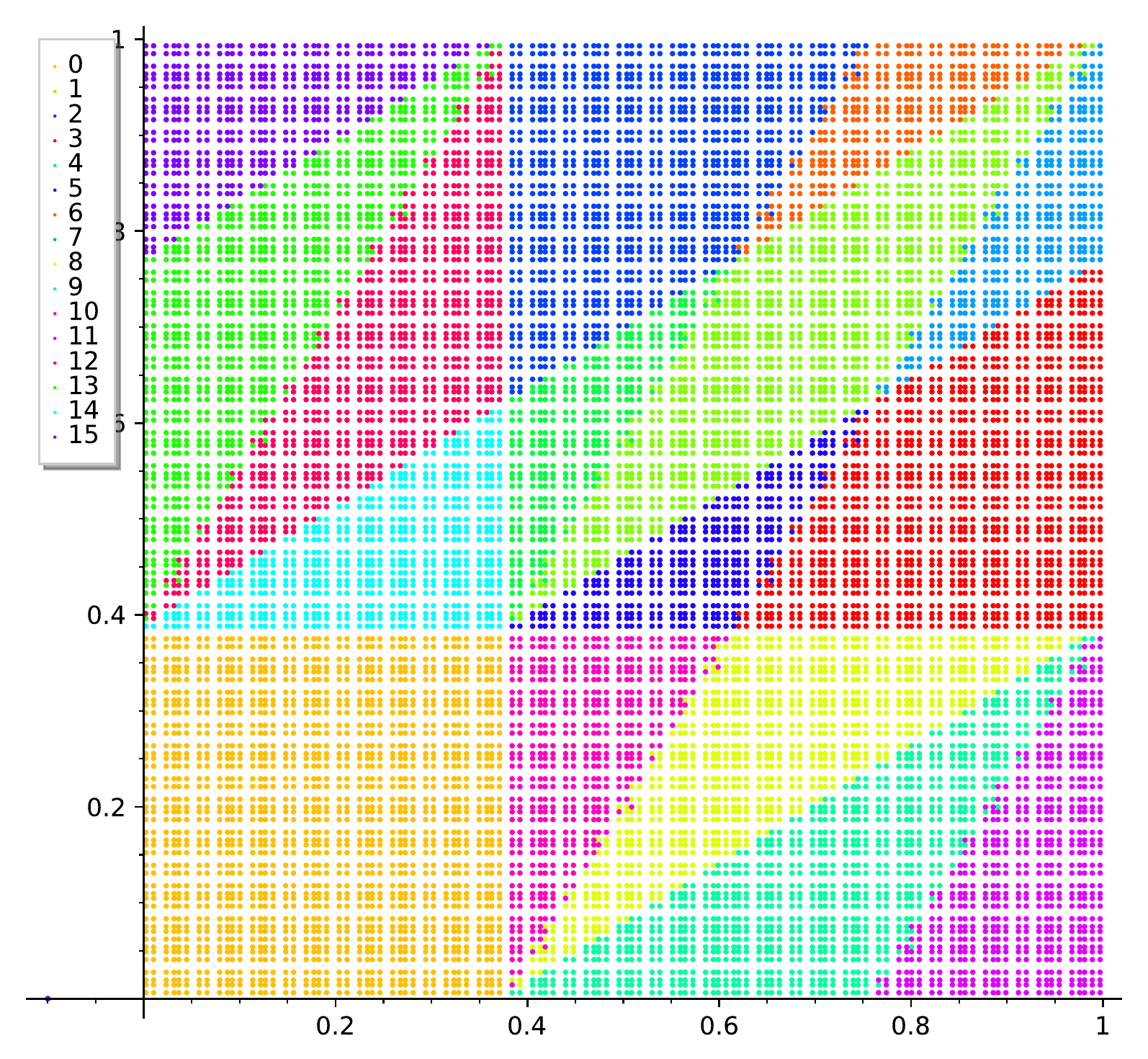}\caption{The dot pattern produced by applying matrix $A_{16}$ to a $120 \times 120$ patch of tiling by the order-16 Wang tile protoset derived from the Ammann A2 aperiodic protoset.} \label{fig:A2_dots}
\end{center}
\end{figure}

\begin{figure}[h]
\centering
\begin{subfigure}[b]{.45\textwidth} 
\centering
\includegraphics[width=\textwidth]{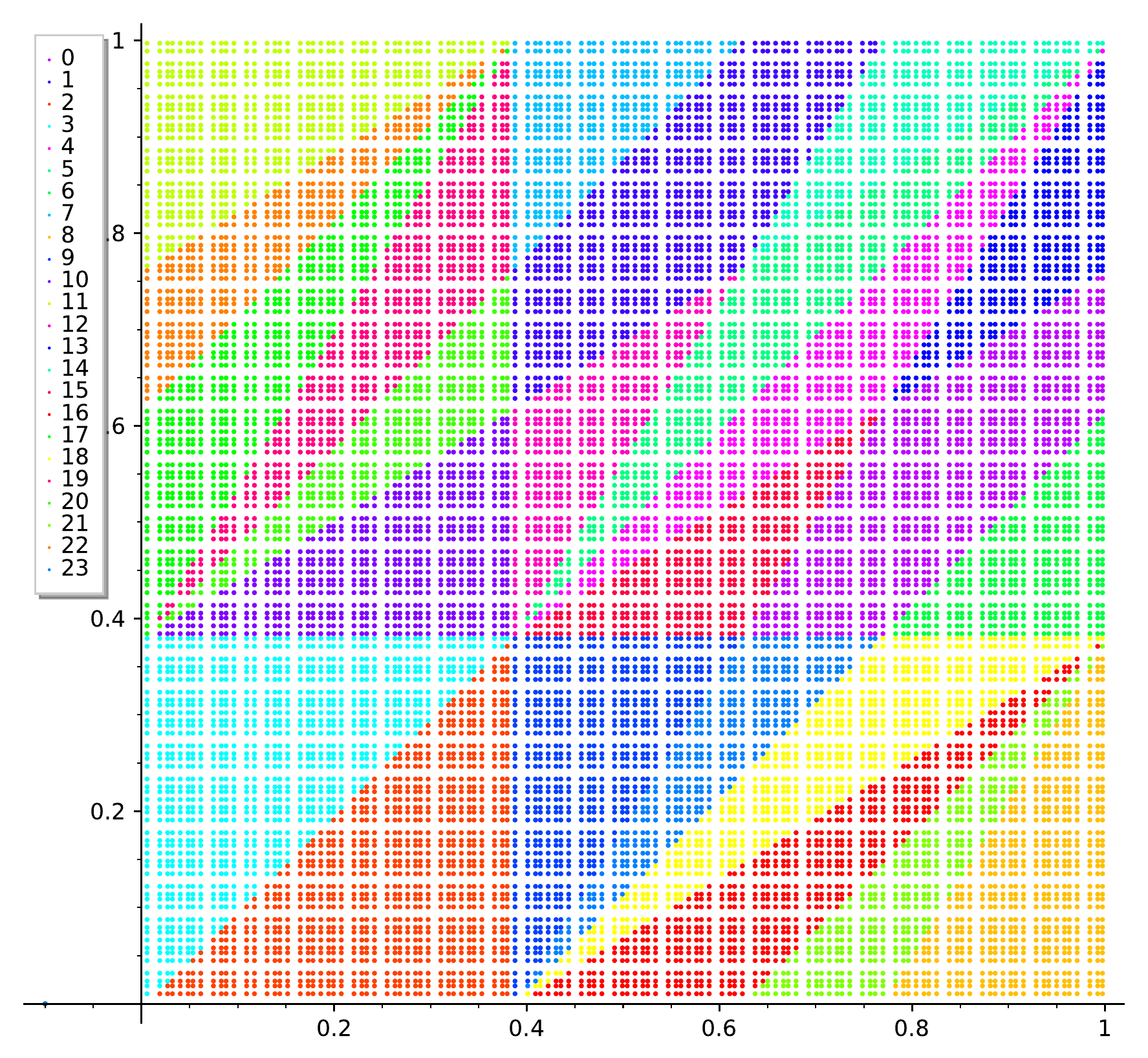}  
\caption{Dot pattern for $\mathcal{T}_{24}$}
\end{subfigure}
\begin{subfigure}[b]{.45\textwidth} 
\centering
\includegraphics[width=\textwidth]{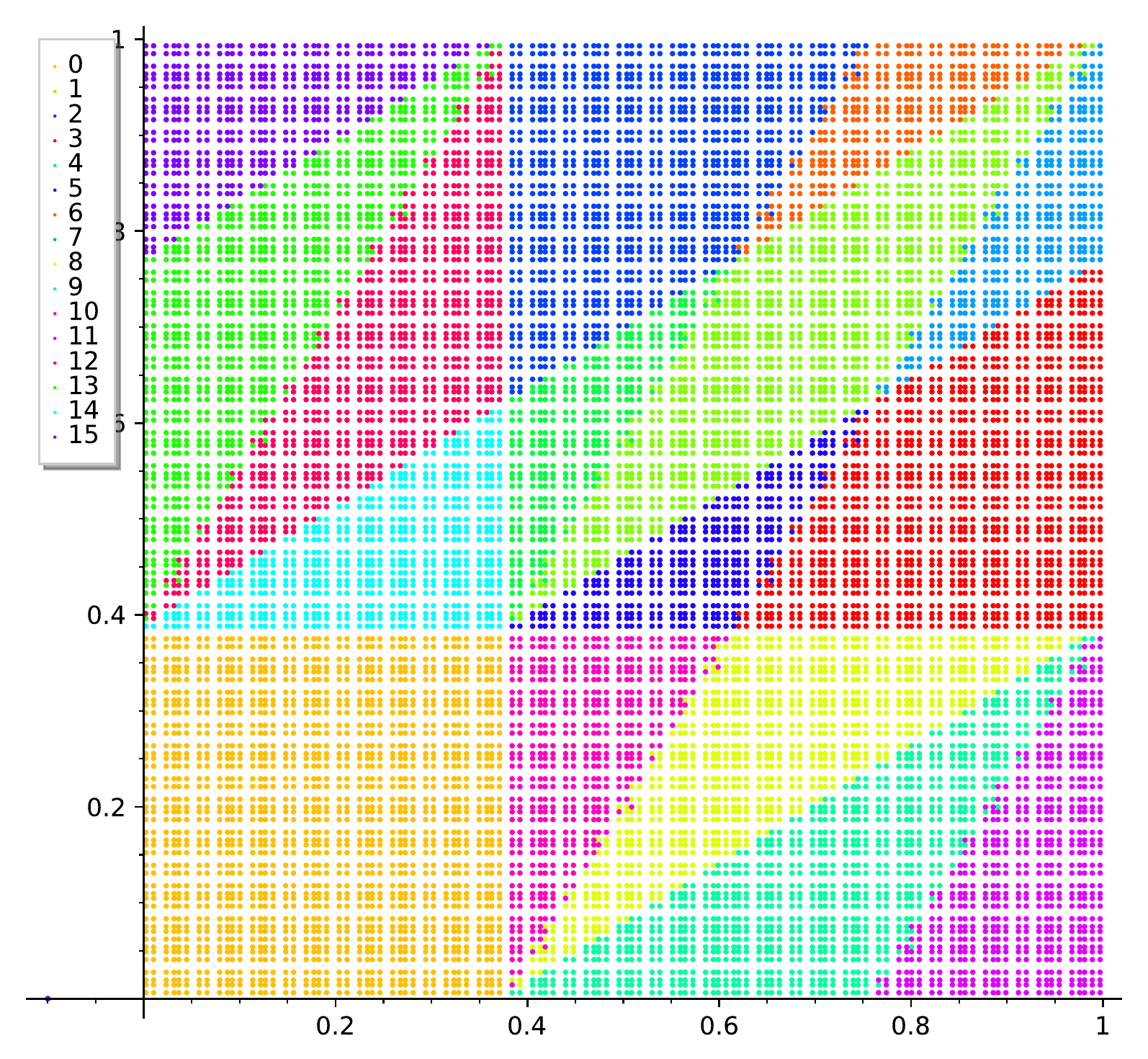}  
\caption{Dot pattern for $\mathcal{T}_{16}$}
\end{subfigure}
\caption{The dot patterns for $\mathcal{T}_{24}$ and $\mathcal{T}_{16}$ are remarkably similar; the pattern for $\mathcal{T}_{16}$ can be obtained from the one for $\mathcal{T}_{24}$ by removing the lines of slope 1.}\label{fig:P3_and_A2_dots}\end{figure}

From the dot pattern obtained for $\mathcal{T}_{16}$, we can determine the exact equations of the lines to get a partition $\mathcal{P}_{16}$, as shown in Figure \ref{fig:Amm_16_partition}. We have scaled the partitioned torus for $\mathcal{T}_{16}$ by a factor of $\varphi$ (as we did for $\mathcal{T}_{24}$). Define the torus $\mathbf{T}$ identifying opposites sides of this square $[0,\varphi]\times[0,\varphi]$ and define the $\Z^2$ action $R_{16}$ on $\mathbf{T}$ by $R_{16}^{\mathbf{n}}(p) = p + \mathbf{n} \pmod{\mathbf{T}}$. We may now consider the space $\mathcal{X}_{\mathcal{P}_{16},R_{16}}$ and consider questions like those outlined in Subsection \ref{subsec:Pen_24}.

\begin{figure}[h]
\begin{center}
\includegraphics[width=0.5\textwidth]{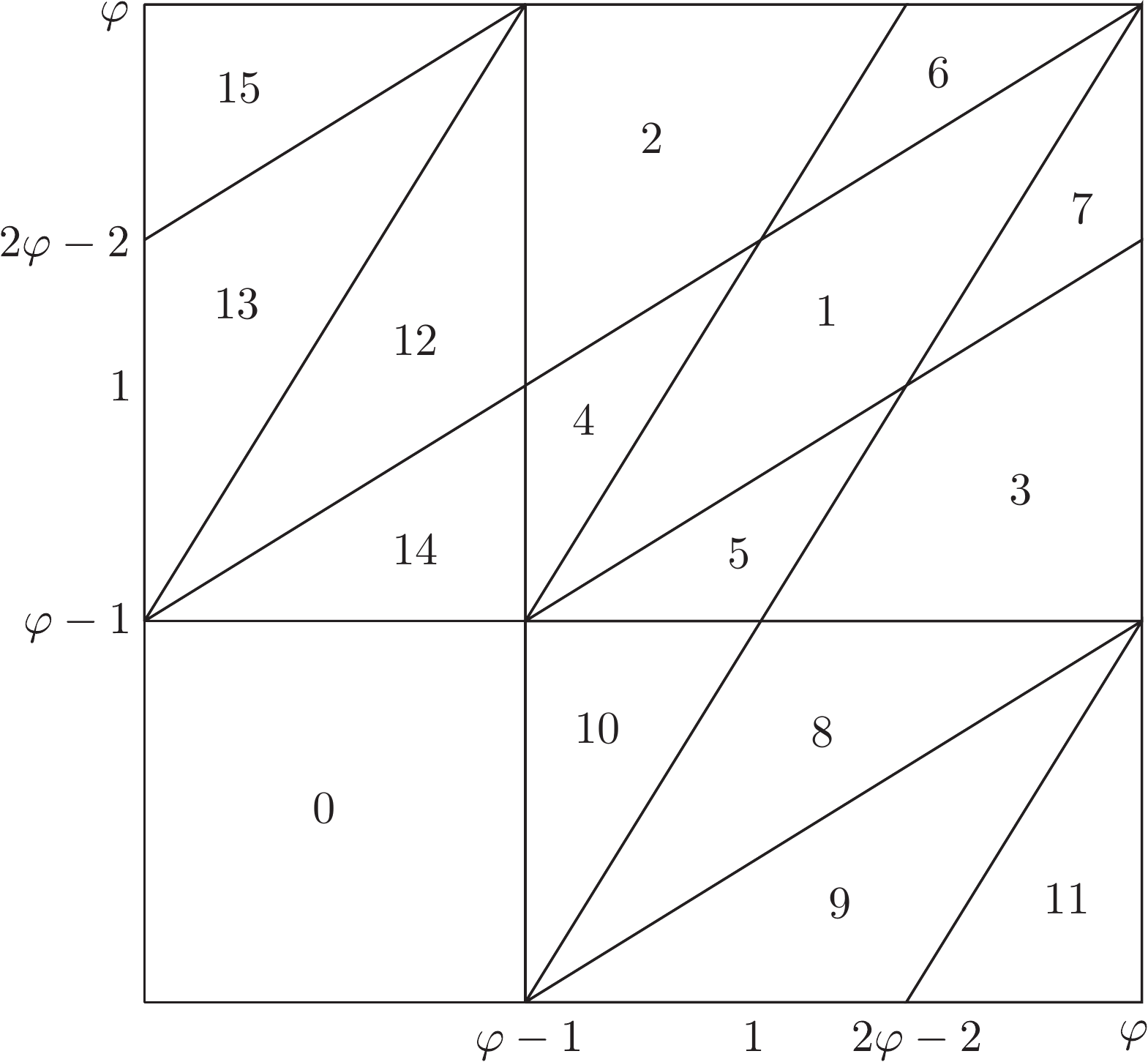}\caption{The partition $\mathcal{P}_{16}$ obtained from the order-16 protoset $\mathcal{T}_{16}$.} \label{fig:Amm_16_partition}
\end{center}
\end{figure}

\subsection{Testing $\mathcal{X}_{\mathcal{P}_{24},R_{24}}$ and $\mathcal{X}_{\mathcal{P}_{16},R_{16}}$}

Before considering more technical theoretical questions about $\mathcal{X}_{\mathcal{P}_{24},R_{24}}$ and $\mathcal{X}_{\mathcal{P}_{16},R_{16}}$, we will test, experimentally, if the configurations in these spaces seem to be valid tilings. In doing so, we will consider orbits of points on the torus that intersect the boundary and points whose orbits do not.

Let $\mathbf{T} = [0,\varphi]\times[0,\varphi]$ where $\mathbf{T}$ has partition $\mathcal{P}_{24}$ and let $p = (0.1, 0.2) \in \mathbf{T}$. We then consider the orbit $\mathcal{O}_{R_{24}}(p) \subset \mathbf{T}$. We wish to see if the tiling (or at least a finite portion we can generate) spelled out by $\mathcal{O}_{R_{24}(p)}$ seems to be valid. This process could be automated, but we will do it ``by hand", drawing the orbit in Figure \ref{fig:Pen_24_orbit} and the corresponding partial configuration in Figure \ref{fig:Pen_24_tiling}. Notice that the partial tiling so generated is valid. Also note that we could easily convert the Wang tiling of Figure \ref{fig:Pen_24_tiling} into a tiling by the Penrose rhombs (P3) simply by substituting the patches described in Figure \ref{fig:Pen_patch_proto} for the corresponding Wang tiles. Similarly, in Figure \ref{fig:Amm_16_orbit} we see part of the orbit $\mathcal{O}_{R_{16}}(p)$ of the point $p = (0.1,0,1)$ in the torus $\mathbf{T}$ partitioned by $\mathcal{P}_{16}$ the corresponding partial tiling by the order-16 Ammann Wang tiles ($\mathcal{T}_{16}$) in \ref{fig:Amm_16_tiling}, which again appears to be valid.

It is not hard to check that the orbit of any point $p \in \mathbf{T}$, under either action $R_{24}$ or $R_{16}$, will be dense in $\mathcal{T}$. Thus, one immediate consequence of $\mathcal{X}_{\mathcal{P}_{24},R_{24}}$ and $\mathcal{X}_{\mathcal{P}_{16},R_{16}}$ being subspaces of $\Omega_{24}$ and $\Omega_{16}$ (respectively) is that the areas of atoms in $\mathcal{P}_{24}$ and $\mathcal{P}_{16}$ correspond to the proportion of tiles of specific kinds in tilings. For example, the area of the atom labeled $0$ in $\mathcal{P}_{16}$ (Figure \ref{fig:Amm_16_partition}) is $(\varphi-1)^2 = 2 - \varphi$, which comprises $(2-\varphi)/\varphi^2 = 5-3\varphi \approx 14.59\%$ of the area of the torus $\mathbf{T}$. This means that in any tiling in $\mathcal{X}_{\mathcal{P}_{16},R_{16}}$, about 14.59\% of the tiles will be copies of the prototile $T_0$.

\begin{figure}[h]
\begin{center}
\includegraphics[width=0.7\textwidth]{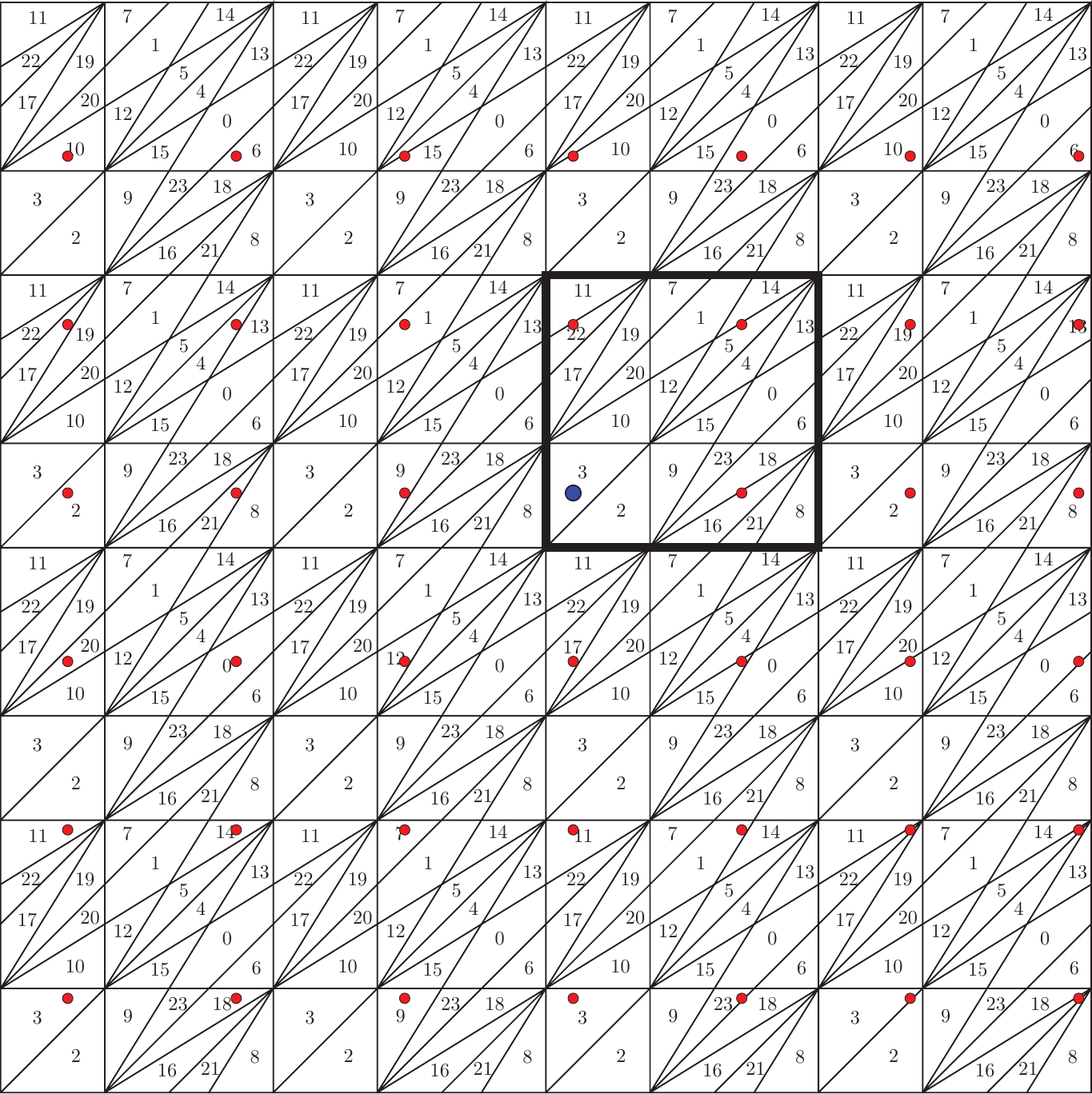}\caption{The orbit of the point $p = (0.1, 0.2)$ under the action $R_{24}$ in the torus $\mathbf{T} = [0,\varphi]\times [0,\varphi]$ where $\mathbf{T}$ has the partition $\mathcal{P}_{24}$. The point $p$ is in blue, and the rest of its orbit is in red.} \label{fig:Pen_24_orbit}
\end{center}
\end{figure}

\begin{figure}[h]
\begin{center}
\includegraphics[width=0.6\textwidth]{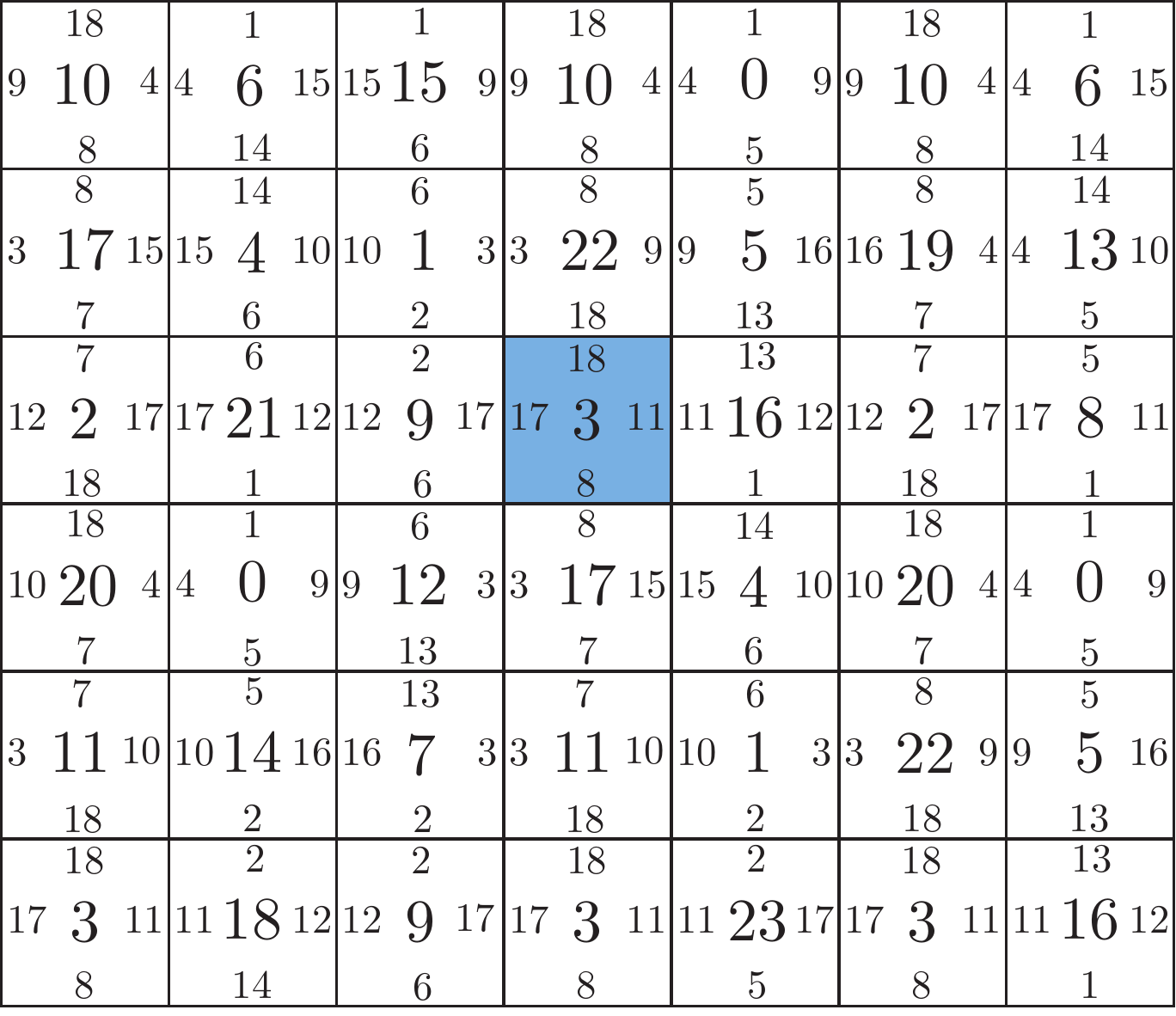}\caption{The $\mathcal{T}_{24}$ tiling corresponding to the orbit $\mathcal{O}_{R_{24}}((0.1,0.2))$ shown in Figure \ref{fig:Pen_24_orbit}. The blue tile corresponds to the point $(0.1,0.2)$.} \label{fig:Pen_24_tiling}
\end{center}
\end{figure}

\begin{figure}[h]
\begin{center}
\includegraphics[width=0.7\textwidth]{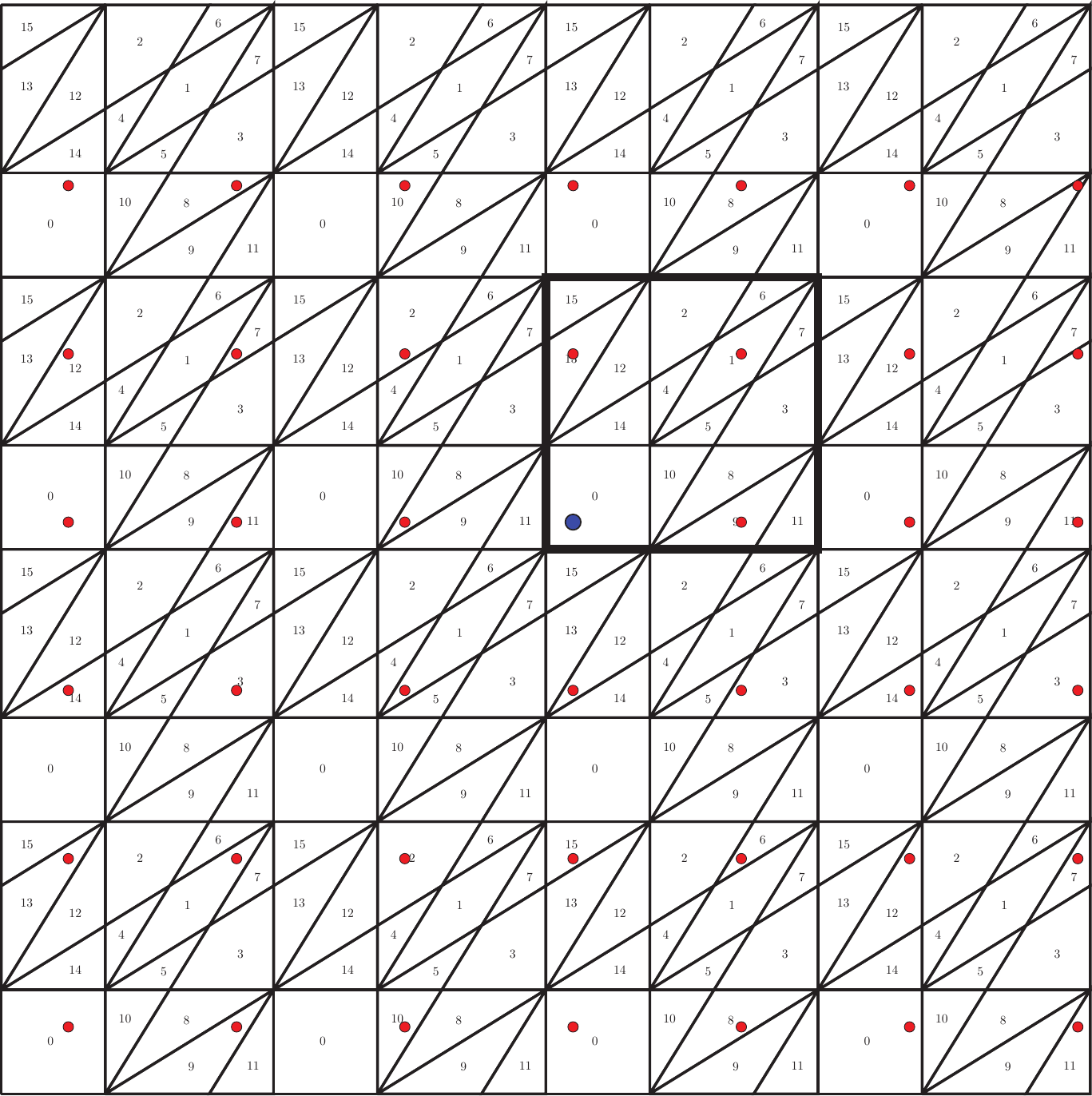}\caption{The orbit of the point $p = (0.1, 0.1)$ under the action $R_{16}$ in the torus $\mathbf{T} = [0,\varphi]\times [0,\varphi]$ where $\mathbf{T}$ has the partition $\mathcal{P}_{16}$. The point $p$ is in blue, and the rest of its orbit is in red.} \label{fig:Amm_16_orbit}
\end{center}
\end{figure}

\begin{figure}[h]
\begin{center}
\includegraphics[width=0.6\textwidth]{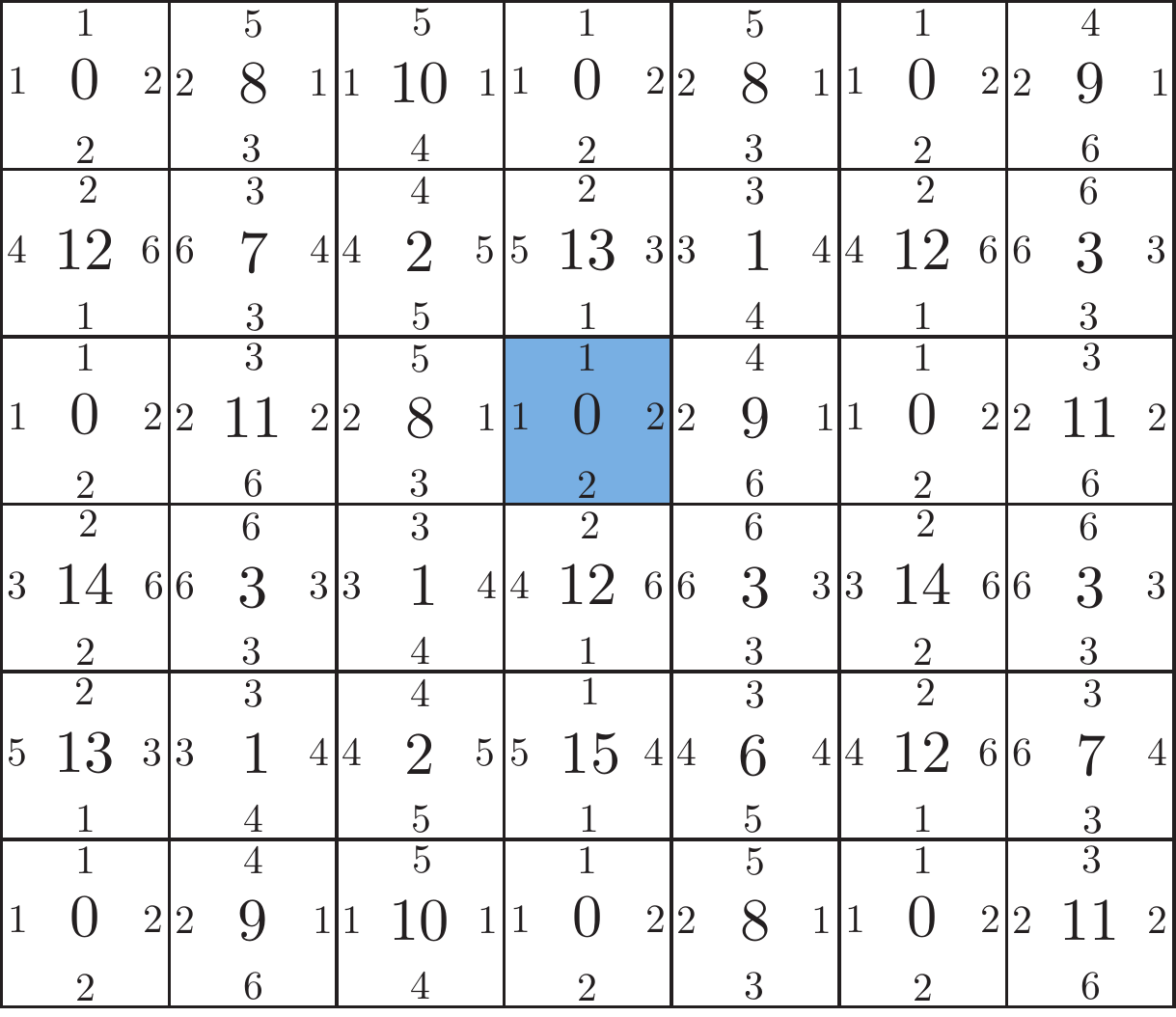}\caption{The $\mathcal{T}_{16}$ tiling corresponding to the orbit $\mathcal{O}_{R_{16}}((0.1,0.1))$ shown in Figure \ref{fig:Amm_16_orbit}. The blue tile corresponds to the point $(0.1,0.1)$.} \label{fig:Amm_16_tiling}
\end{center}
\end{figure}

\subsection{Orbits that intersect the boundary of $\mathcal{P}$}
Let $\Delta_{16}$ be the union of the boundary lines of the atoms in $\mathcal{P}_{16}$. In this subsection we consider an example where the orbit of a point in $\mathcal{T}$ intersects $\Delta_{16}$, and we note that the details are similar in $\mathcal{P}_{24}$. We will illustrate by example how the ambiguity that arises in assigning tiles to points in the orbit on $\Delta_{16}$ is resolved and we will see a certain interesting phenomenon in the dynamics of $\mathcal{X}_{\mathcal{P},R}$. 

Let $p = \mathbf{0} = (0,0) \in \mathbf{T}$ and let $q = (1/2,(3/2)\varphi - 1) \in \mathbf{T}$. Notice that $p$ lies at the intersection of lines of all 4 slopes of boundary lines in $\mathcal{P}$ since this torus $\mathbf{T}$ is formed as a quotient by identifying opposite sides. Also notice that $q$ lies on a boundary segment of slope $\varphi$ connecting $(0,\varphi-1)$ and $(\varphi -1, 1)$. Consider the orbit $\mathcal{O}_{R_{16}}(p)$. Using the computer, inside of a finite window of size $[-40,40]\times[-40,40]$, we can check which points $\mathbf{n} \in \Z^2$ satisfy $R_{16}^{\mathbf{n}}(\mathbf{0}) \pmod{\mathbf{T}} \in \Delta_{16}$, and a plot of these points $\mathbf{n}$ is shown in Figure \ref{fig:Amm_0_orbit}. Similarly, inside the same size window, we plot the points $\mathbf{n} \in \Z^2$ satisfying $R_{16}^{\mathbf{n}}(q) \pmod{\mathbf{T}} \in \Delta_{16}$ in Figure \ref{fig:Amm_line_1_orbit}. Interestingly, we find that the points $\mathbf{n}\in\Z^2$ where the orbit falls on the boundary forms straight strips of points, sometimes a single strip (when the orbit touches lines in the boundary of just one slope) or four separate strips (when the orbit touches lines in the boundary of different slopes). The directions of the strips are called the \emph{nonexpansive directions} of $\mathcal{X}_{\mathcal{P}_{16},R_{16}}$. In \cite{Mann2022}, the authors provide methodology for determining the exact directions of such nonexpansive directions, and we leave that topic for the reader to explore further. We only mention here that these nonexpansive directions correspond to what are commonly called \emph{Conway worms} in Penrose tilings \cite{Gardner1}; nonexpansive directions provides a very nice and more general explanation for such behavior in some tilings. 

In either situation, the ambiguity in the corresponding tiling can be resolved by choosing a direction to determine what tile to place at the points $\mathbf{n}$ such that $R^{\mathbf{n}}_{16}(p)$ lies on $\Delta_{16}$. For example, if we again consider the orbit $\mathcal{O}_{R_{16}}(q)$ where $q$ is as in Figure \ref{fig:Amm_line_1_orbit}, let us specify a direction vector $\mathbf{v} = (-1,1)$. Notice that no line segment in $\Delta_{16}$ is parallel to $\mathbf{v}$. Now, at each point of $\mathbf{n}$ where $\mathcal{O}_{R_{16}}(q)$ intersects $\Delta_{16}$, place vector $\mathbf{v}$ with its initial point at $\mathbf{v}$. Because $\mathbf{v}$ is not parallel to any of the segments comprising $\Delta_{16}$, then $\mathbf{v}$ is pointing from $\mathbf{n}$ into one of the atoms $P_i \in \mathcal{P}_{16}$. Thus, we can place a copy of prototile $T_i$ at $\mathbf{n}$. The other choice is use $-\mathbf{v}$ to determine the choice of tile at each such point $\mathbf{n}$, which points toward the region on the opposite side of the boundary segment as does $\mathbf{v}$. This is illustrated in Figures \ref{fig:Amm_line_1_tiling_points} and \ref{fig:Amm_res}.

\begin{figure}[h]
\centering
\begin{subfigure}[b]{.45\textwidth} 
\centering
\includegraphics[width=.95\textwidth]{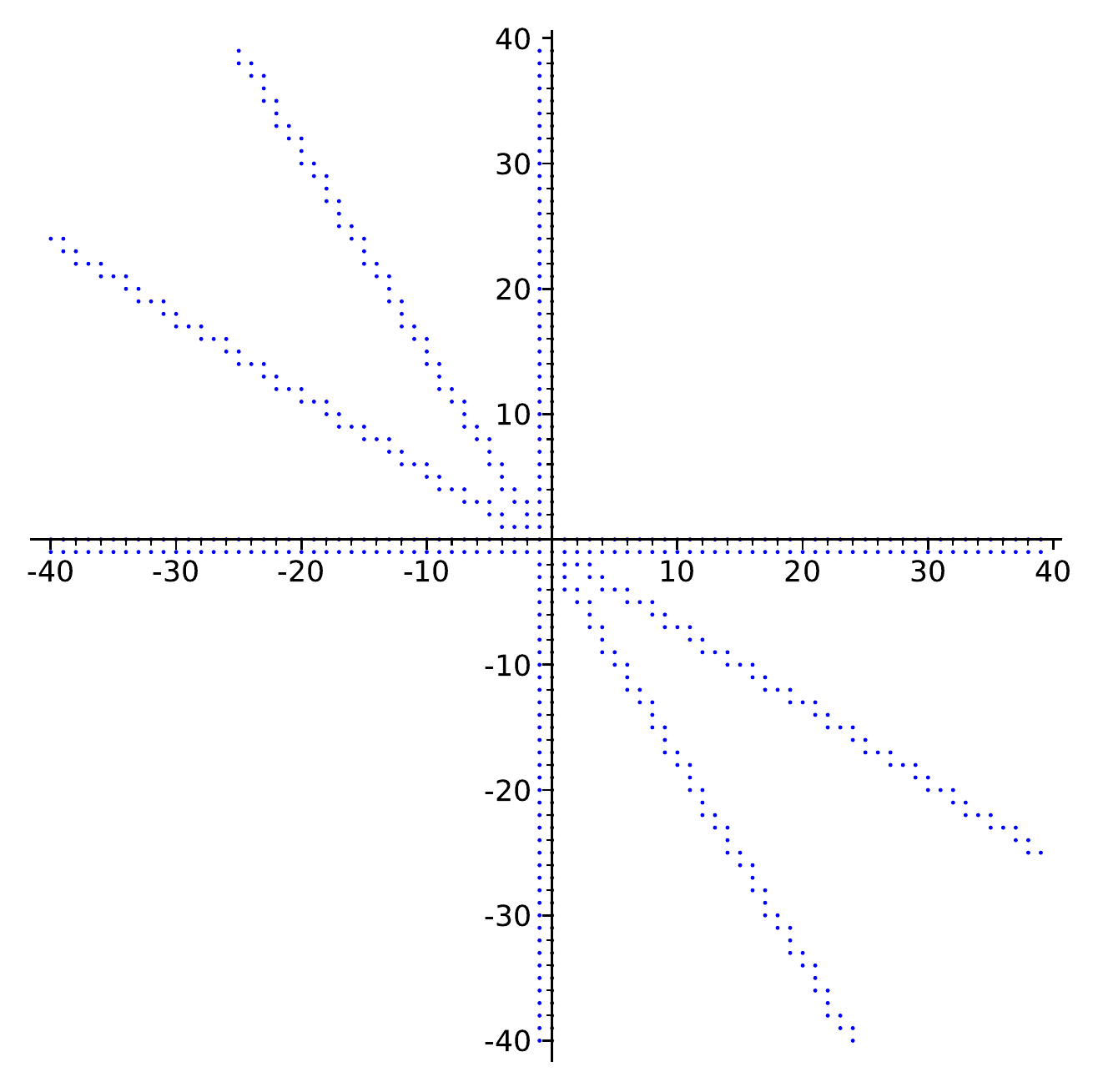}  
\caption{Points $\mathbf{n}$ such that $R^{\mathbf{n}}_{16}(\mathbf{0}) \in \Delta_{16}$}
\label{fig:Amm_0_orbit}
\end{subfigure}
\begin{subfigure}[b]{.45\textwidth} 
\centering
\includegraphics[width=.95\textwidth]{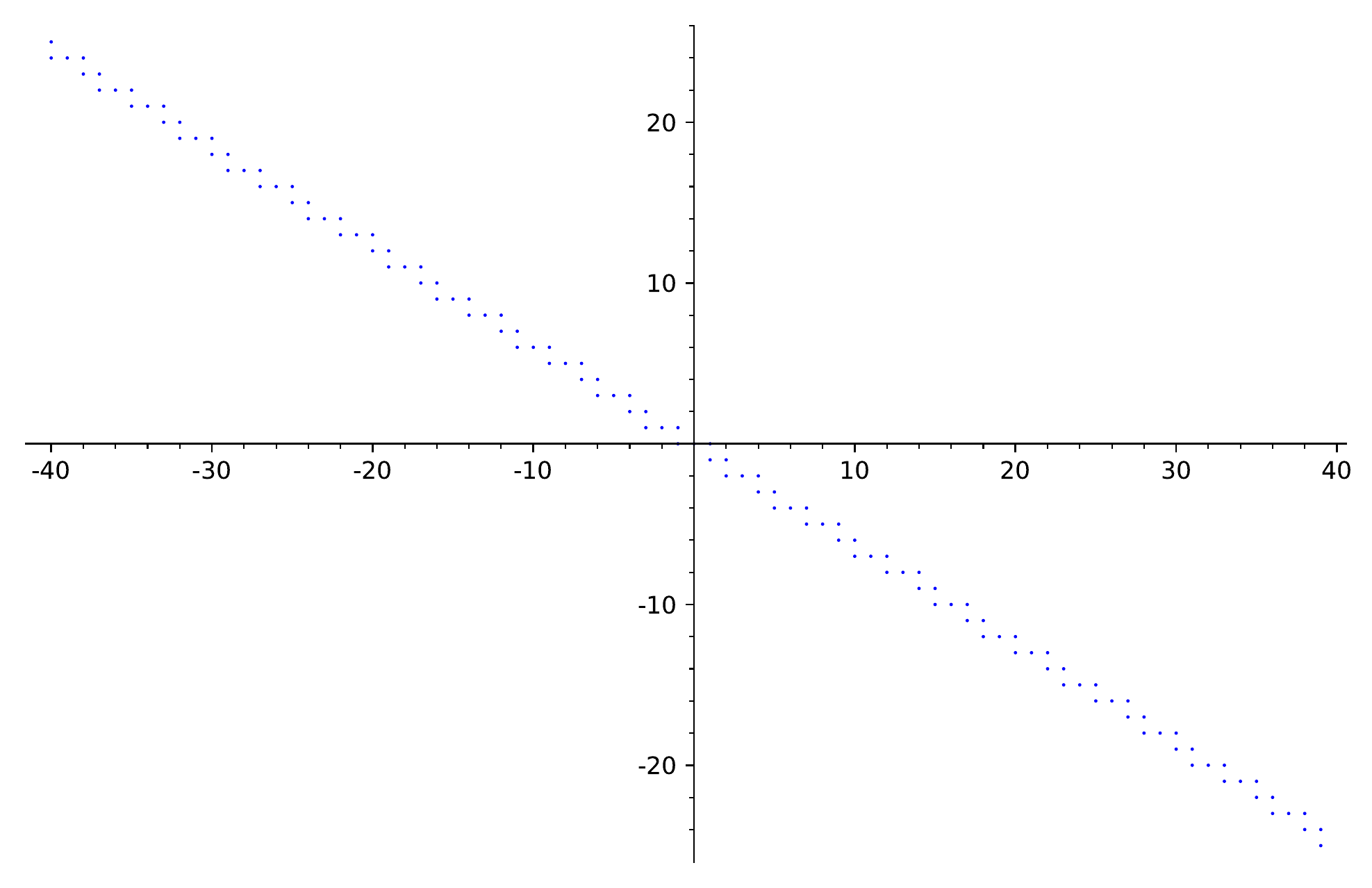}  
\caption{Points $\mathbf{n}$ such that $R^{\mathbf{n}}_{16}(\mathbf{q}) \in \Delta_{16}$, where $q = (1/2,(3/2)\varphi - 1)$.}
\label{fig:Amm_line_1_orbit}
\end{subfigure}
\caption{When orbits of points that intersect $\Delta_{16}$}\label{fig:Amm_orbit_boundary}\end{figure}

\begin{figure}[h]
\begin{center}
\includegraphics[width=0.6\textwidth]{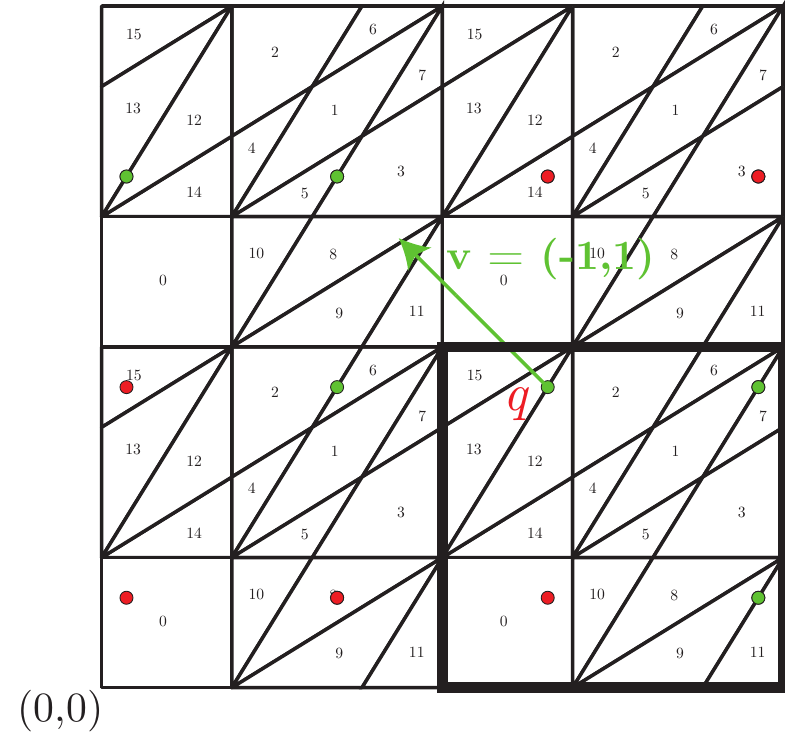}\caption{Using $\mathbf{v}=(-1,1)$ (in green) to resolve the tiles corresponding to points on $\Delta_{16}$. At $q = (1/2,(3/2)\varphi - 1)$, $\mathbf{v}$ points to region $13$ and $-\mathbf{v}$ points to region 12. The point $q$ corresponds to the tiles shaded red in Figure \ref{fig:Amm_res}.} \label{fig:Amm_line_1_tiling_points}
\end{center}
\end{figure}

\begin{figure}[h]
\centering
\begin{subfigure}[b]{.45\textwidth} 
\centering
\includegraphics[width=.95\textwidth]{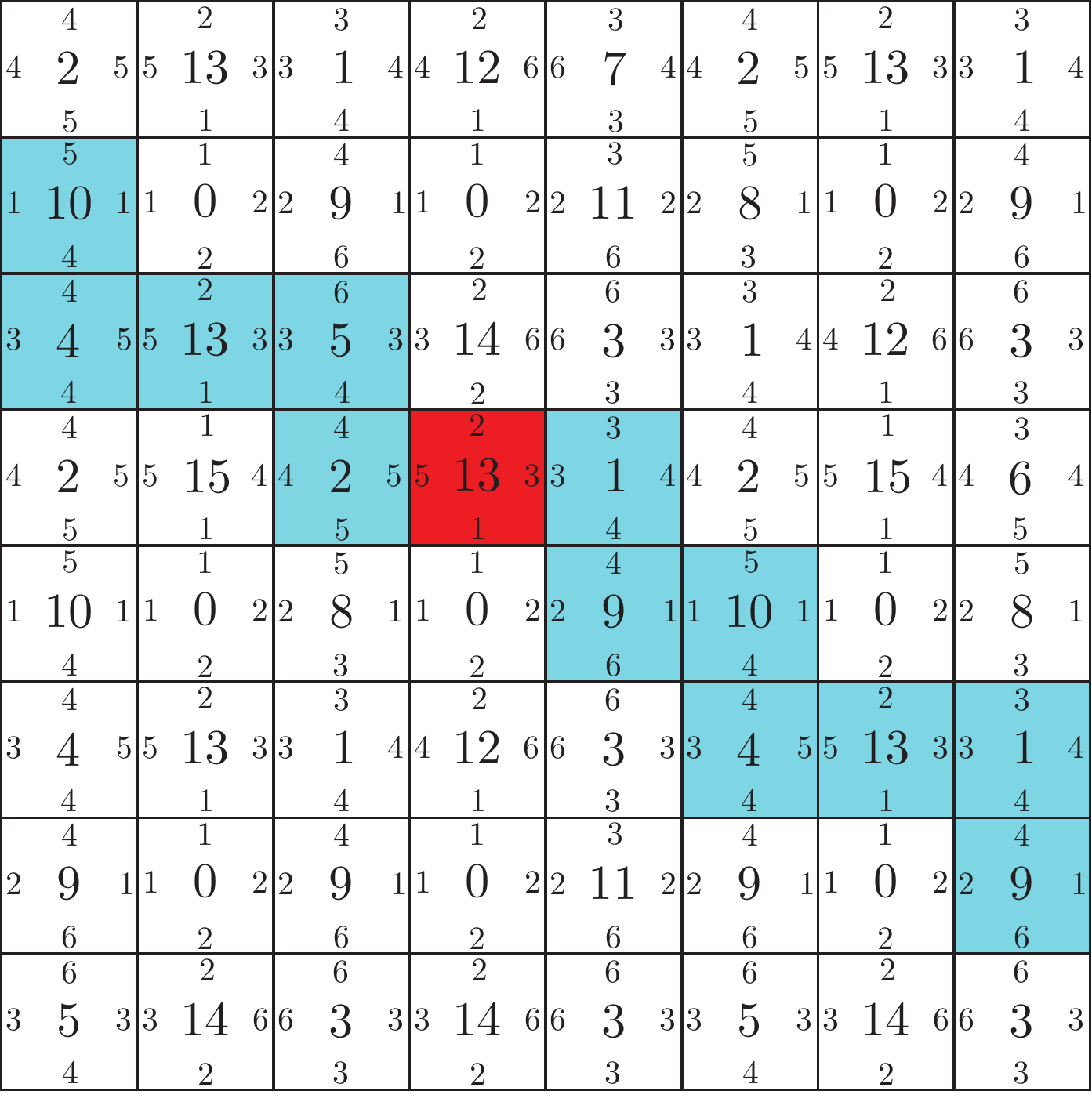}  
\caption{Using $\mathbf{v}=(-1,1)$ to resolve the tiles corresponding to points on $\Delta_{16}$.}
\label{fig:Amm_16_res_1}
\end{subfigure}
\begin{subfigure}[b]{.45\textwidth} 
\centering
\includegraphics[width=.95\textwidth]{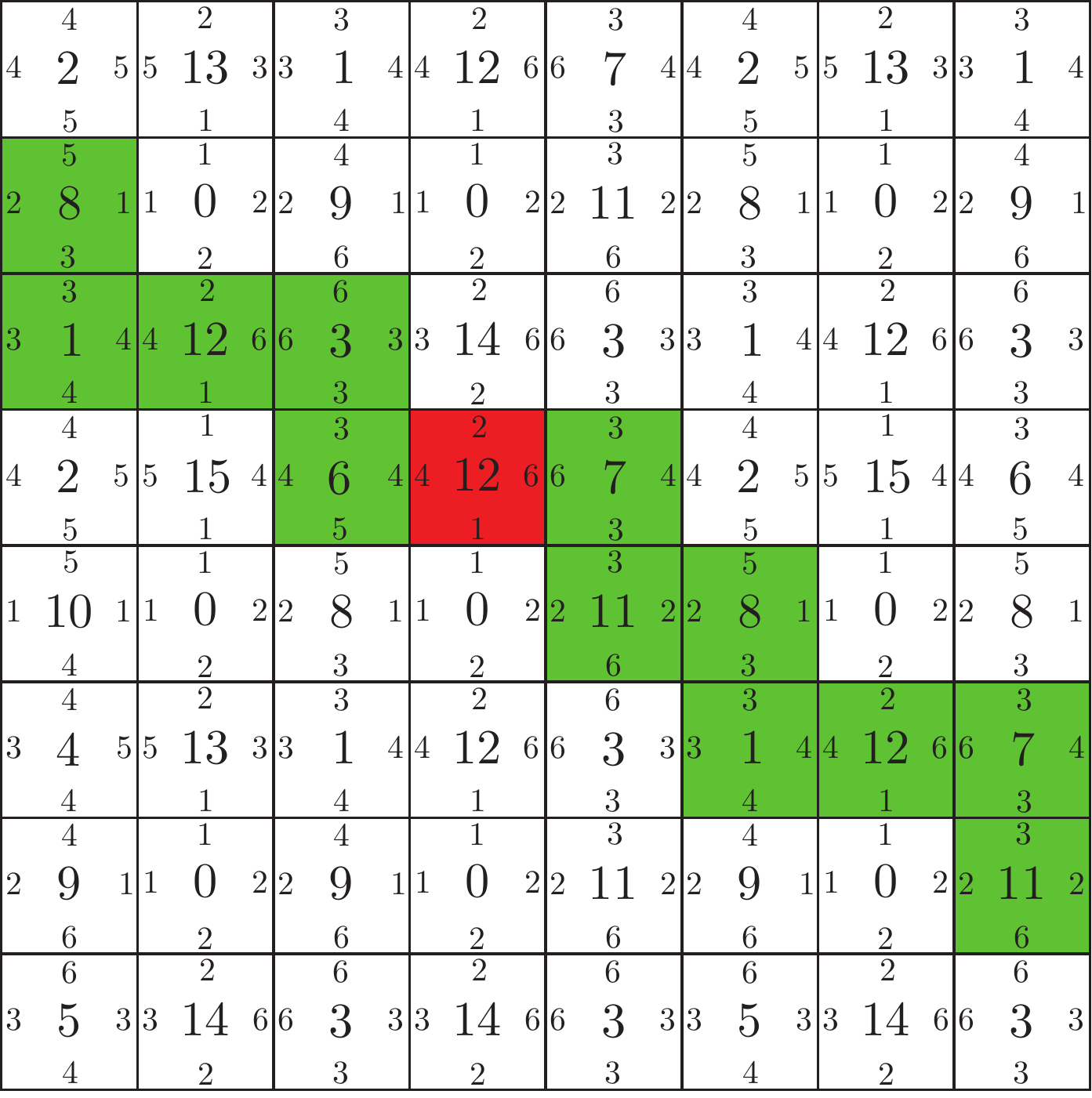}  
\caption{Using $-\mathbf{v}=(1,-1)$ to resolve the tiles corresponding to points on $\Delta_{16}$.}
\label{fig:Amm_16_res_2}
\end{subfigure}
\caption{Two tilings that are the same except at the colored tiles, which correspond to points $\mathbf{n}$ where the orbit of a point intersects $\Delta_{16}$. The tiles in red correspond to the point $q = (1/2,(3/2)\varphi - 1)$ in Figure \ref{fig:Amm_line_1_tiling_points}.}\label{fig:Amm_res}\end{figure}

\section{Candidate Jeandel-Rao Aperiodic Protosets}
In \cite{JR1}, the authors provide 33 candidate order-11 aperiodic protosets; the computer algorithm that they used to find the aperiodic set $\mathcal{T}_0$ could not rule out these 33 candidate sets as ones that admit periodic tilings, so it is likely that they are aperiodic protosets. In any event, as a matter of investigation, we will consider one of these candidate sets as an example. This set, depicted in Figure \ref{fig:JR2}, we will call $\mathcal{T}_2$.

\begin{figure}[h]
\begin{center}
\includegraphics[width=0.9\textwidth]{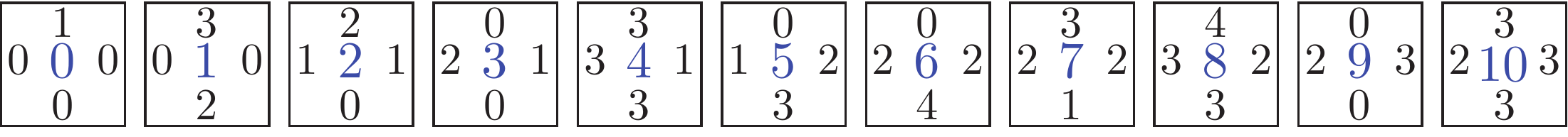}\caption{Jeandel-Rao candidate aperiodic protoset $\mathcal{T}_2$.} \label{fig:JR2}
\end{center}
\end{figure}

We find that the lattice $\Gamma_0$ associated with the original Jeandel-Rao aperiodic protoset does not resolve a coherent dot pattern. However, a variation of $\Gamma_0$ does: Define $\Gamma_2$ to be the lattice in $\R^2$ generated by the vectors $(p,0)$ and $(2-p,p+3)$. We see that $\Gamma_2$ does succeed in resolving a clear dot pattern, shown in Figure \ref{fig:JR2_dots}. Noticing that this dot pattern is very similar to the partition $\mathcal{P}_0$, it is not hard to arrive at a precise partition $\mathcal{P}_2$ for $\mathcal{T}_2$, which is shown in Figure \ref{fig:JR2_partition}. Finally, it appears, at least experimentally, that using the $\Z^2$ action $R_{2}^{\mathbf{n}}(p) = p + \mathbf{n} \pmod{\Gamma_2}$ defined on the torus $\mathbf{T}_2 = [0,\varphi]\times[0,\varphi+3]$ that is endowed with the partition $\mathcal{P}_2$, produces valid tilings (see figure \ref{fig:JR2_tiling}). 
Perhaps an analysis of the dyanamical systems properties of $\mathcal{X}_{\mathcal{P}_2,R_2}$ can be leveraged to argue that the full Wang shift $\Omega_2$ is aperiodic or has other interesting properties.

\begin{figure}[h]
\begin{center}
\includegraphics[width=0.6\textwidth]{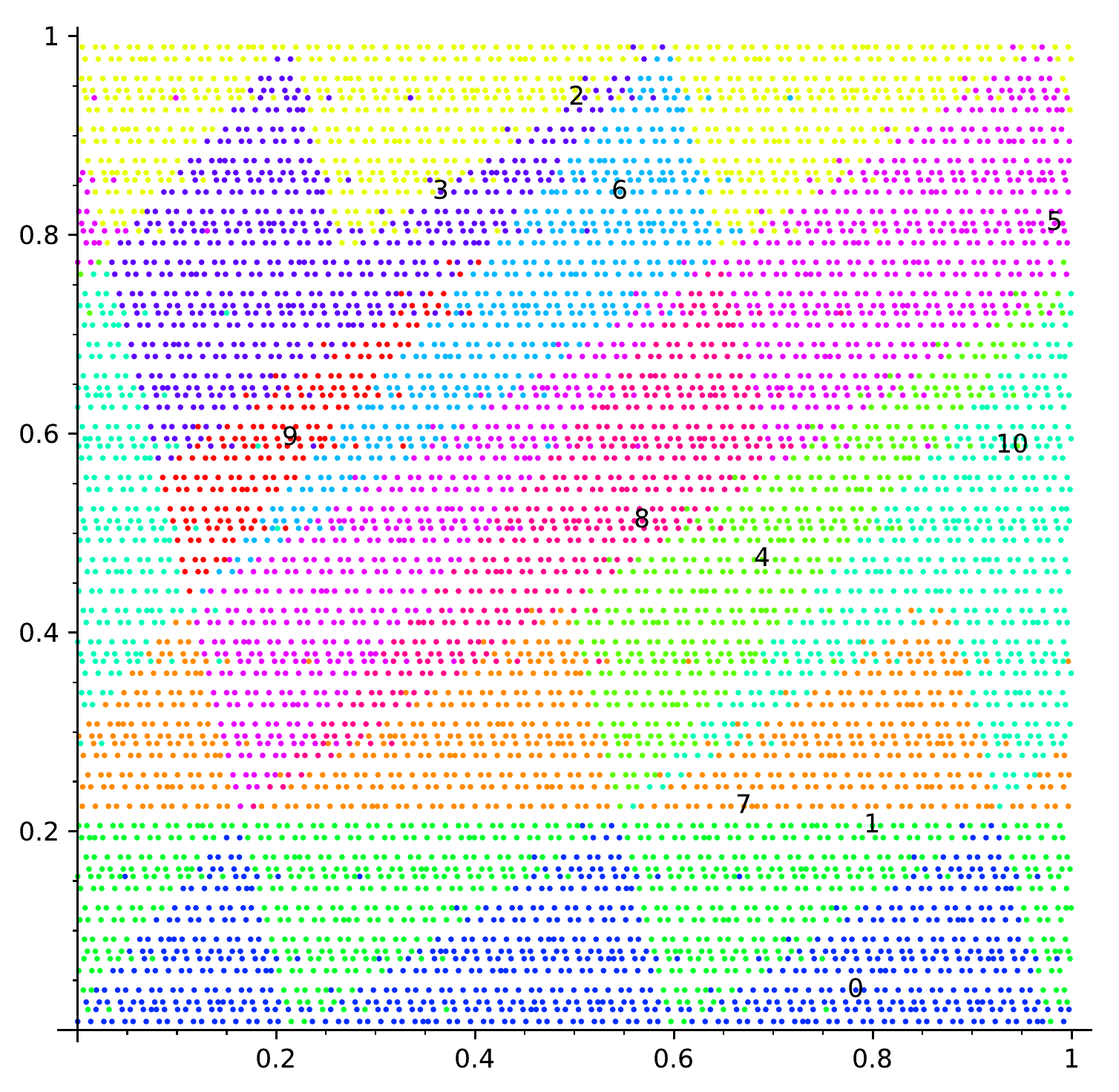}\caption{A dot pattern found by applying matrix $\left( \begin{matrix}
    \varphi & 2-\varphi\\
    0 & \varphi+3
\end{matrix}\right)^{-1}$ to a large patch of tiling by $\mathcal{T}_2$.} \label{fig:JR2_dots}
\end{center}
\end{figure}

\begin{figure}[h]
\begin{center}
\includegraphics[width=0.3\textwidth]{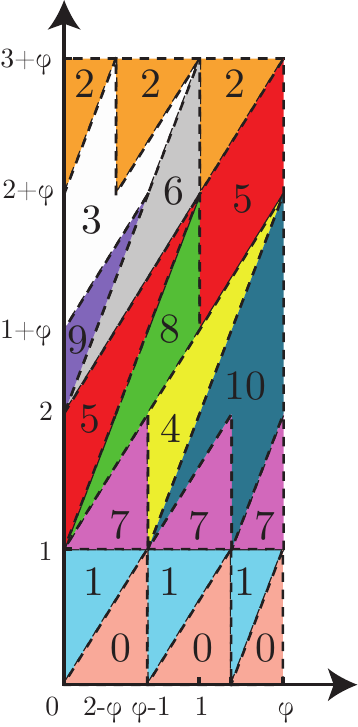}\caption{A partition for $\mathcal{T}_2$.} \label{fig:JR2_partition}
\end{center}
\end{figure}

\begin{figure}[h]
\begin{center}
z\includegraphics[width=0.95\textwidth]{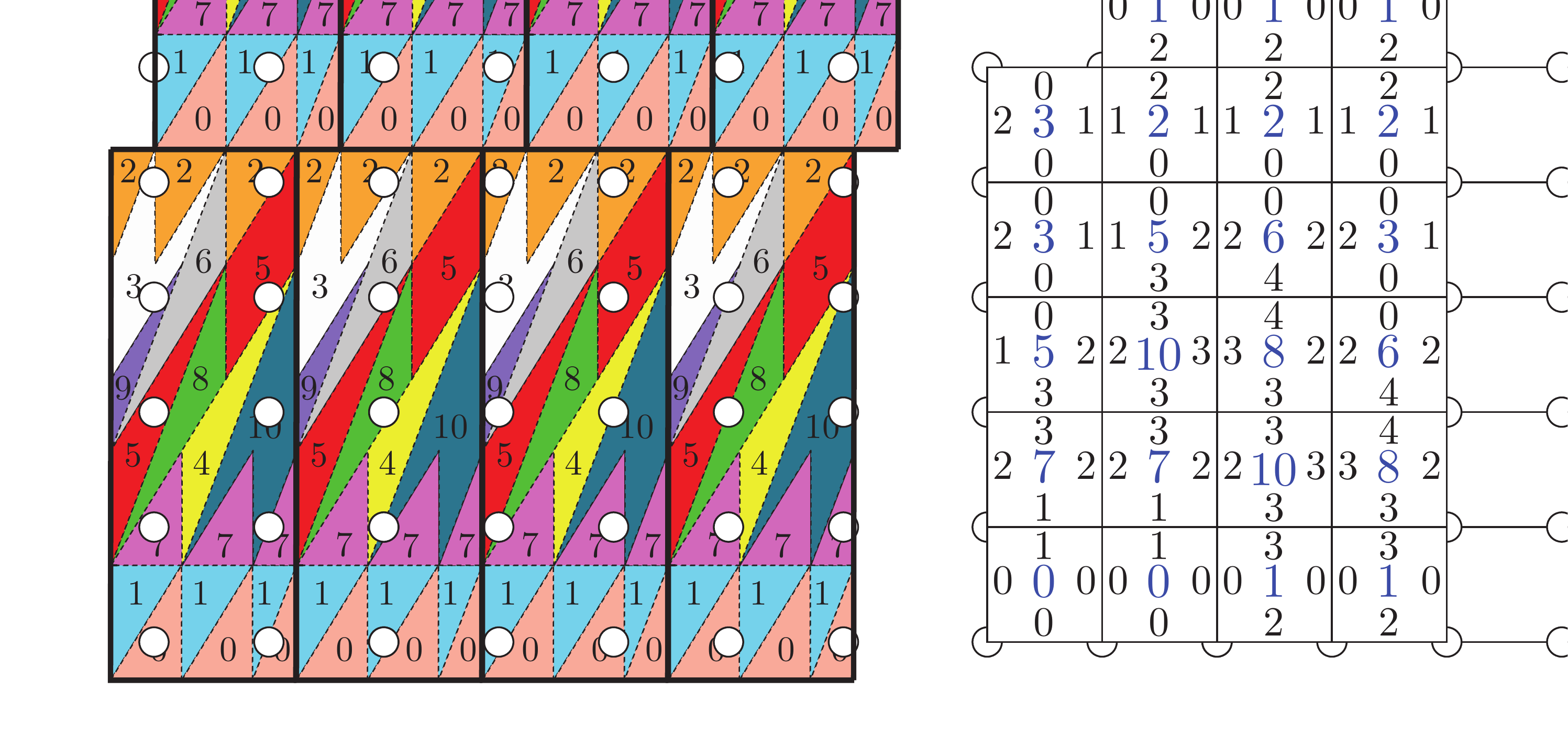}\caption{Orbits of points in $\mathbf{T}_2$ paritioned with $\mathcal{P}_2$ under the action $R_2$ seem to produce valid tilings. The reader is encouraged to produce the next two columns of the tiling at right!} \label{fig:JR2_tiling}
\end{center}
\end{figure}

\section{Identifying the Signature of the Underlying Lattice}
Here we discuss some methodology for determining the lattices for Wang shifts such as $\Omega_0$, $\Omega_{16}$, $\Omega_{24}$ and $\Omega_2$. Let $\mathcal{T}$ be some protoset of Wang tiles, and suppose that there is a hidden encoding of tilings in $\Omega_{\mathcal{T}}$ via a $\Z^2$ action on a partitioned torus, as with the Wang shifts discussed above. Then there is a lattice $\Gamma$ generated by vectors $\gamma_1$ and $\gamma_2$ giving rise to a torus $\mathbf{T} = \R^2 / \Gamma$, along with a $\Z^2$ action $R$ on $\mathbf{T}$ defined by $R^{\mathbf{n}}(x) = x + \mathbf{n} \pmod{\Gamma}$, and there is a topological partition $\mathcal{P} = \{P_0,P_1,\ldots,P_{n-1}\}$ of $\mathcal{T}$ used to encode the orbits of points in $\mathcal{T}$ under $R$ as tilings.

When we look at a tiling in $\Omega_{\mathcal{T}}$, we may notice the repeating patterns within the tiling that suggest ``golden" rotations are responsible for the patterns. Such rotations have some signature behavior, such as having ``near periods" of length approximately equal to Fibonnaci numbers and encoding infinite patterns that are in correspondence with the Fibonacci word. An example of such a golden rotation is a 1-dimensional $\Z$ action (rotation) $R_{\varphi}$ of the unit interval $[0,\varphi+1]$ defined by $R_{\varphi}^{\mathbf{n}}(x) = x + n \pmod{\varphi+1}$. We partition $[0,\varphi+1]$ into two subintervals $P_0 = (0,\varphi)$ and $P_1 = (\varphi,\varphi+1)$. We can use the action $R_{\varphi}$ and the partition $\mathcal{P}=\{P_0,P_1\}$ to encode points in $[0,\varphi]$ as bi-infinite sequences as follows: For each $x \in [0,\varphi+1]$ and $n \in \Z$, if $R^{n}(x) \in P_0 \cup \{0\}$, we place 0 at the $n$-th decimal place of a binary decimal expansion, and if $R^{n}(x) \in P_1 \cup \{\varphi\}$, we place 1 at the $n$-th decimal place. For example, if we encode $0 \in [0,\varphi]$ according to this scheme, we get the bi-infinite binary decimal expansion \[\ldots 01001010010010100101.001001010010010100101\ldots,\] which is the well-known Fibonacci word. The presense of repeating patterns in a tiling along a line that follow the pattern of a Fibonacci word suggests that the toral rotation in the direction of that line is something akin to $R_{\varphi}$. 

Another signature is the spacing of repeating structures. Going back to $R_{\varphi}$, consider the orbit $\mathcal{O}_{R_{\varphi}}(x)$ of some point $x \in [0,\varphi+1]$. It is not hard to check that $\mathcal{O}_{R_{\varphi}}(x)$ is dense in $[0,\varphi+1]$. Thus, there are many values of $n \in \Z$ such that $R^{n}(x) \approx x$. For what values of $n$ is $R^{n}(x) \approx x$? It turns out that when $|n|$ is a Fibonacci number, we see that $R^{n}(x) \approx x$. This is a result of a property of the Fibonacci sequence $(F_n) = (1,1,2,3,5,8,\ldots)$ which has the well-known property that $F_{n+1}/F_{n} \rightarrow \varphi$ as $n \rightarrow \infty$. From this limit, we see that a Fibonacci number (which is integer valued) is approximately a multiple of $\varphi$. Thus, we see that $R^{n}(x) = x + n \pmod{\varphi+1} \approx x$ whenever $|n|$ is Fibonacci. The upshot here is that we expect to see repeating patterns within the encoding of $x$ with spacing between these patterns being Fibonacci numbers.

Slightly generalizing the rotation $R_{\varphi}$, we may consider rotations $R_{(a,b)}$ of the interval $[0,a\varphi + b]$ ($a,b\in \Z$) defined by $R_{(a,b)}(x) = x + n \pmod{a\varphi + b}$. Whatever partition we place on $[0,a\varphi+b]$, in the encoding of $x \in [0,a\varphi + b]$ under $R_{(a,b)}$, we expect to see repeating patterns with spacing between these patterns being Fibonacci numbers because $R_{(a,b)}^{n}(x) = x + n \pmod{a\varphi + b} \approx x$ when $n \in \Z$ is approximately a multiple of $a\varphi + b$, which again happens when $n$ is a multiple of $a \varphi$, which happens when $n$ is a Fibonacci number. Thus, seeing a repeating pattern with gaps between being the size of Fibonacci numbers is suggestive of a rotation of the form $R_{(a,b)}$. 
  
Now let us consider 2-dimensional $\Z^2$ toral rotations of a certain form: Suppose $R$ is defined on a torus $\mathbf{T}$ by $R^{\mathbf{n}}(x) = x + \mathbf{n} \pmod{\Gamma}$ where $x \in \mathbf{T}$, $\mathbf{n} \in \Z^2$, and $\gamma_1 = (a\varphi+b,c\varphi+d)$ and $\gamma_2 = (e\varphi + f,g\varphi+h)$ ($a,b,c,d,e,f,g,h \in \Z$) are the generating vectors of the lattice $\Gamma$. If $\gamma_1$ is horizontal, then when $R^{\mathbf{n}}$ is restricted to those values of $\mathbf{n} = (n_1,n_2)$ where $n_2$ is fixed, $R^{\mathbf{n}}$ will behave like a golden rotation in the horizontal direction. Thus, to pick up the signature of horizontal golden rotation in a tiling, we look for repeating horizontal patterns in the tiling that have spacing between these patterns being Fibonacci numbers. A similar thing could happen in the vertical direction or any slant direction; in a slant direction, we would look for a repeating pattern with spacing that is a Fibonacci number multiple of a vector parallel to that direction.

While spotting the tell-tale behavior of a golden rotation in a direction inside a tiling is not too difficult to do, finding this direction is of the form $(a\varphi+b, c\varphi + d)$ and finding the specific values of $a,b,c,$ and $d$ are two different questions. We have not yet discovered a way to pinpoint these values from properties of the tiling, but it is not too hard to have the computer test values of $a,b,c,$ and $d$ over a range of integers, say where each variable ranges over integers between -3 and 3, and see if any of those values produce dot patters indicating clear partitions. And this is exactly what one can do to find dot patterns for the partitions $\mathcal{P}_{0}$, $\mathcal{P}_{2}$, $\mathcal{P}_{16}$, and $\mathcal{P}_{24}$. 

As an example, consider the computer generated partial tiling admitted by $\mathcal{T}_{16}$ (the Wang tile version of the Ammann A2 protoset) shown in Figure \ref{fig:Amm_16_marked_tiling}. Tile 15 has been colored to make it stand out in contrast to the other tiles of the tiling. One notices the signatures of golden rotations within this tiling in both the horizontal and vertical directions - See Figure \ref{fig:Amm_16_marked_tiling}. This is suggestive of golden toral rotations in vertical and horizontal directions of the form $(a\varphi+b, 0)$ and $(0,c\varphi + d)$, and luckily, the simplest vectors $\gamma_1 = (\varphi,0)$ and $\gamma_2 = (0,\varphi)$ produce a nicely resolved dot pattern from which we determined the parittion $\mathcal{P}_{16}$ for $\mathcal{T}_{16}$.

\begin{figure}[h]
\begin{center}
\includegraphics[width=0.8\textwidth]{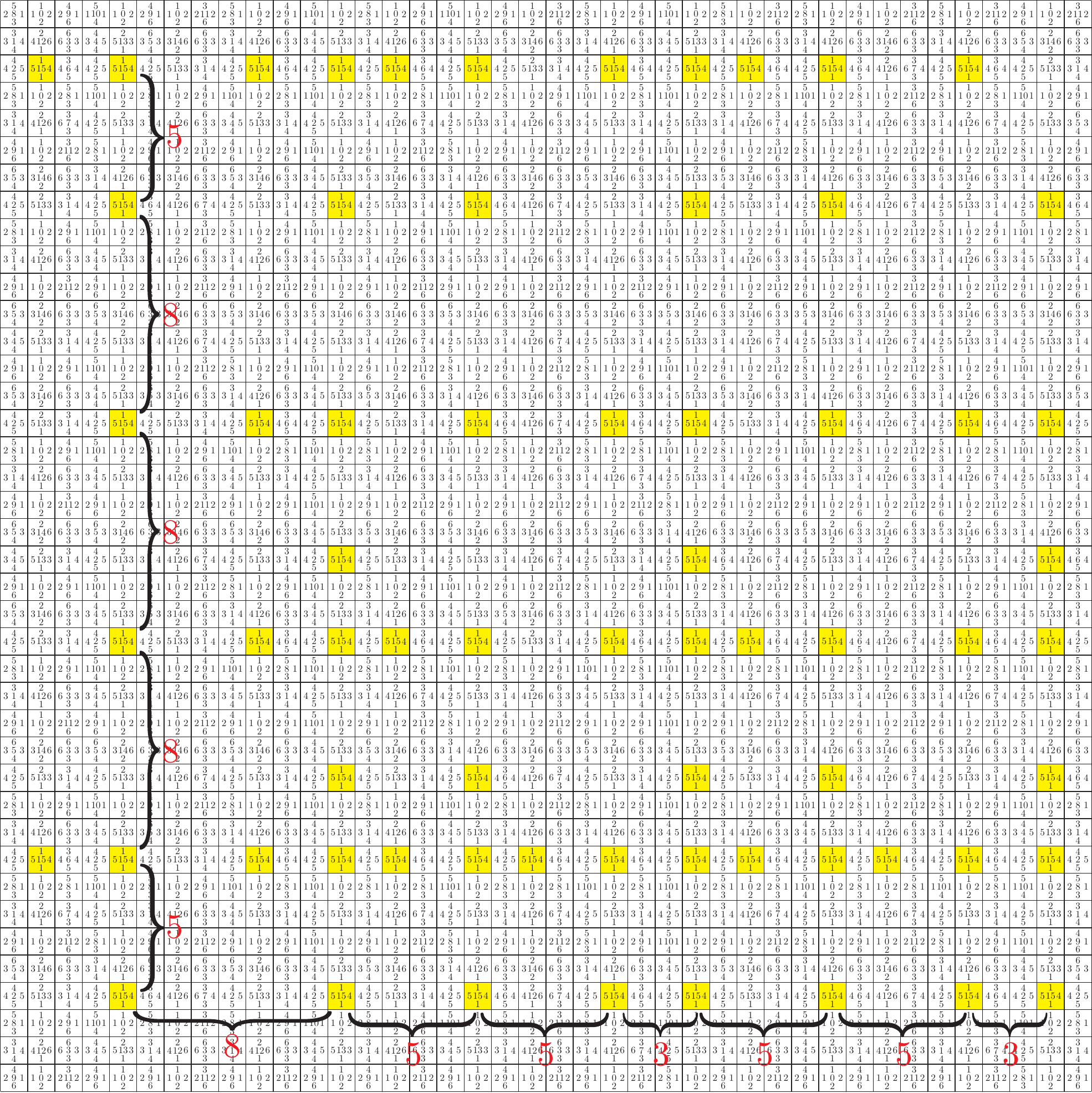}\caption{A partial tiling by $\mathcal{T}_{16}$ exhibiting evidence of golden toral rotations. Tile $T_{15}$ is colored yellow.} \label{fig:Amm_16_marked_tiling}
\end{center}
\end{figure}

Other tilings exhibiting golden rotations do not give up their secret lattices quite as easily. For example, in Figure \ref{fig:JR2_marked_tiling}, we give a portion of a tiling admitted by the Jeandel-Rao candidate aperiodic protoset $\mathcal{T}_2$. Using the computer to find clearly resolved dot patterns turned up a few essentially equivalent possibilities, one of which is depicted in Figure \ref{fig:JR2_dots}. The golden rotational behavior (Fibonacci spacing) corresponding to the two vectors generating that dot pattern is illustrated in Figure \ref{fig:JR2_marked_tiling}.

\begin{figure}[h]
\begin{center}
\includegraphics[width=0.8\textwidth]{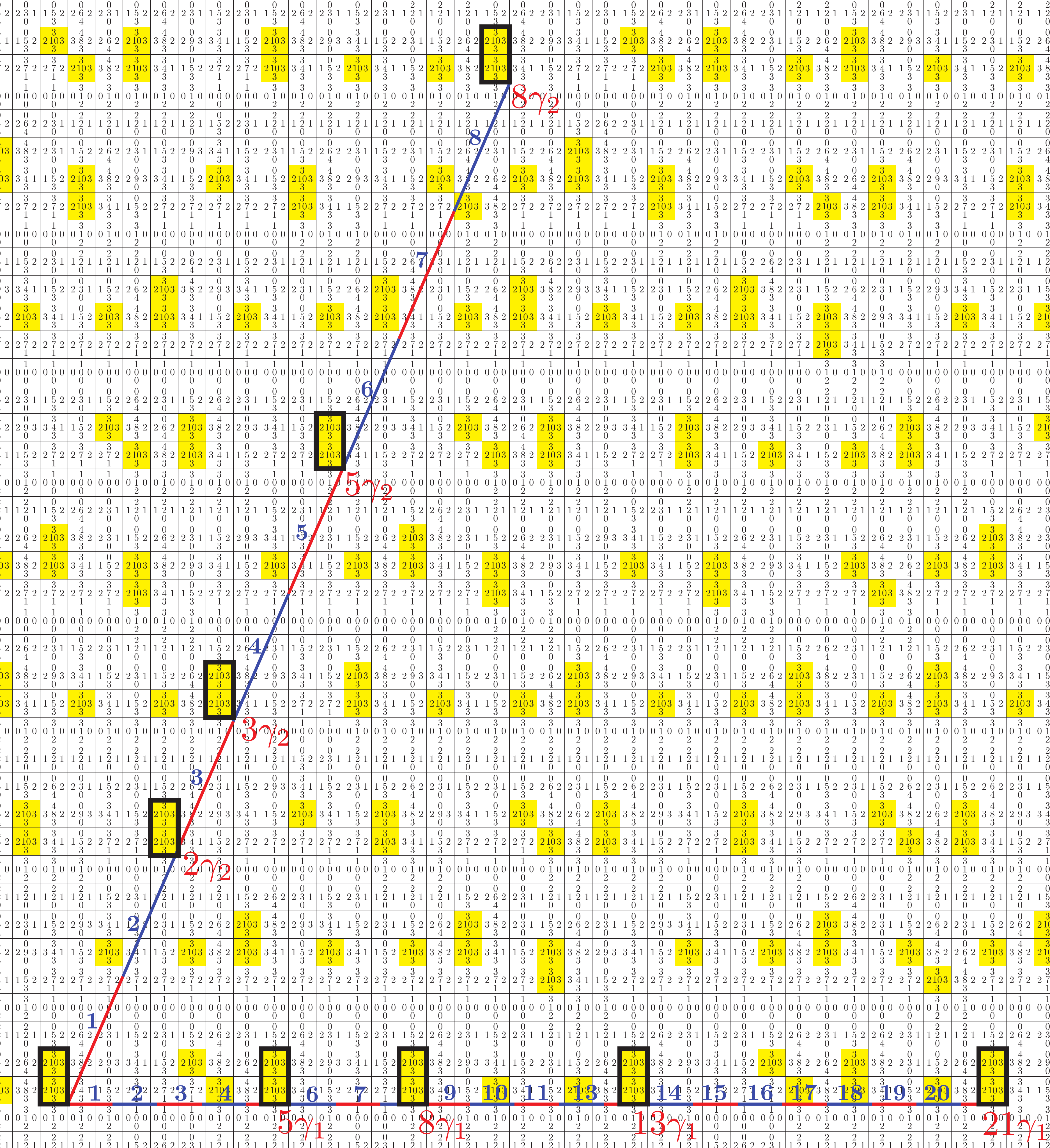}\caption{A partial tiling by $\mathcal{T}_{2}$ exhibiting evidence of golden toral rotations. $\gamma_1$ and $\gamma_2$ are vectors generating a lattice that resolves the tiling to a clear dot pattern as seen in Figure \ref{fig:A2_dots}. Tile $T_{10}$ is colored yellow.} \label{fig:JR2_marked_tiling}d
\end{center}
\end{figure}

We conclude this section by giving a snippet of code that show how we found the lattice vectors for $\mathcal{T}_2$ that produced a clearly resolved dot pattern, revealing the underlying partition.

\begin{lstlisting}
from slabbe import WangTileSet
import itertools

#Creating the Wang tiles
T2_tiles = [(0,1,0,0),(0,3,0,2),(1,2,1,0),(1,0,2,0),(1,3,3,3),(2,0,1,3),(2,0,2,4),(2,3,2,1),(2,4,3,3),(3,0,2,0),(3,3,2,3)]
T2 = WangTileSet(T2_tiles)

#Generating a large patch of tiling with bottom row all copies of tile 6
 preassigned_tiles = {(a,0):6 for a in range(100)};
S2 = T2.solver(100,62,preassigned_tiles=preassigned_tiles);
T2_tiling = S2.solve(solver='glucose');

#Coloring all copies of tile 10 yellow in the tiling
positions = T2_tiling.tile_positions([10])
extra = []
    for (x,y) in positions:
        extra.append(r'\draw[fill=yellow]({},{})rectangle({},{});'.format(x,y,x+1,y+1))
extra = '\n'.join(extra)

#Displaying the tiling with and without the colored copies of tile 10
T2s.tikz()
T2_tiling.tikz(extra_before=extra)

#Testing a range of lattice generators (columns of M)
#to see if they resolve a clear dot pattern for a partition
for P in itertools.product([-1,0,1],[0,1,2],[0,1,2],[0,1,2],[0,1,2,3]):
M = matrix.column([[p+P[0],0],[P[1]*p+P[2],P[3]*p+P[4]]]);
if M.determinant() != 0:
    print(P)
    T2_tiling.plot_points_on_torus(M.inverse()).show()
\end{lstlisting}

\section*{Acknowledgements}
The authors would like to thank S\'{e}bastien Labb\'{e} for his helpful suggestions.

\begin{appendices}
\section{Dynamical Systems and Wang Tilings}\label{apx:dyn_sys}

Here we will dive a little deeper into the Labb\'{e}'s theory, given more definitions and citing several theorems, then applying it to prove some facts about $\mathcal{X}_{\mathcal{P}_{16},R_{16}}$. First, we provide some additional details and definitions, following the ideas laid out in \cite{Labb2021}.

Recall from Subsection \ref{subsec:dyn_sys_intro} that a topological dynamical system is a triple $(X,G,S)$ where $X$ is a topological space, $G$ is a topological group, and $S$ is a continuous function $S \!:\! G \times X \rightarrow X$ which defines a left action of $G$ on $X$. For $A \subseteq X$, we will use $\overline {A}$ to indicate the \emph{\textbf{topological closure}} of $A$ in $X$. We also define another kind of closure - the \emph{\textbf{$S$-closure}} - by $S(A) = \cup_{g \in G}S^{g}(A)$, and we will say that $A$ is \emph{\textbf{$S$-invariant}} if $S(A) = A$. We say the system $(X,G,S)$ is \emph{\textbf{minimal}} if $X$ contains no proper, nonempty, (topologically) closed, $S$-invariant subsets.

Recall from Subsection \ref{subsec:dyn_sys_intro} that the configuration space, $\mathcal{T}^{\Z^2}$, where $\mathcal{T}$ is a Wang tile protoset, is thought of as the space of all possible tilings (valid and nonvalid) of the plane by $\mathcal{T}$. The configuration space $\mathcal{T}^{\Z^2}$ is a topological space when endowed with the compact product topology, and moreover, this topology is generated by the metric $d$ on $\mathcal{T}^{\Z^2}$ defined by $d(x,x') = 1/2^{\min\{\|\mathbf{n}\| \in \Z^2:x_{\mathbf{n}} \neq x_{\mathbf{n}}\}}$. In the formula for this metric, $\min\{\|\mathbf{n}\| \in \Z^2:x_{\mathbf{n}} \neq x_{\mathbf{n}}\}$ is a value of $\mathbf{n}$ closest to the origin where $x$ and $x'$ do not agree, so $d(x,x')$ is close to 0 when $x$ and $x'$ agree on a large patch centered at the origin, and $d(x,x')$ is close to 1 when $x$ and $x'$ disagree near the origin. $(\mathcal{T}^{\Z^2},\Z^2, \sigma)$, where $\sigma$ is the shift action, is a topological dynamical system, and in this particular setting, if $X \subseteq \mathcal{T}^{\Z^2}$ is $\sigma$-invariant, we say $X$ is \emph{\textbf{shift-invariant}}. For a space of Wang tilings $X$ (say, a space of Wang tilings admitted by a protoset, or a space of tilings having some defining property), shift-invariance is clearly a desirable property, for if $x' = \sigma(x)$ where $x \in X$, that just means $x'$ is a translation of $x$, so $x$ and $x'$ are really the same tilings, and thus, we would want $X$ to contain both. If $X$ is topologically closed and shift invariant in $\mathcal{T}^{\Z^2}$, then we call $X$ a \emph{\textbf{subshift}} of $\mathcal{T}^{\Z^2}$. Notice that $(X,\Z^2,\sigma)$ is a topological dynamical system when $X$ is a subshift of $\mathcal{T}^{\Z^2}$.

\subsection{Shifts of Finite Type}\label{subsec:SFT}
Of particular importance in the context of tilings are \emph{shifts of finite type}. A subshift of Wang tilings $X$ being a shift of finite type essentially means, informally, that the tilings in $X$ are completely determined a set of forbidden arrangements of tiles, such as when two tiles cannot be placed adjacent to one another due to a mismatch in the edge matching rules. Formally, to define shift of finite type, for any finite subset $S \subseteq \mathcal{T}^{\Z^2}$, we define $\mathcal{T}^{S} = \{y\!:\! S \rightarrow \mathcal{T}$\}, and we call any element $p \in \mathcal{T}^S$ a \emph{\textbf{pattern}}. Next, we define the projection map $\pi_{S}\!:\!\mathcal{T}^{\Z^2} \rightarrow \mathcal{T}^{S}$ which restricts any configuration $x \in \mathcal{T}^{\Z^2}$ to $S$ (i.e. $\pi_{S}(x) = x|_{S}$). We say $X \subseteq \mathcal{T}^{\Z^2}$ is a \textbf{\emph{shift of finte type (SFT)}} if there exists a finite set of patterns $\mathcal{F}$ called \textbf{\emph{forbidden patterns}} such that \[X = \{x \in \mathcal{T}^{\Z^2} | \pi_S \circ \sigma^{\mathbf{n}}(x) \notin \mathcal{F} \text{ } \forall \mathbf{n} \in \Z^2, S \subset \Z^2\}.\] It is clear that the full Wang shift $\Omega_{\mathcal{T}}$ of a protoset of Wang tiles $\mathcal{T}$ is an SFT - the set of finite set forbidden patterns is formed simply as a set of representative invalid horizontal and vertical pairings of tiles from the protoset. However, for a subshift $X$ of $\Omega_{\mathcal{T}}$, such as $\mathcal{X}_{\mathcal{P},R}$ formed from a partition, it may take some work to argue that $X$ is an SFT, for it is possible that $X$ has forbidden patterns in addition to the invalid horizontal and vertical pairings of prototiles.

\subsection{Symbolic Representations and Markov Partitions}
Here will will give precise definition to the term \emph{Markov partition}. Suppose that $(M,\Z^2,R)$ is a topological dynamical system where $M$ is a compact topological space. Recall that a \emph{topological partition} of $M$ is a collection $\mathcal{P} = \{P_0, P_1, \ldots, P_{n-1}\}$ of disjoint open sets such that $M = \overline{P_0} \cup \overline{P_1} \cup \cdots \cup \overline{P_{n-1}}$. Let $\mathcal{A} = \{0,1,\ldots,r-1\}$. For any finite set $S \subset \Z^2$, let $\mathcal{A}^{S}
= \{x \!:\! S \rightarrow \mathcal{A}\}$. A \emph{pattern} $x \in \mathcal{A}^S$ is \emph{allowed} for $\mathcal{P},R$ if 
\[
   \bigcap_{\mathbf{k}\in S}R^{-\mathbf{k}}(P_{w_{\mathbf{k}}}) \neq \emptyset.
\] 

We denote the set of all allowed patterns for $\mathcal{P}$ and $R$ by  $\mathcal{L}_{\mathcal{P},R}$. Labb\'{e} points out in \cite{Labb2021} that $\mathcal{L}_{\mathcal{P},R}$ is the language of a subshift of $\mathcal{X}_{\mathcal{P},R} \subseteq \mathcal{A}^{\Z^2}$ and leads to the following definition.

\begin{definition} \label{dfn:symb_dyn_sys} Let \[
    \mathcal{X}_{\mathcal{P},R} = 
    \{x\in\Acal^{\Z^2} \mid \pi_S\circ\sigma^\bn(x)\in\Lcal_{\mathcal{P},R}
    \text{ for every } \bn\in\Z^2 \text{ and finite subset } S\subset\Z^2\}.\] We call $\mathcal{X}_{\mathcal{P},R}$ the \textbf{symbolic dynamical system} corresponding to $\mathcal{P},R$.
\end{definition} For $w \in \mathcal{X}_{\mathcal{P},R} \subset \mathcal{A}^{\Z^2}$ and $n \geq 0$,  define  
\[D_n(w) = \bigcap_{\|\mathbf{k}\| \leq n} R^{-\mathbf{k}}(P_{w_{\mathbf{k}}}) \subseteq M.\] We see that $\overline{D}_n(w)$ is compact and $\overline{D_n}(w) \supseteq \overline{D_{n+1}}(w)$ for each $n \geq 0$, and consequently, $\cap_{n = 0}^{\infty} \overline{D_n}(w) \neq \emptyset$. 

The hope is that configurations $x \in \mathcal{X}_{\mathcal{P},R}$ are generated uniquely as orbits of points in $M$, and the above intersection being a single point guarantees this:

\begin{definition} A topological partition $\mathcal{P}$ of $M$ gives a \textbf{\emph{symbolic representation}} of $(M, \Z^2,R)$ if for every $w \in \mathcal{X}_{\mathcal{P},R}$, the intersection $\cap_{n = 0}^{\infty}\overline{D_n}(w)$ consists of exactly one point $m \in M$. The configuration $w \in \mathcal{X}_{\mathcal{P},R}$ is called a \textbf{\emph{symbolic representation}} of $m \in M$.  \label{dfn:symbolic_rep}\end{definition}

Labb\'{e} provides the following alternative criterion for determining when a partition gives a symbolic representation. 

\begin{lemma}[\cite{Labb2021}]\label{lem:LabbeMinimality}Let $(M,\Z^2,R)$ be a minimal dynamical system and $\mathcal{P} = \{P_0,P_1,\ldots,P_{r-1}\}$ be a topological partition of $M$. If there exists an atom $P_i$ which is invariant only under the trivial translation in $M$, then $\mathcal{P}$ gives a symbolic representation of $(M,\Z^2,R)$.\end{lemma}

The following definition is taken from \cite{Labb2021}.

\begin{definition}\label{def:Markov} A topological partition $\mathcal{P}$ of $M$ is a \textbf{Markov partition} for $(M,\Z^2,R)$ if 
\begin{itemize}
    \item $\mathcal{P}$ gives a symbolic representation of $(M,\Z^2,R)$ and
    \item $\mathcal{X}_{\mathcal{P},R}$ is a shift of finite type (SFT).
\end{itemize}\end{definition}

\subsection{Conjugacy and Factor Maps}
With $(M,\Z^2,R)$ and $\mathcal{P}$ encoding a space of configurations $(\mathcal{X}_{\mathcal{P},R},\Z^2,\sigma)$, it is natural to consider morphisms from $(M,\Z^2,R)$ to $(\mathcal{X}_{\mathcal{P},R},\Z^2,\sigma)$ and vice versa. Let $(X,G,S)$ and $(Y,G,T)$ be topological dynamical systems having the same group G. A \textbf{\emph{homomorphism}} of dynamical systems is a continuous map $f \!:\! X \rightarrow Y$ such that $T^g \circ f = f \circ S^g$ for all $g \in G$. If $f$ is a homomorphism, we call $f$ an \textbf{\emph{embedding}} if it is one-to-one; we call $f$ a \textbf{\emph{factor map}} if it is onto, in which case we call $Y$ a \textbf{\emph{factor}} of $X$ and $X$ an \textbf{\emph{extension}} of $Y$; and we call $f$  a \textbf{\emph{topological conjugacy}} if it is bijective with continuous inverse, in which case we say $(X,G,S)$ and $(Y,G,T)$ are \textbf{\emph{topologically conjugate}}.

\subsection{From Symbolic  Representations to Wang Shifts}\label{sec:symbrep}

Here we borrow more theory from \cite{Labb2021}. The idea that we want to make precise is that when $\mathcal{P}$ gives a symbolic representation of $(M,\Z^2,R)$, there is a protoset $\mathcal{T}$ such that $\mathcal{X}_{\mathcal{P},R} \subset \Omega_{T}$; if $\mathcal{T}$ happens to be the protoset of interest (i.e., the protoset you used to generate the dot pattern which in turn you used to form $\mathcal{P}$), then we have necessary result that orbits of points under the action of $R$ in $M$ and the partition $\mathcal{P}$ encode valid tilings.

Suppose $\mathcal{P} = \{P_a\}_{a \in \mathcal{A}}$ is a topological partition of the 2-dimensional torus $\mathbf{T}$, and suppose that $\mathcal{P}$ gives a symbolic representation of $(\mathbf{T},\Z^2,R)$. The \textbf{boundary} of $\mathcal{P}$ is the set \[\Delta = \bigcup_{a \in \mathcal{A}}\partial P_a,\] and \[ \Delta_{\mathcal{P},R} = \bigcup_{\mathbf{n}\in \Z^2} R^{\mathbf{n}}(\Delta) \subset \mathbf{T}. \] is the set of points in $\mathbf{T}$ whose orbits intersect $\Delta$. It can be seen that $\mathbf{T} \setminus \Delta_{\mathcal{P},R}$ is dense in $\mathbf{T}$. We will assume that $\Delta$ consists of straight line segments.

We get the following from \cite{Labb2021}:

\begin{proposition}\label{prop:factor-map}{\rm\cite[Prop.~5.1]{Labb2021}}
    Let $\mathcal{P}$ give a symbolic representation of the dynamical system $(\mathbf{T},\Z^2,R)$ such that $R$ is a $\Z^2$-rotation on $\mathbf{T}$. Let $f:\mathcal{X}_{\mathcal{P},R}\to \mathbf{T}$ be defined such that $f(w)$ is the unique point in the intersection $\cap_{n=0}^{\infty}\overline{D}_n(w)$. The map $f$ is a factor map from $(\mathcal{X}_{\mathcal{P},R},\Z^2,\sigma)$ to $(\mathbf{T},\Z^2,R)$ such that $R^{\boldsymbol{k}}\circ f = f\circ\sigma^{\boldsymbol{k}}$ for every $\boldsymbol{k}\in\Z^2$. The map $f$ is  one-to-one on $f^{-1}(\mathbf{T}\setminus\Delta_{\mathcal{P},R})$. \end{proposition}

Following Section 4 of \cite{Labb2021}, for each $\boldsymbol{x} \in \mathbf{T} \setminus \Delta_{\mathcal{P},R}$, we define a map $\textsc{Cfg}_{\boldsymbol{x}}^{\mathcal{P},R} \,:\, \Z^2 \rightarrow \mathcal{A}$ by \[\textsc{Cfg}_{\boldsymbol{x}}^{\mathcal{P},R}(\mathbf{n}) = a \text{\hspace{.2in}if and only if \hspace{.2in} } R^{\mathbf{n}}(\boldsymbol{x}) \in P_a\] (``$\textsc{Cfg}$" for ``configuration"). $\textsc{Cfg}_{\boldsymbol{x}}^{\mathcal{P},R}$ gives rise to a map $\textsc{SR}\,:\, \mathbf{T} \setminus \Delta_{\mathcal{P},R} \rightarrow \mathcal{A}^{\Z^2}$ defined by \[\textsc{SR}(\boldsymbol{x}) = \textsc{Cfg}_{\boldsymbol{x}}^{\mathcal{P},R}\] (``$\textsc{SR}$"
 for ``symbolic representation"). Next, we wish to extend $\textsc{SR}$ to a map on all $\mathbf{T}$, including $\Delta_{\mathcal{P},R}$. This is accomplished by choosing a direction $\mathbf{v}$ that is not parallel to any of the lines comprising $\Delta$, and then defining $\textsc{SR}^{\mathbf{v}}\,:\, \mathbf{T} \rightarrow \mathcal{A}^{\Z^2}$ by \[\textsc{SR}^{\mathbf{v}}(\boldsymbol{x}) = \lim_{\varepsilon \rightarrow 0^{+}} \textsc{SR}(\boldsymbol{x} + \varepsilon \mathbf{v}).\] The idea here is that if $R^{\mathbf{n}}(\boldsymbol{x}) \in \Delta$, there is ambiguity as to which atom's label is assigned to $\mathbf{n}$ by $\textsc{Cfg}$; the direction $\mathbf{v}$ settles that ambiguity by assigning to $\mathbf{n}$ the label of the atom on the $\mathbf{v}$ side of $\Delta$ where $R^{\mathbf{n}}(\boldsymbol{x})$ lies.

The following lemma from \cite{Labb2021} characterizes $\mathcal{X}_{\mathcal{P},R}$ in terms of map $\textsc{SR}$.

\begin{lemma}
For each dirction $\mathbf{v}$ not parallel to a line segment in $\Delta$, we have \[\overline{\textsc{SR}^{\mathbf{v}}(\mathbf{T})} = \overline{\textsc{SR}(\mathbf{T} \setminus \Delta_{\mathcal{P},R})} = \mathcal{X}_{\mathcal{P},R}\] where the overline indicates topological closure in the compact product topology on $\mathcal{A}^{\Z^2}$.
\end{lemma}

We summarize further pertinent results from  \cite{Labb2021} here: \begin{enumerate}
    \item $\textsc{SR}^{\mathbf{v}} \,:\, \mathbf{T} \rightarrow \mathcal{X}_{\mathcal{P},R}$ is 1-1, and so
    \item the inverse map can be extended to the factor map $f:\mathcal{X}_{\mathcal{P},R} \rightarrow \mathbf{T}$ such that for each configuration $\omega \in \mathcal{X}_{\mathcal{P},R}$, $f(\omega) = \cap_{n=0}^{\infty}\overline{D_\omega}$ (so $f$ is as defined in Proposition \ref{prop:factor-map}). 
    \item $f$ is 1-1 on $f^{-1}(\mathbf{T} \setminus \Delta_{\mathcal{P},R})$.
    \item The factor map commutes the shift actions $R$ on $\mathbf{T}$ and $\sigma$ on $\mathcal{X}_{\mathcal{P},R}$; that is, $R^{\mathbf{n}}f = f \sigma^{\mathbf{n}}$.
\end{enumerate}

Using the factor map $f$ and these properties of $f$, Labb\'{e} proved the following important fact \cite{Labb2021}:

\begin{lemma}
Suppose $\mathcal{P}$ gives a symbolic representation of $(\mathbf{T},\Z^2,R)$. Then \begin{enumerate}
    \item if $(\mathbf{T},\Z^2,R)$ is minimal, then $(\mathcal{X}_{\mathcal{P},R},\Z^2,\sigma)$ is minimal, and 
    \item if $R$ is a free $\Z^2$-action on $\mathbf{T}$, then $\mathcal{X}_{\mathcal{P},R}$ is aperiodic.
\end{enumerate}
\label{lem:aperiodic}\end{lemma}

The last results we quote from \cite{Labb2021} pertain to understanding how a topological partition $\mathcal{P}$ gives rise to Wang tile protoset $\mathcal{T}$ such that $\mathcal{X}_{\mathcal{P},R} \subset \Omega_{\mathcal{T}}$. To that end, suppose that $(\mathbf{T},\Z^2,R)$ is a dynamical system, let $\mathcal{P}_r = \{P_i\}_{i \in I}$ (the \emph{right side partition}) and $\mathcal{P}_b = \{P_{j}\}_{j \in J}$ (the \emph{bottom side partition}) be two finite topological partitions of $\mathbf{T}$, and then let $\mathcal{P}_l = \{R^{(1,0)}(P_a)\,:\,P_a \in \mathcal{P}_r\}$ (the \emph{left side partition}) and $\mathcal{P}_b = \{R^{(0,1)}(P_a)\,:\,P_a \in \mathcal{P}_t\}$ (the \emph{top side partition}); the labels of the atoms of $\mathcal{P}_l$ and $\mathcal{P}_r$ (i.e. the set $I$) are the colors of the right and left sides of tiles in a soon-to-be-defined Wang tile protoset, and the labels of the atoms of $\mathcal{P}_b$ and $\mathcal{P}_t$ (i.e. the set $J$) are the colors of the top and bottom sides in that same Wang tile protoset. For each $(i,j,k,\ell) \in I \times J \times I \times J$, let \[P_{(i,j,k,\ell)} = P_i \cap P_j \cap P_k \cap P_{\ell}.\] Next, define \[\mathcal{T} = \{\tau \in I \times J \times I \times J \,:\, P_{\tau} \neq \emptyset\}.\] We interpret each $\tau \in \mathcal{T}$ as a Wang tile as described in Section \ref{subsec:dyn_sys_intro}. $\mathcal{T}$ is naturally associated with the \emph{tile partition} $\mathcal{P} = \{P_{\tau} \,:\, \tau \in \mathcal{T}\}$ which is the refinement of $\mathcal{P}_l$, $\mathcal{P}_r$, $\mathcal{P}_b$ and $\mathcal{P}_t$, and each point $\boldsymbol{x} \in \mathbf{T} \setminus \Delta$ corresponds to a unique tile in $\mathcal{T}$. With $\mathcal{P}$ defined as a partition of the torus $\mathcal{T}$ in this way, we find the following lemmas in \cite{Labb2021}.

\begin{lemma}
With $\mathcal{P}$ defined as above as a refinement of $\mathcal{P}_l$, $\mathcal{P}_r$, $\mathcal{P}_b$, and $\mathcal{P}_t$, we have $\mathcal{X}_{\mathcal{P},R} \subseteq \Omega_{\mathcal{T}}$. \label{lem:Chi_in_Omega}
\end{lemma}

\begin{lemma}
For every direction $\mathbf{v}$ not parallel to a line segment in $\Delta$, $\textsc{SR}^{\mathbf{v}}\,:\, \mathbf{T} \rightarrow \Omega_{\mathcal{T}}$ is a one-to-one map.
\end{lemma}

\section{Application of Labb\'{e}'s Theory to Experimentally Determined Partitions}\label{apx:app}

In this section, we apply the theory as outlined in \ref{apx:dyn_sys} to the partition $\mathcal{P}_{16}$ corresponding to the Ammann A2-derived Wang protoset $\mathcal{T}_{16}$. Let us begin by pointing out that it is easily checked that the partition $\mathcal{P}_{16}$ is the refinement of the partitions $\mathcal{P}_{16,t}$, $\mathcal{P}_{16,\ell}$, $\mathcal{P}_{16,b}$, and $\mathcal{P}_{16,r}$, as depicted in Figure \ref{fig:partition_refinement}, and so by Lemma \ref{lem:Chi_in_Omega}, we know that $\mathcal{X}_{\mathcal{P}_{16},R_{16}} \subseteq \Omega_{16}$.

\pagebreak
\begin{figure}[H]
\centering
\begin{subfigure}[h]{.49\textwidth} 
\centering
\includegraphics[width=.9\textwidth]{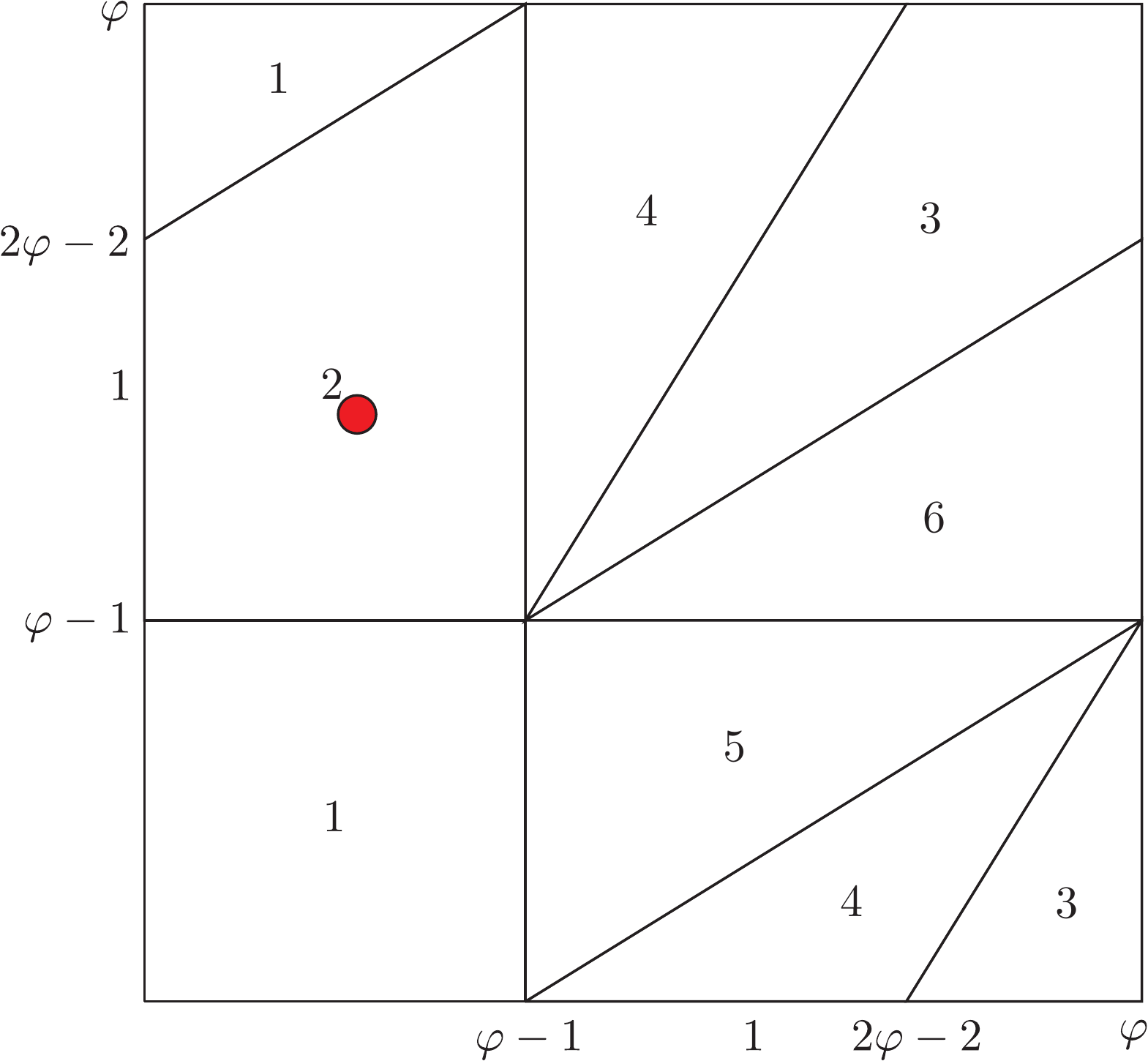}  
\caption{$\mathcal{P}_{16,t}$ - the top side partition}
\label{fig:P_l}
\end{subfigure}
\begin{subfigure}[h]{.49\textwidth} 
\centering
\includegraphics[width=.9\textwidth]{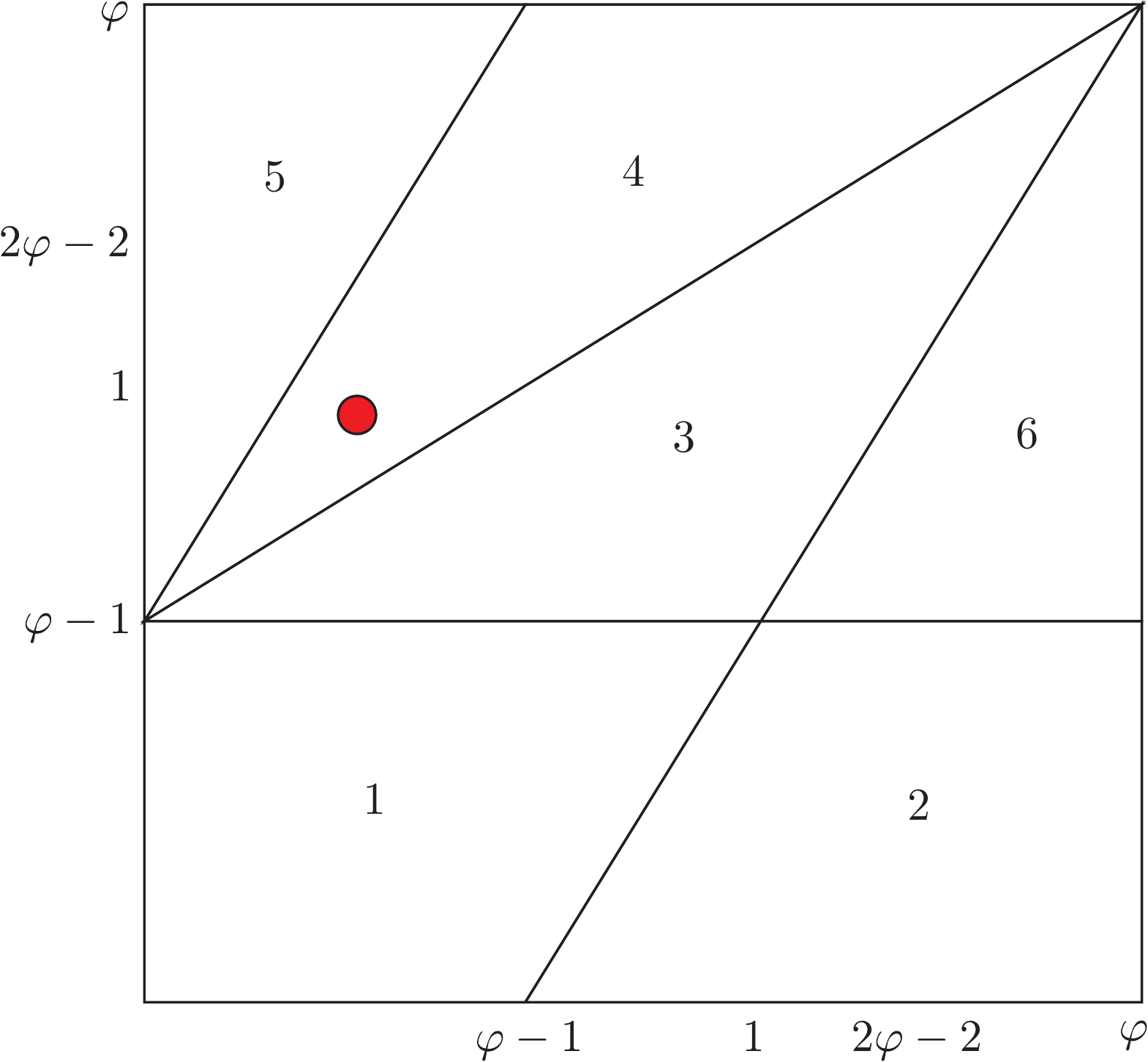}  
\caption{$\mathcal{P}_{16,\ell}$ - the left side partition}
\label{fig:P_r}
\end{subfigure}
\begin{subfigure}[h]{.49\textwidth} 
\centering
\includegraphics[width=.9\textwidth]{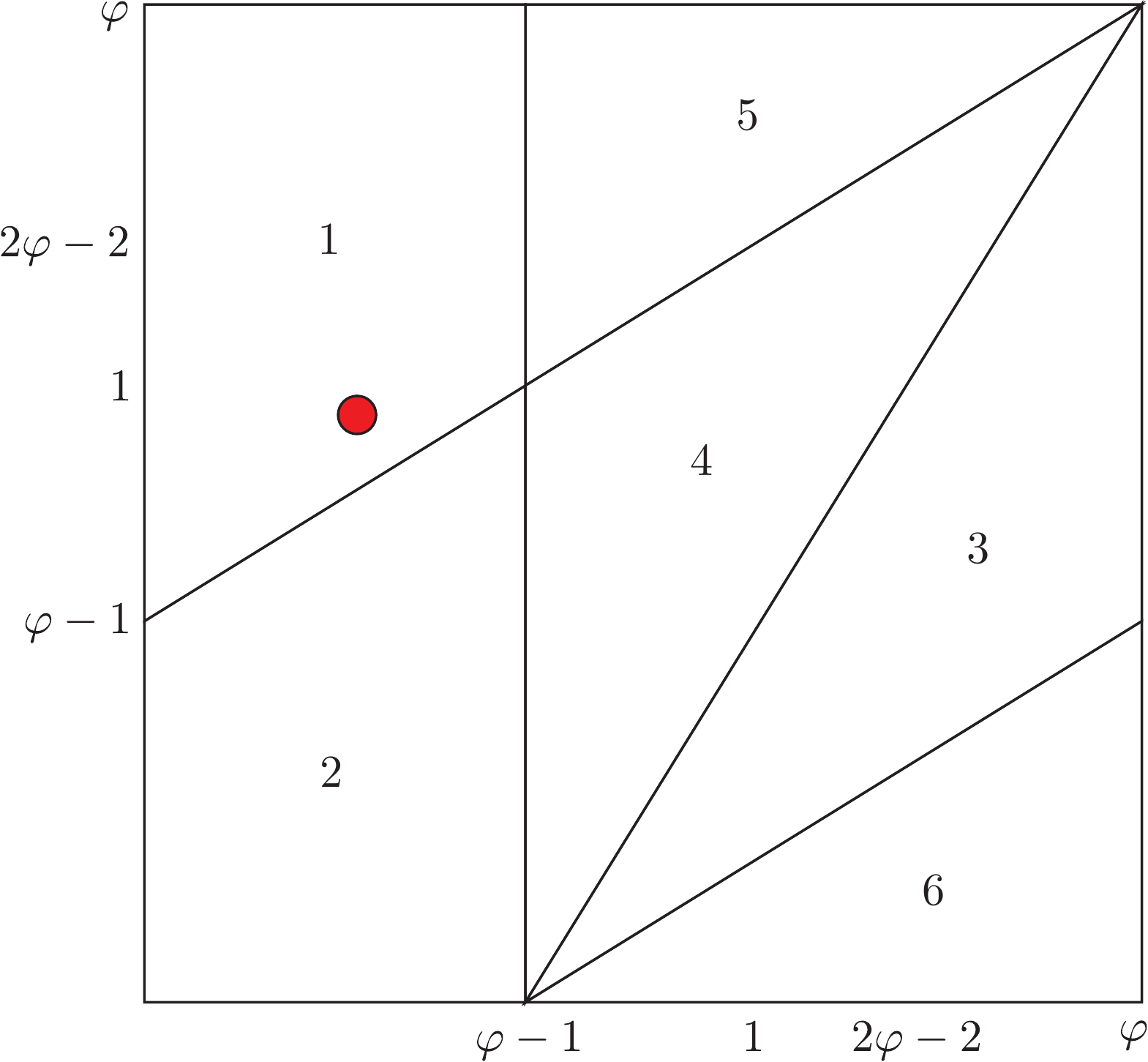}  
\caption{$\mathcal{P}_{16,b}$ - the bottom side partition}
\label{fig:P_b}
\end{subfigure}
\begin{subfigure}[h]{.49\textwidth} 
\centering
\includegraphics[width=.9\textwidth]{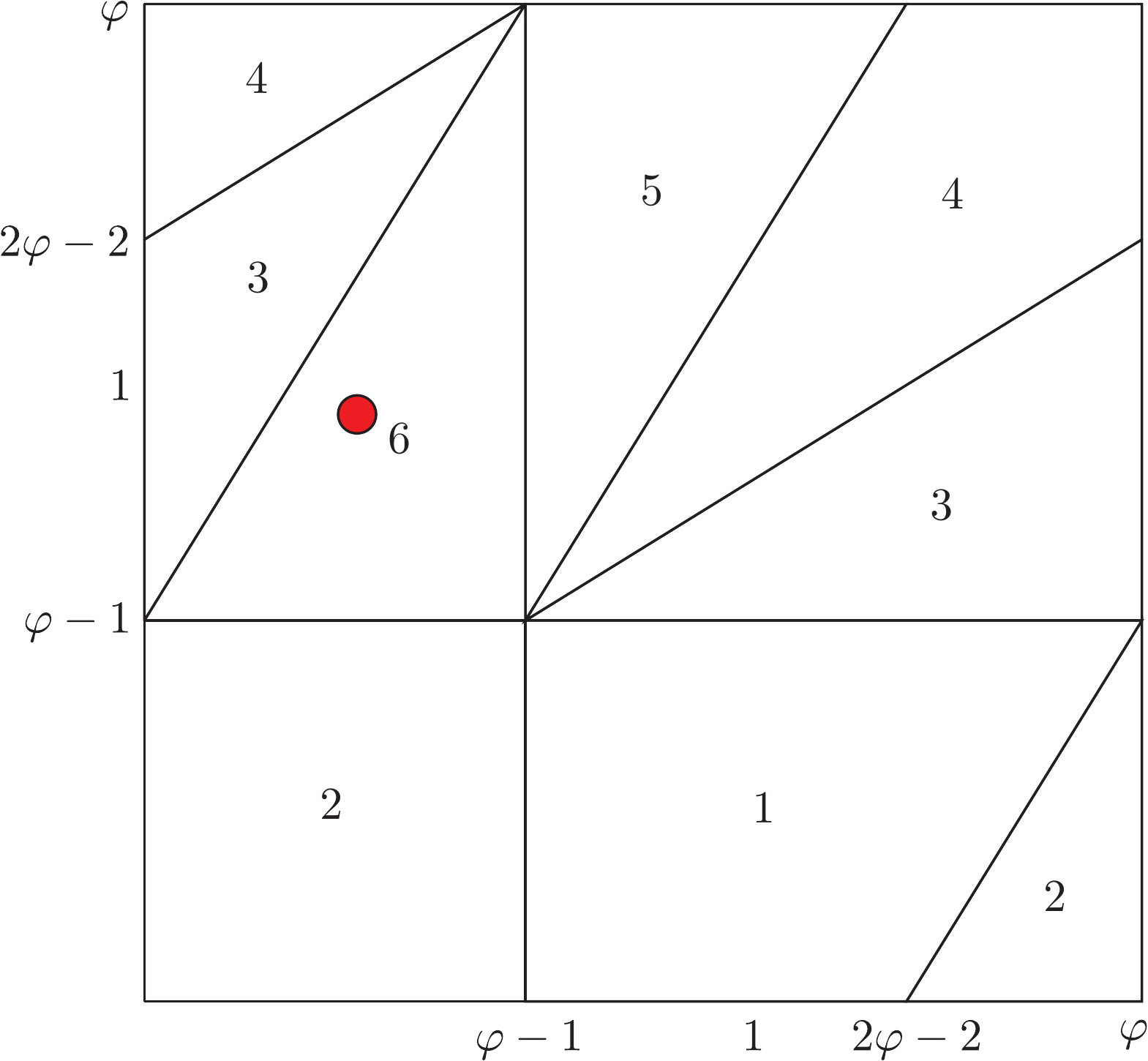}  
\caption{$\mathcal{P}_{16,r}$ - the right side partition}
\label{fig:P_t}
\end{subfigure}
\begin{subfigure}[h]{.49\textwidth} 
\centering
\includegraphics[width=.9\textwidth]{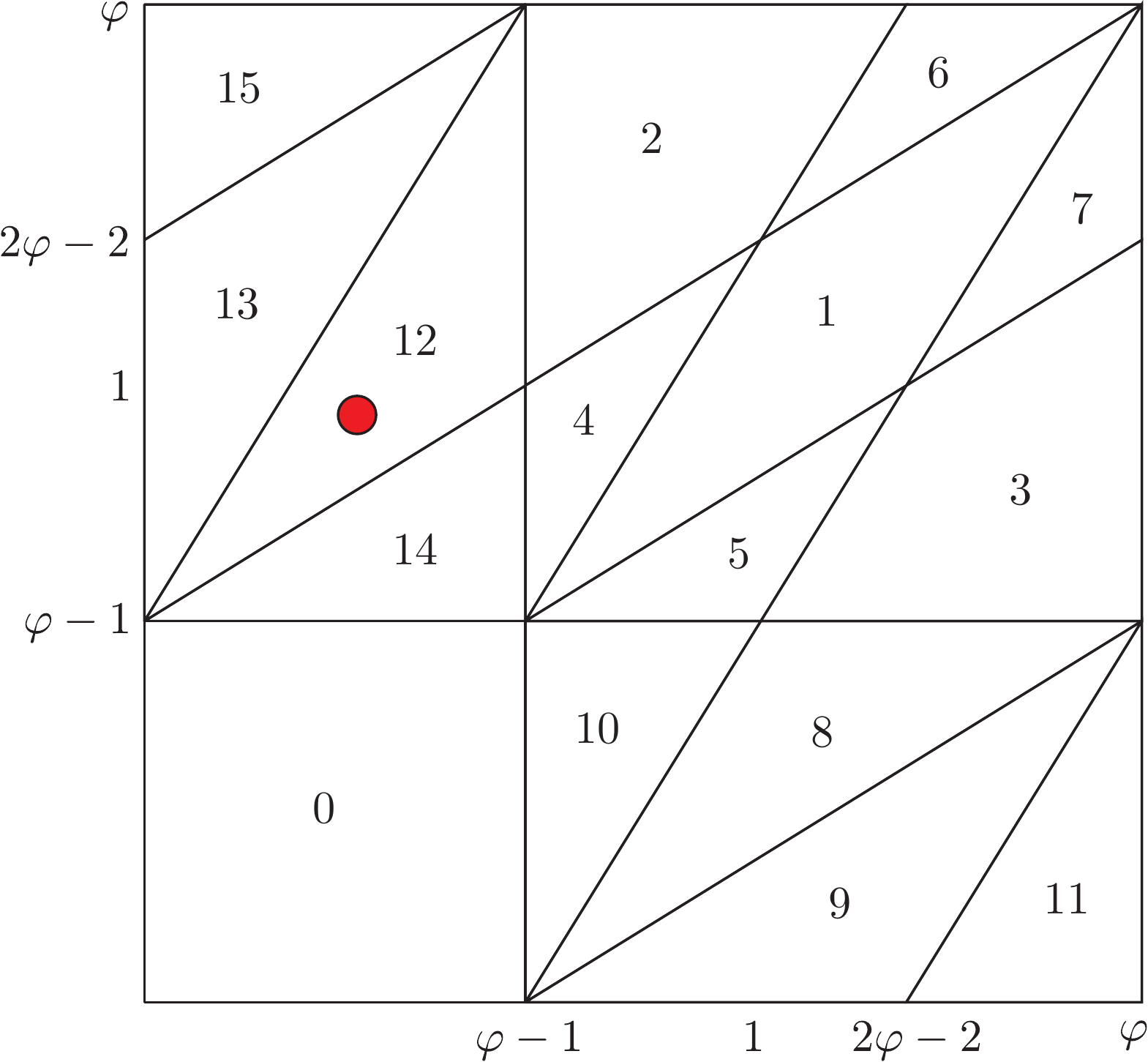}  
\caption{$\mathcal{P}_{16}$ - the tile partition}
\label{fig:P}
\end{subfigure}\caption{Refining the side partitions to obtain the tile partition $\mathcal{P}$. The red dot corresponds to the tile $(r,t,l,b) = (6,2,4,1)$, which is tile 12 from Figure \ref{fig:Ammann_16_Protoset}, so the corresponding region in $\mathcal{P}_{16}$ which is the intersection of the regions containing the red dot, is labeled 12.}\label{fig:partition_refinement} 
\end{figure}
\pagebreak

Next, we want to demonstrate that $\mathcal{P}$ gives a symbolic representation of $(\mathbf{T}_{16},\Z^2,\R_{16})$, but we will first establish a few supporting results.

\begin{lemma}\label{lem:minimality1}
    Let $\boldsymbol{v} = (v_x,v_y) \in \mathbf{T}_{16} = [0,\varphi]\times [0,\varphi]$ and let $x_0,x_1 \in [0,\varphi)$ with $x_0<x_1$. Then, there exists some $\boldsymbol{v}' = (v_x',v_y) \in \mathcal{O}_{R_{16}}(\boldsymbol{v})$ with $x_0<v'_x < x_1$. \end{lemma}

\begin{proof}
    Consider the set $\{R_{16}^{(n,0)}(\mathbf{v}) \,:\, n \in \Z\} = \{(v_x +n \pmod{\varphi},v_y) \,:\,n \in \Z\}\subset \mathcal{O}_{R_{16}}(\mathbf{v})$, which we see is an irrational rotation of the point $\mathbf{v}$ horizontally around $\mathbf{T}_{16}$. Consequently, $\{R_{16}^{(0,n)}(\mathbf{v}) \,:\, n \in \Z\}$ is dense on the horizontal line through $\mathbf{v}$, and it follows that  there exists some $\mathbf{v}' = (v_x',v_y) \in \mathcal{O}_{R_{16}}(\mathbf{v})$ with $x_0<v'_x < x_1$.
\end{proof}

In an analogous way, we can prove the following Lemma:

\begin{lemma}\label{lem:minimality2}
    Let $\mathbf{v} = (v_x,v_y) \in \mathbf{T}_{16}$ and let $y_0,y_1 \in [0,\varphi)$ with $y_0<y_1$.  Then there exists some $\mathbf{v}' = (v_x,v_y') \in \mathcal{O}_{R_{16}}(\mathbf{v})$ with $y_0 < v_y' < y_1$. \end{lemma}

Lemmas \ref{lem:minimality1} and \ref{lem:minimality2} show that the orbit of any point in $\mathbf{T}_{16}$ has a non-empty intersection with any horizontal or vertical strip on the torus, which allows us to prove the following theorem.

\begin{proposition}
    $(\mathbf{T}_{16}, \Z^2, R_{16})$ is minimal. \label{prop:min}
\end{proposition}

\begin{proof}
We show that for arbitrary $\boldsymbol{p} = (p_x,p_y) \in \mathbf{T}_{16}$, $\mathcal{O}_{R_{16}}(\boldsymbol{p})$ is dense in $\mathbf{T}_{16}$. Let $\boldsymbol{a} \in \mathbf{T}_{16}$ and let $U$ be an open square of side length $2\varepsilon$ centered at $\boldsymbol{a} = (a_x,a_y)$ in $\mathbf{T}$, and without loss of generality, choose $\varepsilon$ sufficiently small so that $0 < a_x - \varepsilon < a_x + \varepsilon < \varphi$ and $0 < a_y - \varepsilon < a_y + \varepsilon < \varphi$. By Lemma \ref{lem:minimality2}, there exists $n \in \Z$ such that $\boldsymbol{p}' = (p_x,p_y')= R^{(0,n)}(\boldsymbol{p})$ lies in the horizontal strip bounded by lines $y=a_y - \varepsilon$ and $y = a_y + \varepsilon$. Next, by Lemma \ref{lem:minimality1}, there exists some $m \in \Z$ such that $\boldsymbol{p''}=R^{(m,n)}(\boldsymbol{p'}) = (p_x',p_y')$ with $a_x - \varepsilon < p_x' < a_x + \varepsilon$, and so we see that $\boldsymbol{p''} \in U$, and so $\mathcal{O}_R(\boldsymbol{p}) \cap U \neq \emptyset$. \end{proof}

To prove that $\mathcal{P}$ gives a symbolic representation of $(\boldsymbol{T},\Z^2,R)$, We apply Lemma \ref{lem:LabbeMinimality}, showing that some atom of $\mathcal{P}$ is invariant only under the trivial translation. 

\begin{proposition} \label{prop:symbrep}
    $\mathcal{P}_{16}$ gives a symbolic representation of $(\mathbf{T}_{16},\Z^2,R_{16})$.
\end{proposition}

\begin{proof}
Consider $P_{0} \in \mathcal{P}_{16}$ from Figure \ref{fig:Amm_16_partition}. We show that $P_{0}$ is invariant only under the trivial translation under $R_{16}$. If $R^{\mathbf{n}}_{16}(P_0) = P_0$ for some $\boldsymbol{n} = (n_1,n_2) \in \Z^2$, then by continuity $R^{\mathbf{n}}$ must fix $(0,0) \in \boldsymbol{T}_{16}$, from which we obtain $n_1 \pmod{\varphi} = 0$ and $n_2 = 0 \pmod{\varphi}$. Due to the irrationality of $\varphi$, we conclude that $n_1 = n_2 = 0$, so $R^{\boldsymbol{n}}$ is the trivial translation.
\end{proof}

Having established that $\mathcal{P}_{16}$ gives a symbolic representation of $(\mathbf{T}_{16},\Z^2,R_{16]})$, we then apply Lemma \ref{lem:aperiodic} to conclude that $(\mathcal{X}_{\mathcal{P}_{16},R_{16}},\Z^2,\sigma)$ is minimal and $\mathcal{X}_{\mathcal{P}_{16},R_{16}}$ is an aperiodic subshift of $\Omega_{16}$. Moreover, we point out that it was recently proven by Labb\'{e} \cite{labbé2024metallicI}[Lemma 11.7] that $\Omega_{16}$ is minimal. It follows that $\mathcal{X}_{\mathcal{P}_{16},R_{16}} = \Omega_{16}$, and so $\mathcal{X}_{\mathcal{P}_{16},R_{16}}$ is an SFT,  proving the following.

\begin{theorem}
    $\mathcal{P}_{16}$ is a Markov partition for $(\mathbf{T}_{16},\Z^2,\R_{16}$), $\mathcal{X}_{\mathcal{P}_{16},R_{16}}$ is aperiodic, and $\mathcal{X}_{\mathcal{P}_{16},R_{16}} = \Omega_{16}$.
\end{theorem}

\end{appendices}

\printbibliography

@article{Labb2021,
	doi = {10.5802/ahl.73},
	url = {https://doi.org/10.5802%2Fahl.73},
	year = 2021,
	month = {jan},
  	publisher = {Cellule {MathDoc}/{CEDRAM}},
  	volume = {4},
	pages = {283--324},
	author = {S{\'{e}}bastien Labb{\'{e}}},
	title = {Markov partitions for toral $\Z^2$-rotations featuring Jeandel{\textendash}Rao Wang shift and model sets},
	journal = {Annales Henri Lebesgue}
}

@phdthesis{Jang2021,
author = {Hyeeun Jang},
title = {Directional Expansiveness},
publisher = {The Columbian College of Arts and Sciences of The George Washington University},
address = {Washington D.C.},
year = {2021}
}

@book{GS1987,
  title     = {Tilings and Patterns},
  author    = {Branko Gr{\"{u}}nbaum and G. C. Shephard},
  year      = 1987,
  publisher = {W. H. Freeman and Company},
  address   = {New York}
}

@misc{Mann2022,
  doi = {10.48550/ARXIV.2206.02414},
  
  url = {https://arxiv.org/abs/2206.02414},
  
  author = {Labbé, Sébastien and Mann, Casey and McLoud-Mann, Jennifer},
  
  title = {Nonexpansive directions in the Jeandel-Rao Wang shift},
  
  publisher = {arXiv},
  
  year = {2022},
  
  copyright = {Creative Commons Attribution Share Alike 4.0 International}
}

@article {MR216954,
    AUTHOR = {Berger, Robert},
     TITLE = {The undecidability of the domino problem},
   JOURNAL = {Mem. Amer. Math. Soc.},
  FJOURNAL = {Memoirs of the American Mathematical Society},
    VOLUME = {66},
      YEAR = {1966},
     PAGES = {72},
      ISSN = {0065-9266},
   MRCLASS = {02.74},
  MRNUMBER = {216954},
MRREVIEWER = {J. R. B\"{u}chi},
}

@article{JR1,
   title={An aperiodic set of 11 Wang tiles},
   ISSN={2517-5599},
   url={http://dx.doi.org/10.19086/aic.18614},
   DOI={10.19086/aic.18614},
   journal={Advances in Combinatorics},
   publisher={Alliance of Diamond Open Access Journals},
   author={Jeandel, Emmanuel and Rao, Michaël},
   year={2021},
   month={Jan}
}

@article{Wang,
author = {Wang, Hao},
title = {Proving Theorems by Pattern Recognition — II},
journal = {Bell System Technical Journal},
volume = {40},
number = {1},
pages = {1-41},
doi = {https://doi.org/10.1002/j.1538-7305.1961.tb03975.x},
url = {https://onlinelibrary.wiley.com/doi/abs/10.1002/j.1538-7305.1961.tb03975.x},
eprint = {https://onlinelibrary.wiley.com/doi/pdf/10.1002/j.1538-7305.1961.tb03975.x},
year = {1961}
}

@book {MR2939561,
    AUTHOR = {Berger, Robert},
     TITLE = {T{HE} {UNDECIDABILITY} {OF} {THE} {DOMINO} {PROBLEM}},
      NOTE = {Thesis (Ph.D.)--Harvard University},
 PUBLISHER = {ProQuest LLC, Ann Arbor, MI},
      YEAR = {1965},
     PAGES = {(no paging)},
   MRCLASS = {Thesis},
  MRNUMBER = {2939561},
       URL =
              {http://gateway.proquest.com/openurl?url_ver=Z39.88-2004&rft_val_fmt=info:ofi/fmt:kev:mtx:dissertation&res_dat=xri:pqdiss&rft_dat=xri:pqdiss:0248356},
}

@misc{hat,
      title={An aperiodic monotile}, 
      author={David Smith and Joseph Samuel Myers and Craig S. Kaplan and Chaim Goodman-Strauss},
      year={2023},
      eprint={2303.10798},
      archivePrefix={arXiv},
      primaryClass={math.CO}
}

@misc{spectre,
      title={A chiral aperiodic monotile}, 
      author={David Smith and Joseph Samuel Myers and Craig S. Kaplan and Chaim Goodman-Strauss},
      year={2023},
      eprint={2305.17743},
      archivePrefix={arXiv},
      primaryClass={math.CO}
}

@book {Gardner1,
    AUTHOR = {Gardner, Martin},
     TITLE = {Penrose tiles to trapdoor ciphers},
    SERIES = {MAA Spectrum},
      NOTE = {$\ldots$and the return of Dr. Matrix,
              Revised reprint of the 1989 original},
 PUBLISHER = {Mathematical Association of America, Washington, DC},
      YEAR = {1997},
     PAGES = {x+319},
      ISBN = {0-88385-521-6},
   MRCLASS = {00A08 (52C20)},
  MRNUMBER = {1443205},
}

@misc{labbé2024metallicI,
      title={Metallic mean Wang tiles I: self-similarity, aperiodicity and minimality}, 
      author={Sébastien Labbé},
      year={2024},
      eprint={2312.03652},
      archivePrefix={arXiv},
      primaryClass={math.DS}
}

@software{slabbe_sage,
  author = {Sébastien Labbé},
  title = {slabbe Package for SageMath},
  url = {https://pypi.org/project/slabbe/},
  version = {0.7.6},
    YEAR = {2023},
  date = {2023-23-06},
}

\end{document}